\documentclass[a4paper, 12pt]{article}
\usepackage{geometry} 
\usepackage[ngerman,UKenglish]{babel}
\usepackage[latin1]{inputenc}
\usepackage[T1]{fontenc}
\usepackage{listings}
\usepackage{graphicx}
\usepackage[font = small]{caption}
\usepackage{subfig}
\usepackage{amsopn}
\usepackage{amssymb}
\usepackage{amsmath}
\usepackage{amsthm}
\usepackage{dsfont}
\usepackage{a4wide}
\usepackage[autostyle]{csquotes}
\usepackage{bm,bigints}
\usepackage{array,multirow}
\newcolumntype{B}[1]{>{\centering\arraybackslash}m{#1}}
\usepackage{colortbl}
\usepackage{authblk}

\usepackage{hyperref}
\usepackage{xcolor}
\hypersetup{colorlinks, urlcolor = teal, linkcolor = teal, citecolor = teal}

\DeclareMathOperator{\BV}{BV}

\DeclareMathOperator{\mycurl}{\text{curl}}
\DeclareMathOperator{\mydiv}{\text{div}}
\DeclareMathOperator{\mysheara}{\text{sh$_1$}}
\DeclareMathOperator{\myshearb}{\text{sh$_2$}}
\DeclareMathOperator{\Sym}{\operatorname{Sym}}

\newtheorem{thm}{Theorem}
\newtheorem{lemma}{Lemma}
\newtheorem{rem}{Remark}

\newcommand{\R}{\mathbb{R}}

\begin{document}

\title{\textbf{Unified Models for Second-Order TV-Type Regularisation in Imaging\\ \large{A New Perspective Based on Vector Operators}}}
\author[1]{Eva-Maria Brinkmann}
\author[2]{Martin Burger}
\author[3]{Joana Sarah Grah}
\affil[1]{Applied Mathematics: Institute for Analysis and Numerics, Westf\"alische Wilhelms-Universit\"at M\"unster, Germany\\

\url{e.brinkmann@wwu.de}}
\affil[2]{Department Mathematik, Friedrich-Alexander Universit\"at Erlangen-N\"urnberg, Germany\\

\url{martin.burger@fau.de}}
\affil[3]{Institute for Computer Graphics and Vision, Graz University of Technology, Austria\\

\url{joana.grah@icg.tugraz.at}}

\maketitle
\begin{abstract}
We introduce a novel regulariser based on the natural vector field operations gradient, divergence, curl and shear. 
For suitable choices of the weighting parameters contained in our model it generalises well-known first- and second-order TV-type regularisation methods including TV, ICTV and TGV$^2$ and enables interpolation between them. 
To better understand the influence of each parameter, we characterise the nullspaces of the respective regularisation functionals. 
Analysing the continuous model, we conclude that it is not sufficient to combine penalisation of the divergence and the curl to achieve high-quality results, but interestingly it seems crucial that the penalty functional includes at least one component of the shear or suitable boundary conditions. 
We investigate which requirements regarding the choice of weighting parameters yield a rotational invariant approach. 
To guarantee physically  meaningful reconstructions, implying that conservation laws for vectorial differential operators remain valid, we need a careful discretisation that we therefore discuss in detail.
\end{abstract}


\section{Introduction}
\label{sec:introduction}

\begin{sloppypar}
In the beginning of the 1990s, Rudin, Osher and Fatemi revolutionised image processing and in particular variational methods using sparsity-enforcing terms by introducing total variation (TV) regularisation \cite{ROF}. 
Since then, it has been serving as a state-of-the-art concept for various imaging tasks including denoising, inpainting, medical image reconstruction, segmentation and motion estimation. 
Minimisation of the TV functional, which for $u \in L^1(\Omega)$ is given by 
\begin{equation}
\text{TV}_{\alpha}(u) := \sup\limits_{\substack{\varphi \in C_c^\infty(\Omega,\mathbb{R}^2)\\ ||\varphi||_{\infty} \leq \alpha}} \int_{\Omega} u\ \mydiv(\varphi)\,dx, \tag{TV*}
\label{eq:TV*}
\end{equation} 
provides cartoon-like images with piecewise constant areas that are separated by sharp edges. 
Note that here and in the following $\Omega \subseteq \mathbb{R}^2$ is an open, bounded image domain with Lipschitz boundary and $\alpha > 0$.
With regard to the TV model, it is a well-known fact that there are two major drawbacks inherent in this method: on the one hand solutions typically suffer from a loss of contrast.
On the other hand they often exhibit the so-called 'staircasing-effect', where areas of gradual intensity transitions are approximated by piecewise constant regions separated by sharp edges such that the intensity function along a line profile in 1D is reminiscent of a staircase.
To address the former deficiency, Osher and coworkers proposed the use of Bregman iterations \cite{BregmanIterations}, a semi-convergent iterative procedure that allows for a regain of contrast and details in the recovered images. 
More recently, various debiasing techniques \cite{DebiasingDelledale1,Debiasing,DebiasingDelledale2} have been introduced to compensate for the systematic error of the lost contrast. 
In this paper, we shall however focus on the latter issue. 
To this end, we propose a novel regularisation functional composed of natural vector field operators that is capable of providing solutions with sharp edges and smooth transitions between intensity values simultaneously. 
This approach certainly stands in the tradition of several modified TV-type regularisation functionals that have been contrived to cure the staircasing effect by incorporating penalisation of second-order total variation, which is given by (cf.\ for example \cite{TV2,HigherOrderTV})
\begin{align}
\text{TV}_{\alpha}^2(u) = \sup\limits_{\substack{\varphi \in C_c^\infty(\Omega,\Sym^2(\mathbb{R}^2))\\ ||\varphi||_{\infty} \leq \alpha}} \int_{\Omega} u\ \mydiv^2(\varphi) ~dx.
\tag{TV2*}
\label{eq:TV2*}
\end{align}
Here, $\Sym^2(\mathbb{R}^2)$ denotes the set of second-order symmetric tensor fields on $\mathbb{R}^2$, i.e.\ the set of symmetric $2 \times 2$-matrices. Moreover, for a symmetric $2 \times 2$-matrix $\varphi$, $\mydiv(\varphi) \in C_0^1(\Omega,\mathbb{R}^2)$ and $\mydiv^2(\varphi) \in C_0(\Omega)$ are defined by
\begin{align}
(\mydiv(\varphi))_i &= \sum_{j=1}^2 \frac{\partial \varphi_{ij}}{\partial x_j}, \notag \\
\mydiv^2(\varphi) &= \sum_{i=1}^2 \frac{\partial^2 \varphi_{ii}}{\partial x_i^2} + 2 \sum_{i < j} \frac{\partial^2 \varphi_{ij}}{\partial x_i \partial x_j} = \frac{\partial^2 \varphi_{11}}{\partial x_1^2} + \frac{\partial^2 \varphi_{22}}{\partial x_2^2} + 2 \frac{\partial^2 \varphi_{12}}{\partial x_1 \partial x_2}.
\tag{div2}
\label{eq:first_and_second_order_divergence}
\end{align}
\end{sloppypar}

\begin{sloppypar}
Let us briefly recall the most popular instances of second-order TV-type regularisers in a formal way. Note first that for $u \in W^{1,1}(\Omega)$, the (first-order) total variation functional can be rephrased as
\begin{equation}
\text{TV}(u) = \int_{\Omega} |\nabla u|\,dx, \tag{TV}
\label{eq:TV}
\end{equation}
where here and in the following we always denote by $\nabla u$ the gradient of $u$ in the sense of distributions and by $\vert \cdot \vert$ the Euclidean norm. 
Against the backdrop of \eqref{eq:TV}, Chambolle and Lions \cite{ICTV} proposed to compose regularisers for image processing tasks by coupling several convex functionals of the gradient by means of the infimal convolution, defined for two functionals as
\begin{equation}
J_1(u) \square J_2(u) = \inf\limits_{u_2} J_1(u-u_2) + J_2(u_2). \tag{IC}
\label{eq:IC}
\end{equation}
In particular, they suggested to use a combination of first and second derivatives
\begin{align}
\text{ICTV}_{(\alpha_1,\alpha_0)}(u) = \inf\limits_{u_2 \in W^{2,1}(\Omega)} \alpha_1 \int_{\Omega}  \vert \nabla u - \nabla u_2 \vert ~dx + \alpha_0 \int_{\Omega} \vert \nabla \left( \nabla u_2 \right) \vert ~dx,
\tag{ICTV}
\label{eq:ICTV}
\end{align}
where here and in the following $\alpha_1, \alpha_0 > 0$ and we denote by $\vert \cdot \vert$ the Frobenius norm whenever the input argument is a matrix.
Following this train of thought, Chan, Esedoglu and Park \cite{CEP} proposed another variant of such a composed regularisation functional, namely
\begin{align}
\label{eq:CEP}
\text{CEP}_{(\alpha_1,\alpha_0)}(u) = \inf\limits_{u_2 \in W^{2,1}(\Omega)} \alpha_1 \int_{\Omega}  \vert \nabla u - \nabla u_2 \vert ~dx + \alpha_0 \int_{\Omega} \vert \mydiv\,(\nabla u_2) \vert ~dx.
\tag{CEP}
\end{align}
More recently, Bredies, Kunisch and Pock \cite{TGV} suggested to generalise the TV functional to the higher-order case in a different way. In comparison to the second-order TV functional \eqref{eq:TV2*}, they further constrained the set over which the supremum is taken by imposing an additional requirement on the supremum norm of the divergence of the symmetric tensor field. Thus, they introduced the total generalised variation (TGV) functional, which in the second-order case is given by
\begin{align}
\begin{split}
&\text{TGV}_{(\alpha_1,\alpha_0)}^2(u)
 = \sup_{\varphi \in \mathcal{B}_0} \int_{\Omega} u\ \mydiv^2(\varphi)~dx,\\
&\mathcal{B}_0 = \{ \varphi \in C_c^{\infty}(\Omega,\Sym^2(\mathbb{R}^2)): \Vert \varphi \Vert_\infty \leq \alpha_0,
\Vert \mydiv (\varphi) \Vert_\infty \leq \alpha_1 \}.
\end{split}
\tag{TGV*}
\label{eq:TGV*}
\end{align}
Considering the corresponding primal definition of this functional, we obtain the following unconstrained formulation: 
\begin{equation}
\text{TGV}_{(\alpha_1,\alpha_0)}^2(u) = \inf\limits_{w \in W^{1,1}(\Omega,\mathbb{R}^2)} \alpha_1 \int_{\Omega}  |\nabla u - w| ~dx + \alpha_0 \int_{\Omega} |\mathcal{E} (w)| ~dx. 
\tag{TGV}
\label{eq:TGV}
\end{equation}
In this case one naturally obtains a minimiser for $w$ in the space $BD(\Omega)$ of vector fields of bounded deformation, i.e.\ $w \in L^1(\Omega,\mathbb{R}^2)$ such that the distributional symmetrised derivative $\mathcal{E}(w)$ given by 
\begin{equation*}
\mathcal{E}(w) = \frac{1}{2}\left( \nabla w + \nabla w^T \right)
\tag{symG}
\label{eq:symmetrised_gradient}
\end{equation*}
is a $\Sym^2(\mathbb{R}^2)$-valued Radon measure. 
Note that we will very briefly recall the definition of Radon measures and some related notions in the subsequent section.
Looking closely at the \eqref{eq:TGV} functional, similarities and differences to the other second-order TV-type regularisation functionals introduced so far are revealed:  
all these approaches have in common that they employ the infimal convolution to balance between enforcing sparsity of the gradient of the function $u$ and sparsity of some differential operator of a vector field resembling the gradient of $u$. Thus, they locally emphasise penalisation of either the first- or the second-order derivative information, which will become visually apparent in Section \ref{sec:results}, Figures \ref{fig:resultsComparisonTrui05} and \ref{fig:resultsComparisonTest05}.
As a consequence, in comparison to the original TV regularisation, all the previously recalled second-order models introduce an additional optimisation problem. 
On the other hand, we can already observe a difference between the former two models and the latter approach: while in the ICTV and the CEP functional the gradient respectively the divergence operator is applied to the gradient of $u_2$, the symmetrised derivative in the TGV functional is applied to a vector field $w$ that does not necessarily have to be a gradient field. 
We will come back to this point later on. 
In the course of this paper, we will moreover show that our novel functional, which will be introduced below, can be seen as a generalisation of all aforementioned first- and second-order TV-type models, since for suitable parameter choices we (in the limit) obtain each of these approaches as a special case. 
This way, we do not only shed a new light on the relation of these well-established regularisation functionals and provide a means of interpolating between them, but we will also discuss properties of further second-order TV-type approaches that can be obtained by different weightings between the natural vector field operators our model builds upon.
\end{sloppypar}

\begin{sloppypar}
Let us now introduce our novel approach in more detail. In \cite{SVF}, we proposed a variational model for image compression that was motivated by earlier PDE-based methods \cite{Mainberger_et_al,MainbergerWeickert}: essentially, images are first encoded by performing edge detection and by saving the intensity values at pixels on both sides of the edges and this data is then decoded by performing homogeneous diffusion inpainting. In this context, our key observation was that the encoding step amounts to the search for a suitable image representation by means of a vector field whose non-zero entries are concentrated at the edges of the image to be compressed. Therefore, we conceived a minimisation problem that directly promotes such a sparse vector field $v$ and at the same time guarantees a certain fidelity of the decoded image $u$ to the original image $f$: 
\begin{align*}
&\frac{1}{2} \int_{\Omega} \left( u - f \right)^2 ~dx + \alpha \int_{\Omega} | v | ~dx \rightarrow \min\limits_{u,v} \qquad \text{subject to} \quad \mydiv\left( \nabla u - v \right) = 0,
\end{align*}
or equivalently, defining $w = \nabla u - v$,
\begin{align}
\frac{1}{2} \int_{\Omega} \left( u - f \right)^2 \,dx  + \alpha \int_{\Omega} | \nabla u - w | \,dx + \chi_0 \left( \mydiv(w) \right) \rightarrow  \min\limits_{u,w},
\tag{SVF1} 
\label{eq:SVF1unc}
\end{align}
where $\chi_0$ denotes the characteristic function of the set of divergence-free vector fields $w$.
Figure \ref{fig:SVF_compression} illustrates the sparse vector fields (SVF) method for image compression in an intuitive way. The input image $f$ (Figure \ref{fig:SVF_compression}, left image) is encoded via the two components of the vector field $v$ (second and third image) with the corresponding decoded image $u$ (right image) satisfying $\mydiv(\nabla u - v) = 0$.
\begin{figure}[h]
\captionsetup[subfigure]{labelformat=empty}
\centering
\subfloat[Original image $f$]{\includegraphics[height=2.75cm]{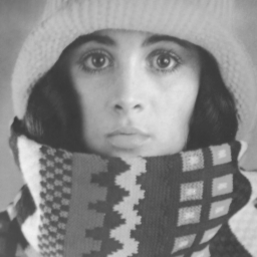}}\hfill
\subfloat[$v$ in $x_1$-direction]{\includegraphics[height=2.75cm]{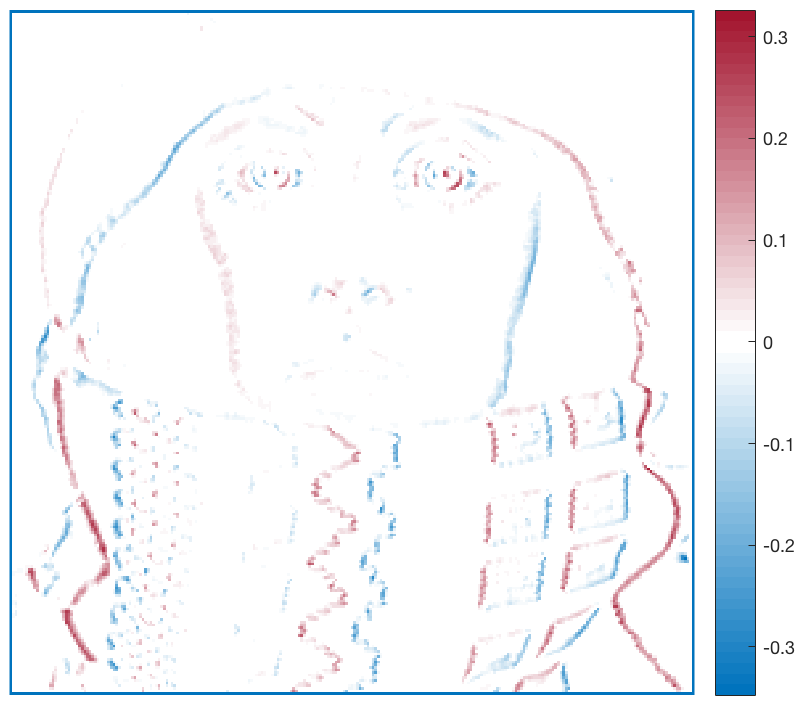}}\hfill
\subfloat[$v$ in $x_2$-direction]{\includegraphics[height=2.75cm]{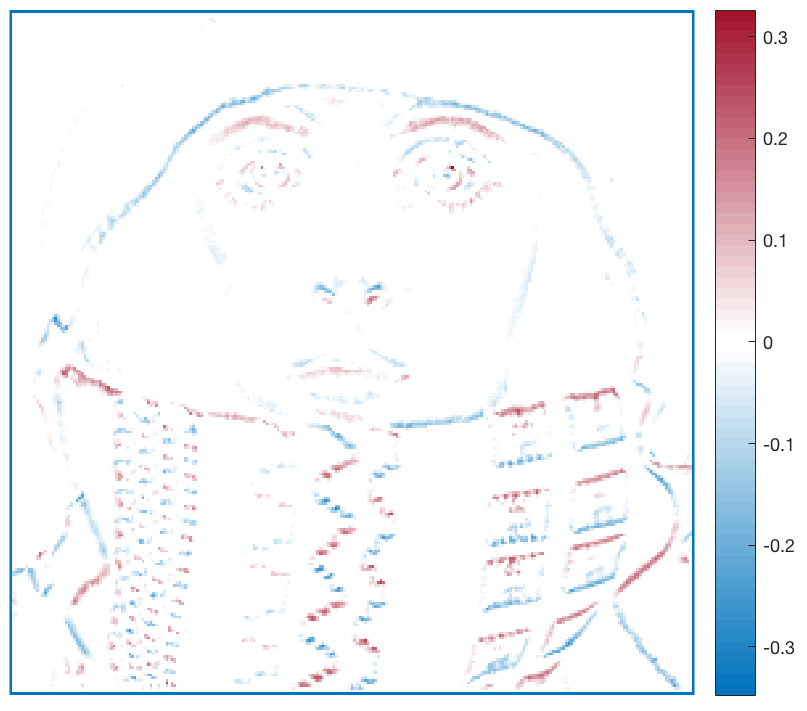}}\hfill
\subfloat[$|v|$]{\includegraphics[height=2.75cm]{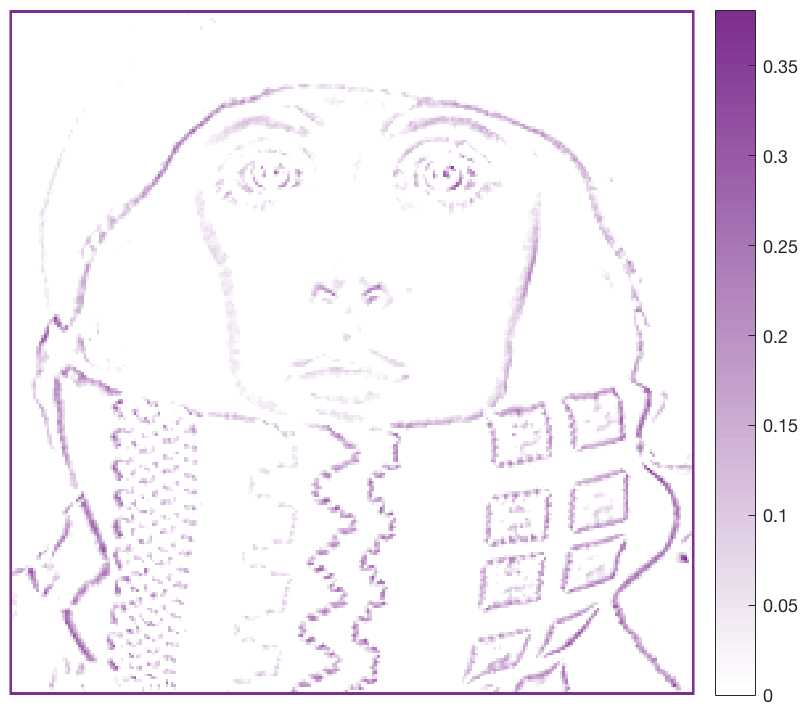}}
\hfill
\subfloat[Decoded image $u$]{\includegraphics[height=2.75cm]{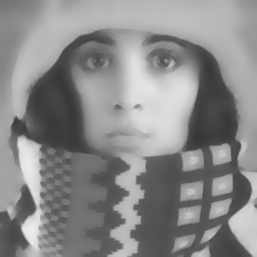}}
\caption{Illustration of the SVF image compression approach \eqref{eq:SVF1unc} for $\alpha = \frac{1}{15}$}
\label{fig:SVF_compression}
\end{figure}
\noindent Looking at these results, we concluded that the support of $v$ (fourth image) indeed corresponds well to an edge indicator, confirming the relation to \cite{Mainberger_et_al,MainbergerWeickert}. Moreover, we observed that on the one hand, our method preserves the main edges well while on the other hand, the decoded images (cf.\ Figure \ref{fig:SVF_compression}, right) exhibit a higher spatial smoothness in comparison to the original input images (cf.\ Figure \ref{fig:SVF_compression}, left). Since this increased smoothness did not come along with characteristic artefacts like the staircasing-effect in case of the TV regularisation, this seemed attractive for further reconstruction tasks. Therefore, we already back then considered the SVF model for homogeneous diffusion inpainting-based image denoising. In order to obtain higher flexibility, we reformulated the minimisation problem to
\begin{align}
\frac{1}{2} \int_{\Omega} (u - f)^2 \,dx + \alpha  \int_{\Omega} | \nabla u - w | ~dx + \alpha\sqrt{\beta} \, \int_{\Omega} | \mydiv(w)| ~dx \rightarrow  \min\limits_{u,w}
\tag{SVF}
\label{eq:SVF}
\end{align}
with $\beta>0$. In this form the \eqref{eq:SVF} model reveals strong similarities to the \eqref{eq:CEP} model with the only difference that $w$ does not necessarily have to be a gradient field. 
However, we had to realise that the denoising performance of this model was not convincing, since point artefacts were created at reasonable choices of the regularisation parameter (cf.\ \cite[Fig.\ 5]{SVF}). In particular, these point artefacts are also apparent in the second image of Figure \ref{fig:sparseDiffOp} in Section \ref{subsec:diffop2dvf}. As we will elaborate on in greater detail in Section \ref{sec:diffop}, these artefacts are indeed inherent in this method. Against the backdrop of the Helmholtz decomposition theorem, stating that every vector field can be orthogonally decomposed into one divergence-free component and a second curl-free one, we proposed in \cite{SVF} to extend the SVF model by incorporating penalisation of the curl of $w$. 
However, as we shall dwell on in Section \ref{sec:unifiedmodel}, such an extended model still had not yet provided satisfactory results, since the point artefacts could indeed be reduced, but were still visible. Hence, we concluded that further adjustments to our model were needed. Inspired by the idea to combine penalisations of divergence, curl and shear to regularise motion flow fields \cite{Schnoerr}, we eventually contrived the following image denoising model, which (dependent on the weights chosen) enforces a joint vector operator sparsity (VOS) of divergence, curl and the two components of the  shear:
\begin{align}
\frac{1}{2} \int_{\Omega} (u - f)^2 \,dx + \alpha 
\int_{\Omega} | \nabla u - w | \,dx + \alpha \bigints_{\Omega} \begin{vmatrix}\sqrt{\beta_1}\,\mycurl(w)\\ \sqrt{\beta_2}\,\mydiv(w)\\ \sqrt{\beta_3}\,\mysheara(w)\\ \sqrt{\beta_4}\, \myshearb(w)
\end{vmatrix} \,dx \rightarrow  \min\limits_{u,w},
\tag{VOS}
\label{eq:VOS}
\end{align}
where $\alpha>0$ is a regularisation parameter in the classical sense while the $\beta_i>0$ are determining the specific form of the regularisation functional.
 
In this paper, we will show results for image denoising, but similar to existing TV-type regularisers our novel approach is not limited to this field of application, but can rather be used as a regulariser for a large variety of image reconstruction problems. To apply the \eqref{eq:VOS} model in the context of a different imaging task, the squared 
$L^2$-norm would have to be replaced by a suitable distance measure $D(Au,f)$, where $A$ denotes the bounded linear forward operator between two Banach spaces corresponding to the reconstruction problem to be solved. The fidelity term $D(Au,f)$ would have to be chosen in dependence on the expected noise characteristics and specific application as it is common practice in variational modelling (cf. \cite{TVZoo,daspaper}). However, for the sake of simplicity and to provide a good intuition for the effects of our novel regulariser on the reconstruction result, we will adhere to image denoising for the remainder of this paper.
\end{sloppypar}

\begin{sloppypar}
To summarise our contributions, we provide a way of looking at well-established TV-type regularisation methods from a new angle. 
We introduce a functional that generalises both our model presented in \cite{SVF} and the methods discussed above, formulated by applying sparsity constraints to common natural differential vector field operators. 
In contrast to improving state-of-the-art imaging methods, we rather focus on a sound mathematical analysis of our regulariser incorporating analysis of the nullspaces, which allows us to draw conclusions on optimal parameter combinations. 
Even more, we investigate under which conditions imposed on the weighting parameters we obtain rotational invariance. 
We also show that we can yield competitive denoising results sharing the ability of second-order models to reconstruct sharp edges and smooth intensity transitions simultaneously. 
Moreover, we highlight the fact that our model is able to interpolate between \eqref{eq:ICTV} and \eqref{eq:TGV} by only modifying one parameter. 
We also include a discussion on our discretisation, which is different from the one for the latter models, but has its own merits with respect to compliance with conservation laws.

Particularly, the remainder of this paper is organised as follows: 
In the subsequent section we very briefly recall some notions in the context of Radon measures relevant for the further course of this work. 
Afterwards, exact definitions of the differential operators included in the \eqref{eq:VOS} model will be stated in Section \ref{sec:diffop}. 
We will investigate both theoretically and practically how regularisation where only one $\beta_i$ is non-zero affects image reconstruction. 
In fact, all of the four resulting cases will involve certain characteristic artefacts that can be rigorously explained by studying the corresponding nullspaces of the regulariser. As we will show in Section \ref{sec:unifiedmodel}, the VOS model is indeed capable of producing denoising results with sharp edges and smooth transitions between intensity values simultaneously at suitable choices of the weighting parameters. 
Even more, a rigorous discussion and analysis of this model will reveal further properties and will pave the way for the insight that our novel approach is a means of unifying the well-known first- and second-order TV-type models introduced above and as such it naturally offers possibilities for interpolation between them. In Section \ref{sec:discretisation}, the discretisation of our model is explained in detail, as it is not straightforward to choose due to the various vector field operators involved. We compare our specific type of discretisation with the one in \cite{TGV} and justify our choice by showing that we comply with various conservation laws. In Section \ref{sec:results}, we briefly discuss the numerical solution of our model, compare the best result we can obtain to state-of-the-art methods illustrating that the proposed approach can indeed compete with those of existing second-order TV-type models. We furthermore present statistics on how various parameter combinations affect reconstructions with respect to different quality measures. We conclude the paper with a summary of our findings and future perspectives in Section \ref{sec:conclusion}.
\end{sloppypar}

\section{Preliminaries}
\label{sec:prelim}
\begin{sloppypar}
In the previous section we have introduced the total variation of a function $u \in L^1(\Omega)$ as
\begin{equation*}
\text{TV}(u) = \sup\limits_{\substack{\varphi \in C_c^\infty(\Omega,\mathbb{R}^2)\\||\varphi||_{\infty} \leq 1}} \int_{\Omega} u\ \mydiv(\varphi)\,dx.
\end{equation*}
On this basis one defines the space of functions of bounded variation by
\begin{equation*}
BV(\Omega) = \lbrace u \in L^1(\Omega): \text{TV}(u) < \infty \rbrace,
\end{equation*}
which equipped with the norm 
\begin{equation*}
\Vert u \Vert_{\BV} = \Vert u \Vert_1 + \text{TV}(u)
\end{equation*}
constitutes a Banach space. 
It is a well-known fact (cf. e.g. \cite{aubert_kornprobst_Mathematical_Problems_in_Image_Processing}, Chapter 2) that for $u \in BV(\Omega)$ the distributional gradient $\nabla u$ of $u$ can be identified with a finite vector-valued Radon measure, which can be characterised in the following way (cf. e.g. \cite{ambrosio_et_al_Functions_of_Bounded_Variation}, Chapter 1): 
Let $\mathcal{B}(\Omega)$ denote the Borel $\sigma$-algebra generated by the open sets in $\Omega$. 
Then we call a mapping $\mu\colon \mathcal{B}(\Omega) \rightarrow \mathbb{R}^d$, $d \geq 1$, an $\mathbb{R}^d$-valued, finite Radon measure if $\mu(\emptyset) = 0$ and $\mu$ is $\sigma$-additive, i.e. for any sequence $(A_n)_{n \in \mathbb{N}}$ of pairwise disjoint elements of $\mathcal{B}(\Omega)$ it holds that $ \mu\left( \bigcup_{n = 1}^\infty A_n \right) = \sum_{n =1}^\infty \mu(A_n) $. Moreover, we denote the space of $\mathbb{R}^d$-valued finite Radon measures by
\begin{align*}
\mathcal{M}(\Omega,\mathbb{R}^d) = \lbrace \mu\colon\mathcal{B}(\Omega) \rightarrow \mathbb{R}^d:\mu \text{ is }\mathbb{R}^d\text{-valued, finite Radon measure} \rbrace.
\end{align*}
By means of the Riesz-Markov representation theorem the space of the $\mathbb{R}^d$-valued finite Radon measures can be identified with the dual space of $C_0(\Omega,\mathbb{R}^d)$ under the pairing
\begin{equation*}
(\varphi,\mu) = \sum_{i=1}^d \int_{\Omega} \varphi_i \, d\mu_i \quad \text{ for } \varphi \in C_0(\Omega,\mathbb{R}^d).
\end{equation*}
Consequently, we equip the space of the $\mathbb{R}^d$-valued finite Radon measures with the dual norm
\begin{align*}
\Vert \mu \Vert_{\mathcal{M}(\Omega,\mathbb{R}^d)} = \sup\limits_{\substack{\varphi \in C_0(\Omega,\mathbb{R}^d)\\||\varphi||_{\infty} \leq 1}} \vert (\varphi,\mu) \vert = \sup\limits_{\substack{\varphi \in C_0(\Omega,\mathbb{R}^d)\\||\varphi||_{\infty} \leq 1}} \sum_{i=1}^d \int_{\Omega} \varphi_i \, d\mu_i
\end{align*}
yielding a Banach space structure for $\mathcal{M}(\Omega,\mathbb{R}^d)$. 
Now taking into account that for $u \in BV(\Omega)$ the distributional gradient is a finite $\mathbb{R}^2$-valued Radon measure we can consider
\begin{equation*}
\Vert \nabla u \Vert_{\mathcal{M}(\Omega,\mathbb{R}^2)} = \sup\limits_{\substack{\varphi \in C_0(\Omega,\mathbb{R}^2)\\||\varphi||_{\infty} \leq 1}} \vert (\varphi,\nabla u) \vert.
\end{equation*}
By the density of the space of test functions $C_c^\infty(\Omega)$ in $C_0(\Omega)$, we moreover obtain the following identity:
\begin{align*}
\Vert \nabla u \Vert_{\mathcal{M}(\Omega,\mathbb{R}^2)} = \sup\limits_{\substack{\varphi \in C_c^\infty(\Omega,\mathbb{R}^2)\\||\varphi||_{\infty} \leq 1}} \vert (\varphi,\nabla u) \vert = \sup\limits_{\substack{\varphi \in C_c^\infty(\Omega,\mathbb{R}^2)\\||\varphi||_{\infty} \leq 1}} \int_{\Omega}u \mydiv(\varphi) \, dx = \text{TV}(u),
\end{align*}
where the second equality results from the definition of the distributional gradient.
We thus see that for $u \in BV(\Omega)$ its total variation equals just the Radon norm of its distributional gradient. 
For this reason an alternative approach towards the definition of the space of bounded variation characterises functions $u \in L^1(\Omega)$ as elements of $BV(\Omega)$ if their distributional gradient is representable by a finite $\mathbb{R}^d$-valued Radon measure. 
However, there also exists a dissimilarity between $\Vert \nabla u \Vert_{\mathcal{M}(\Omega,\mathbb{R}^2)}$ and $\text{TV}(u)$: while by its characteristic as a norm the former can only attain values in $\left[0,\infty\right)$, the latter can not only be defined for functions in $BV(\Omega)$, but also for any function in $L^1(\Omega)$, since it can equal infinity. We will come back to this point shortly.
\end{sloppypar}

\begin{sloppypar}
In view of the previously summarised insights it seems natural to implement the infimal convolution to balance between enforcing sparsity of the distributional gradient of $u$ and some differential operator of a finite $\mathbb{R}^d$-valued Radon measure $w$ resembling $\nabla u$ by means of Radon norms. 
In the following, we will thus pursue this approach.
In doing so, we however will slightly abuse notation by extending the Radon norm to a broader class of generalised functions similar to $\text{TV}$ that is defined for a broader class of functions than the actual Radon norm of the distributional gradient. Here, we will adhere to the notation of the Radon norm and just set it to infinity whenever the argument is no finite $\mathbb{R}^d$-valued Radon measure, but only an element of the more general class of distributions.
\end{sloppypar}


\section{Natural Differential Operators on Vector Fields}
\label{sec:diffop}
In Section \ref{sec:introduction} we recalled the \eqref{eq:SVF} model for image denoising and already mentioned that due to point artefacts the obtained denoising results were unsatisfactory. 
Nevertheless, we decided to adhere to the idea of realising penalisation of second-order derivative information by applying natural vector operators to a two-dimensional vector field $w$ resembling the gradient of $u$. 
Against the backdrop of the Helmholtz respectively the Hodge decomposition theorem and inspired by the work of Schn\"orr \cite{Schnoerr}, the differential operators we are going to consider besides the divergence are the curl and the two components of the shear.   
In this section, we first give precise definitions of these operators in 2D. 
In a next step, we then reexamine the SVF model and moreover consider three alternatives, where the divergence operator is replaced by one of the aforementioned natural vector operators, namely the curl respectively one component of the shear. 
We show denoising results for the respective models revealing that each regulariser leads to very distinct artefacts that we can explain rigorously by analysing the corresponding nullspaces.

\subsection{Differential Operators on 2D Vector Fields}\label{subsec:diffop2dvf}
The curl is traditionally defined for three-\-di\-men\-sio\-nal vector fields and there is no unique way to define it in two dimensions. We chose the following definition of the curl of a 2D vector field $z$:
\begin{align}
\mycurl(z) = \frac{\partial z_2}{\partial x_1} - \frac{\partial z_1}
{\partial x_2}.
\tag{curl}
\label{eq:defCurl}
\end{align}
The definition of the divergence is well-known and is given as
\begin{align}
\mydiv(z) = \frac{\partial z_1}{\partial x_1} + \frac{\partial z_2}{\partial x_2}.
\tag{div}
\label{eq:defDiv}
\end{align}
As mentioned in Section \ref{sec:introduction}, incorporating the shear as a component of a sparse regulariser for vector fields has first been introduced by Schn\"orr in \cite{Schnoerr}. It consists of two components, each of which we consider separately. Their definitions also differ slightly in the literature and we decided to choose the following two:
\begin{align}
\mysheara(z) &= \frac{\partial z_2}{\partial x_2} - \frac{\partial z_1}{\partial x_1};
\tag{sh1}
\label{eq:defShear1} \\
\myshearb(z) &= \frac{\partial z_1}{\partial x_2} + \frac{\partial z_2}{\partial x_1}.
\tag{sh2}
\label{eq:defShear2}
\end{align}
\subsection{Sparsity of Scalar-Valued Natural Differential Operators}
\begin{figure}[h]
\captionsetup[subfigure]{labelformat=empty}
\centering
\subfloat[Sparse curl]{\includegraphics[height=3.75cm]{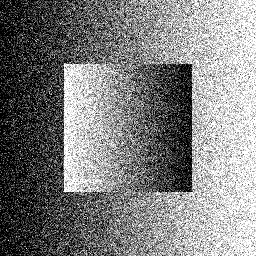}}\hfill
\subfloat[Sparse div]{\includegraphics[height=3.75cm]{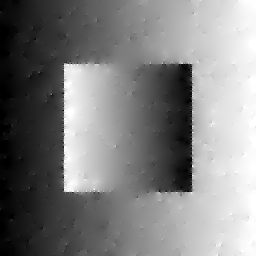}}\hfill
\subfloat[Sparse sh$_1$]{\includegraphics[height=3.75cm]{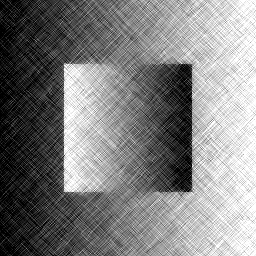}}\hfill
\subfloat[Sparse sh$_2$]{\includegraphics[height=3.75cm]{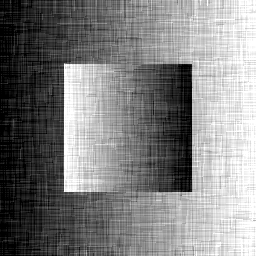}}
\caption{Reconstruction of piecewise affine test image using \eqref{eq:SparseDiffOp} for different vector operators $S$}
\label{fig:sparseDiffOp}
\end{figure}
In Figure \ref{fig:sparseDiffOp}, we can see how enforcing sparsity of one of the four different aforementioned scalar-valued natural vector operators applied to the vector field $w$ in \eqref{eq:SVF} changes the reconstruction $u$. More precisely, we consider the model
\begin{align}
\frac{1}{2} \int_{\Omega} (u - f)^2 \,dx + \inf\limits_{w \in {\cal M}(\Omega,\mathbb{R}^2)} \alpha \Vert \nabla u - w \Vert_{\mathcal{M}(\Omega,\mathbb{R}^2)} + \alpha\sqrt{\beta} \Vert S(w)  \Vert_{\mathcal{M}(\Omega)} \rightarrow \min\limits_{u \in L^2(\Omega)},
\tag{gSVF}
\label{eq:SparseDiffOp}
\end{align}
where $S$ corresponds to one of the vector field operators defined in \eqref{eq:defCurl} - \eqref{eq:defShear2}. Here and in the following we will slightly abuse notation and write derivatives of the measure $w$, which are however to be interpreted in a distributional sense. 
We first identify $S(w)$ with the linear functional
$$ \varphi \in C_0^\infty(\Omega) \mapsto  \int_\Omega S^*\varphi(x) \cdot dw. $$
If this linear functional is bounded in the predual space of $\mathcal{M}(\Omega)$, the space of continuous functions with compact support, then we can identify it with a Radon measure $S(w)$ and define  $\Vert S(w)  \Vert_{\mathcal{M}(\Omega)} $, otherwise we set it to infinity.


In order to understand the appearance of artefacts as above, it is instructive to study the nullspaces of the differential operators, as the following lemma shows, providing a result similar to \cite{GroundStates}:
\begin{lemma} \label{artefactlemma0}
Let $R: L^2(\Omega) \rightarrow \R \cup \{+\infty\}$ be a convex absolutely one-homogeneous functional, i.e.\ $R(c u) = \vert c \vert R(u)\ \forall c \in \R$. Then for each $u_0 \in L^2(\Omega)$ with $R(u_0) = 0$ we have 
\begin{equation}
	R(u+u_0) = R(u), \qquad \forall~ u \in L^2(\Omega).
\end{equation}
Moreover, let $f=f_0+g$ with $R(f_0)=0$ and $\int_\Omega f_0 g ~dx = 0$. Then the minimiser $\hat u$ of 
\begin{equation} \label{eq:Rvariational}
	\frac{1}2 \Vert u -f \Vert^2 + \alpha R(u) \rightarrow \min_{u \in L^2(\Omega)}
\end{equation}
is given by $ \hat u =  f_0 + u_*$ with 
$\int_\Omega f_0 u_*~dx = 0$ and
\begin{align*}
\Vert u_* - g \Vert_{2} \geq \min\{\alpha \lambda_0, \Vert g \Vert_{2} \}, \quad R(u_*) \leq R(g) - \frac{1}{2\alpha} \min\{\alpha \lambda_0, \Vert g \Vert_{2} \}^2,
\end{align*}
where $\lambda_0$ is the smallest positive eigenvalue of $R$. 
\end{lemma}
\begin{proof}
Convexity and positive homogeneity imply a triangle inequality, hence
\begin{equation*}
R(u) - R(-u_0) \leq R(u+u_0) \leq R(u) + R(u_0), 
\end{equation*}
and since $R(u_0) = R(-u_0) = 0$, we conclude $R(u+u_0) = R(u)$.

Now consider the variational model \eqref{eq:Rvariational} and write $u=cf_0 +v$ with 
$\int_\Omega v f_0~dx = 0$. Then we have
\begin{align*}
\frac{1}2 \Vert u - f\Vert_2^2 + \alpha R(u) = \frac{1}2 \Vert (c-1) f_0 \Vert_2^2 + \frac{1}2 \Vert v - g \Vert_2^2 + \alpha R(v).
\end{align*}
The first term is minimised for $c=1$ and the second for $v = u_*$ with $u_*$ being the solution of
\begin{equation*}
\frac{1}2 \Vert u - g \Vert_2^2 + \alpha R(u) \rightarrow \min_{u \in L^2(\Omega)}.
\end{equation*}
It remains to verify that indeed
$\int_\Omega u_* f_0~dx = 0$.
Since the Fr\'echet subdifferential of the functional to be minimised is the sum of the Fr\'echet derivative of the first term and the subdifferential of the regularisation term (cf.\ e.g.\ \cite[Theorem 23.8]{rockafellar}), the solution $u_*$ satisfies the optimality condition $u_* = g + \alpha p_*$ for $p_* \in \partial R(u_*)$. We refer to \cite[Chapter I, Section 5]{EkelandTemam} for a formal definition of the subdifferential.
Since by definition of a subgradient of $R$
\begin{align*}
\int_\Omega p_* f_0 ~dx \leq R(f_0) = 0, \quad \int_\Omega p_* (-f_0) ~dx \leq R(-f_0) = 0 ,
\end{align*}
we obtain the orthogonality relation because 
$\int_\Omega g f_0~dx = 0$. The lower bound on $\Vert u_* -g \Vert_2$ follows from a result in \cite[Section 6]{GroundStates}, the upper bound on the regularisation follows from combining this estimate with
\begin{equation*}
\frac{1}2 \Vert u_* - g \Vert_2^2 + \alpha R(u^*) \leq \alpha R(g),
\end{equation*}
which is due to the fact that $u_*$ is a minimiser of the functional with data $g$. 
\end{proof}

\begin{sloppypar}
Lemma \ref{artefactlemma0} has a rather intuitive interpretation: while the nullspace component with respect to $R$ in the signal is unchanged in the reconstruction, the part orthogonal to the nullspace is changed. Indeed this part is shrunk in some sense, $u_*$ has a smaller value of the regularisation functional than $g$. Hence, when rescaling the resulting image for visualisation, the nullspace component is effectively amplified. As a consequence, we proceed to a study of nullspaces for the different models 
with
\begin{align*}
R(u) = \inf_{w \in \mathcal{M}(\Omega,\mathbb{R}^2)} \Vert \nabla u - w \Vert_{\mathcal{M}(\Omega,\mathbb{R}^2)} + \sqrt{\beta} \Vert S(w) \Vert_{\mathcal{M}(\Omega)}.
\end{align*}
\end{sloppypar}

\begin{itemize}

\item Let $S = \mycurl$ and choose $u \in C^2(\Omega)$, then we can set $w= \nabla u$ and since the curl of the gradient vanishes, we obtain the infimum at zero. By a density argument $R$ vanishes on
$L^2(\Omega)$. Hence, Lemma \ref{artefactlemma0} with $g =0 $ shows that the data $f$ are exactly reconstructed by $\hat u$. 

\begin{sloppypar}
\item  Let $S = \mydiv$, which exactly resembles \eqref{eq:SVF}, and we can observe the point artefacts described above (cf. Figure \ref{fig:sparseDiffOp}, second image). Those are more difficult to be understood from the nullspace, which consists of harmonic functions ($w = \nabla u,\ \mydiv(w) = 0$). The latter is less relevant however for discontinuous functions, which are far away from harmonic ones. We rather expect to have a divergence of $w$ being sparse, i.e.\ a linear combination of Dirac $\delta$-distributions. Hence, with this structure of $\Delta u = \mydiv(w)$ the resulting $u$ would be the sum of a harmonic function and a linear combination of fundamental solutions of the Poisson equation, which exhibits a singularity at its centre in two dimensions. This singularity corresponds to the visible point artefacts.
\end{sloppypar}

\item With $S = \mysheara$, we observe a stripe-like texture pattern in diagonal directions. Here, $w=\nabla u$, $\mysheara(w) = 0$ yields a wave equation $\frac{\partial^2 u}{\partial x_1^2} = \frac{\partial^2 u}{\partial x_2^2}$. According to d'Alembert's formula (cf.\ e.g.\ \cite[pp.\ 65--68]{Evans}), the latter is solved by functions of the form 
$u=U(x_1+x_2) + V(x_1-x_2)$, which corresponds exactly to structures along the diagonal.

\item The artefacts in the case $S = \myshearb$ look similar, but the stripe artefacts are parallel to the $x_1$- and $x_2$-axes. Now the nullspace is characterised by $w=\nabla u,\ \myshearb(w) = 0$, which is equivalent to $\frac{\partial^2 u}{\partial x_1 \partial x_2} = 0$. This holds indeed for $u=U_1(x_1) + U_2(x_2)$, i.e.\ structures parallel to the coordinate axes. 

\end{itemize}

As observed already in the SVF model, we see from the above examples that the functional using any single differential operator has a huge nullspace and will not yield a suitable regularisation in the space of functions of bounded variation. On the other hand, using norms of the symmetric or full gradient as in TGV or ICTV is known to yield a regularisation in this space \cite{HigherOrderTV,TGV}. 
Thus, one may ask which and how many scalar differential operators one should combine to obtain a suitable functional.
In the subsequent section we will deduce an answer to this question, where in the end again a particular focus is laid on the four natural differential operators discussed above.


\section{Unified Model}
\label{sec:unifiedmodel}
\begin{sloppypar}
In view of the insights described in the previous section, we decided to consider a much more general approach, where no longer one natural vector operator is applied to $w$, but instead a general operator $\mathcal{A}$ is applied to the Jacobian of $w$ to penalise second-order derivative information.
We give a rigorous dual definition of the regularisation functional and state the corresponding subdifferential. 
By rephrasing this very general approach appropriately, we are eventually able to show that for a suitable choice of the general operator we can return to a formulation based on a weighted combination of the aforementioned natural vector field operators. 
We analyse the thus obtained model with respect to nullspaces and prove the existence of $BV$ solutions. 
In addition, we unroll that it is indeed justified to call the proposed approach a unified model, since we show that (at least in the limit) we can obtain the well-known second-order TV-type models ICTV, CEP and TGV as well as variations of first-order total variation as special cases.
Finally, we investigate under which conditions the presented approach is rotationally invariant.  
\end{sloppypar}

\subsection{General Second-Order TV-type Regularisations} 
\begin{sloppypar}
In a unified way any of the above regularisation functionals can be written in the form 
\begin{align} \label{eq:Rdefinition0}
R(u) = \inf_{w \in \mathcal{M}(\Omega,\mathbb{R}^2)} \Vert \nabla u - w \Vert_{\mathcal{M}(\Omega,\mathbb{R}^2)} + \Vert {\cal A} \nabla w \Vert_{\mathcal{M}(\Omega,\mathbb{R}^m)}
\end{align}
with a pointwise linear operator $\mathcal{A}: \mathbb{R}^{2 \times 2} \rightarrow \mathbb{R}^m$ independent of $x$ such that $\nabla w(x) \mapsto \mathcal{A}\nabla w(x)$ if $w$ has $C^1$ density, where in the above context $m = 1$. In the general setting we can use the distributional gradient and identify $\mathcal{A}\nabla w$ with the linear form
$$ \varphi \in C_0^\infty(\Omega,\mathbb{R}^m) \mapsto \int_\Omega \mydiv(\mathcal{A}^* \varphi(x)) \cdot dw. $$
We are interested in the case where this linear functional is bounded on the predual space of ${\mathcal{M}(\Omega,\mathbb{R}^m)}$, i.e. the space of continuous vector fields, and thus identify ${\cal A} \nabla w$ with such a vector measure justifying the use of the norm in \eqref{eq:Rdefinition0} 
(see also the equivalent dual definition below). 
Note that for $m < 4$ $\mathcal{A}$ will have a nullspace and hence $\mathcal{A}\nabla w$ being a Radon measure does not imply that $\nabla w$ is a  Radon measure. The product is hence rather to be interpreted as some differential operator $\mathcal{A}\nabla$ applied to the measure $w$ than $\mathcal{A}$ multiplied with $\nabla w$.

In view of \eqref{eq:Rdefinition0}, where as mentioned earlier $m=1$, we can derive a rigorous dual definition starting from 
\begin{align}
\begin{split}
&R(u)=\inf_{w \in {\cal M}(\Omega, \R^2)}  \sup_{(\varphi,\psi) \in {\cal B}_1} \int_\Omega u \mydiv(\varphi) ~dx + \int_\Omega  \varphi \cdot ~dw + \int_\Omega \mydiv ({\cal A}^*\psi) \cdot~dw,\\
&{\cal B}_1 = \{(\varphi, \psi) \in C_c^\infty(\Omega,\mathbb{R}^2) \times C_c^\infty(\Omega) : \Vert \varphi \Vert_\infty \leq 1, \Vert \psi \Vert_\infty \leq 1 \}.
\end{split}
\end{align}
Assuming that we can exchange the infimum and supremum, i.e.
\begin{align*}
R(u)=\sup_{(\varphi,\psi) \in {\cal B}_1} \inf_{w \in \mathcal{M}(\Omega,\mathbb{R}^2)} \int_\Omega u \mydiv(\varphi) ~dx + \int_\Omega  \varphi \cdot ~dw + \int_\Omega  \mydiv ({\cal A}^*\psi) \cdot~dw,
\end{align*}
we see that a value greater than $-\infty$ in the infimum only appears if 
$\varphi + \mydiv ({\cal A}^*\psi) = 0$. Thus, we can restrict the supremum to such test functions, which actually eliminates $w$ and $\varphi$, and  obtain the following formula reminiscent of the TGV-functional \cite{TGV}:
\begin{align} \label{eq:Rdefinition1}
&R(u)= \sup_{\psi \in {\cal B}_1^*} \int_\Omega  u \mydiv^2( {\cal A}^*\psi) ~dx,\\
 \label{eq:B1star}
&{\cal B}_1^* = \{ \psi \in C_c^\infty(\Omega): \Vert \psi \Vert_\infty \leq 1, \Vert \mydiv ({\cal A}^*\psi) \Vert_\infty \leq 1 \}.
\end{align}
\end{sloppypar}

\begin{sloppypar}
We see that there is an immediate generalisation of the above definition when we want to use more than one scalar differential operator for regularising the vector-valued measure $w$, we simply need to introduce 
a pointwise linear operator $\mathcal{A}: \mathbb{R}^{2 \times 2} \rightarrow \mathbb{R}^m$ with $m \geq 1$. 
Then the definition
\eqref{eq:Rdefinition1} remains unchanged if we adapt the admissible set 
\begin{align}\label{eq:B1stargeneralm}
{\cal B}_1^* = \{ \psi \in C_c^\infty(\Omega,\mathbb{R}^m): \Vert \psi \Vert_\infty \leq 1, \Vert \mydiv ({\cal A}^*\psi) \Vert_\infty \leq 1 \}.
\end{align}
\end{sloppypar}

\begin{sloppypar}
Let us provide some analysis of the above formulations. First of all we show that the infimal convolution is exact, i.e. for given $u \in BV(\Omega)$ the infimum is attained for some $\overline{w} \in \mathcal{M}(\Omega,\mathbb{R}^2)$.
\begin{lemma}
Let $u \in BV(\Omega)$, then there exists $\overline{w} \in \mathcal{M}(\Omega,\mathbb{R}^2)$ such that
\begin{align*}
\inf_{w \in \mathcal{M}(\Omega,\mathbb{R}^2)} \Vert \nabla u - w \Vert_{\mathcal{M}(\Omega,\mathbb{R}^2)} + \Vert {\cal A} \nabla w \Vert_{\mathcal{M}(\Omega,\mathbb{R}^m)} = \Vert \nabla u - \overline w \Vert_{\mathcal{M}(\Omega,\mathbb{R}^2)} + \Vert {\cal A} \nabla \overline{w} \Vert_{\mathcal{M}(\Omega,\mathbb{R}^m)}.
\end{align*}
\end{lemma}
\begin{proof}
We consider the convex functional
\begin{align*}
F(w) = \Vert \nabla u - w \Vert_{\mathcal{M}(\Omega,\mathbb{R}^2)} \Vert {\cal A} \nabla w \Vert_{\mathcal{M}(\Omega,\mathbb{R}^m)}.
\end{align*}
First of all $w=0$ is admissible and yields a finite value $F(0) = \Vert \nabla u \Vert_{\mathcal{M}(\Omega,\mathbb{R}^2)} < \infty$, since $u \in BV(\Omega)$. Thus, we can look for a minimiser of $F$ on the set $F(w) \leq F(0)$. For such $w$ the triangle inequality yields the bound
\begin{align*}
\Vert w \Vert_{\mathcal{M}(\Omega,\mathbb{R}^2)} + \Vert {\cal A} \nabla w \Vert_{\mathcal{M}(\Omega,\mathbb{R}^m)} \leq 2 \Vert \nabla u \Vert_{\mathcal{M}(\Omega,\mathbb{R}^2)}.
\end{align*}
In particular, $w$ and $\mathcal{A} \nabla w$ are uniformly bounded in $\mathcal{M}(\Omega,\mathbb{R}^2)$, which consequently also holds for minimising sequences $w_n$ and $\mathcal{A} \nabla w_n$. A standard argument based on the Banach-Alaoglu theorem and the metrisability of the weak-star topology on bounded sets (or alternatively cf.\ \cite[Theorem 1.59]{ambrosio_et_al_Functions_of_Bounded_Variation}) yields the existence of weak-star convergent subsequences $w_{n_k}$ and $\mathcal{A} \nabla w_{n_k}$.
Let $\overline{w} \in \mathcal{M}(\Omega,\mathbb{R}^2)$ denote the limit of the first subsequence $w_{n_k}$. Taking into account the continuity of the operator $\mathcal{A}\nabla$ in the space of distributions, the limit of the second subsequence $\mathcal{A} \nabla w_{n_k}$ equals $\mathcal{A} \nabla \overline{w}$. Then $\overline{w}$ is a minimiser due to the weak-star lower semicontinuity of both summands of $F$.
 \end{proof}
Next, we show the equivalence of the problem formulations in \eqref{eq:Rdefinition0} and \eqref{eq:Rdefinition1}.
\begin{lemma}
The definitions \eqref{eq:Rdefinition0} and \eqref{eq:Rdefinition1} with 
a pointwise linear operator $\mathcal{A}: \mathbb{R}^{2 \times 2} \rightarrow \mathbb{R}^m$ 
are equivalent, i.e.\ for all $u \in BV(\Omega)$ we have
\begin{align*}
\inf_{w \in \mathcal{M}(\Omega,\mathbb{R}^2)} \Vert \nabla u - w \Vert_{\mathcal{M}(\Omega,\mathbb{R}^2)} + \Vert {\cal A} \nabla w \Vert_{\mathcal{M}(\Omega,\mathbb{R}^m)} = \sup_{\psi \in {\cal B}_1^*} \int_{\Omega}  u \, \textup{div}^2( {\cal A}^*\psi) ~dx
\end{align*}
with ${\cal B}_1^*$ given by \eqref{eq:B1stargeneralm}.
\end{lemma}
\begin{proof}
The proof follows the line of argument in \cite{bredies2014regularization} (see also \cite{bredies2011inverse}) and is based on a Fenchel duality argument for the formulation, which we already sketched above. For this sake let $R_P$ denote the primal formulation \eqref{eq:Rdefinition0}
and rewrite the dual formulation $R_D$ given in \eqref{eq:Rdefinition1} as
\begin{align*}
R_D(u) = \sup_{\substack{(v_1,v_2) \in X\\\Lambda v = 0}} \int_\Omega u \mydiv(v_1) ~dx - I_1(v_1) - I_2(v_2) ,
\end{align*}
where we use the spaces 
$X = C_0^1(\Omega,\mathbb{R}^2) \times C_0^2(\Omega,\mathbb{R}^m)$, $Y = C_0^1(\Omega,\mathbb{R}^2)$,
the linear operator $\Lambda: X \rightarrow Y$, $\Lambda(v_1,v_2) = -v_1 + \mydiv ({\cal A}^* v_2)$, and the indicator functions
\begin{align*}
I_j(v_j) = \left\{ \begin{array}{ll} 0 & \text{if } \Vert v_j \Vert_\infty \leq 1 \\ + \infty & \text{else,}\end{array} \right. \quad j = 1,2.
\end{align*}
The equivalence of the supremal formulation on these spaces follows from the density of $C_c^\infty(\Omega)$ in $C_0^k(\Omega)$ for any $k$. 
Using the convex functionals $G: Y \rightarrow \R \cup \{+\infty\}$ as the indicator function of the set $\{0\}$ and $F: X \rightarrow \R \cup \{+\infty\}$ given by
\begin{equation*}
F(v) =\int_\Omega (-u \mydiv (v_1) + I_1(v_1) + I_2(v_2)) ~dx,
\end{equation*}
we can further write
\begin{equation*}
R_D(u) =\sup_{(v_1,v_2) \in X } - F(v) - G(\Lambda v).
\end{equation*}
In view of \cite[p.\ 12]{bredies2014regularization} it is straightforward to verify that
\begin{equation*}
Y = \bigcup_{\lambda \geq 0} \lambda (\text{dom}(G) - \Lambda \text{dom}(F)),
\end{equation*}
where $\text{dom}(F) = \{ x \in X : F(x) < \infty \}$ denotes the effective domain, and hence together with the convexity and lower semicontinuity the conditions for the Fenchel duality theorem \cite[Corollary 2.3]{attouch1986duality} are satisfied. Hence,
\begin{align*}
R_D(u) &= \inf_{w \in Y^*} F^*(-\Lambda^* w) + G^*(w)\\
&=\inf_{w \in Y^*} \sup_{(v_1,v_2) \in X} \left\lbrace (-\Lambda^*w,v) + u \mydiv (v_1) - I_1(v_1) - I_2(v_2)\right\rbrace\\
&=\inf_{w \in Y^*} \sup_{\substack{(v_1,v_2) \in X\\ \Vert v_1 \Vert_\infty \leq 1\\ \Vert v_2 \Vert_\infty \leq 1}} \lbrace (w,-\Lambda v) - (\nabla u,v_1) \rbrace\\
&=\inf_{w \in Y^*} \sup_{\substack{(v_1,v_2) \in X\\ \Vert v_1 \Vert_\infty \leq 1\\ \Vert v_2 \Vert_\infty \leq 1}} \left\lbrace (w,v_1 - \mydiv({\cal A}^* v_2))  - (\nabla u,v_1) \right\rbrace\\
&=\inf_{w \in Y^*} \sup_{\substack{(v_1,v_2) \in X\\ \Vert v_1 \Vert_\infty \leq 1\\ \Vert v_2 \Vert_\infty \leq 1}} \lbrace (w - \nabla u,v_1) +({\cal A} \nabla w,v_2) \rbrace\\
&=\inf_{w \in Y^*} \Vert \nabla u - w \Vert_{\mathcal{M}(\Omega,\mathbb{R}^2)} + \Vert {\cal A} \nabla w \Vert_{\mathcal{M}(\Omega,\mathbb{R}^m)},
\end{align*}
where the last conversion results from the definition of the Radon norm. Since $u \in BV(\Omega)$ the above functional only has a finite value for $w \in \mathcal{M}(\Omega,\mathbb{R}^2).$ Hence the infimum in the larger space $Y^*$ equals the infimum in $\mathcal{M}(\Omega,\mathbb{R}^2).$  This yields the assertion.

\end{proof}
\end{sloppypar}

Based on the dual formulation \eqref{eq:Rdefinition1} we can also understand the subdifferential of the absolutely one-homogeneous functional $R$. We see that $p \in \partial R(u)$ if $p =\mydiv^2( {\cal A}^*\psi)$ for $\psi \in  {{\cal B}_1^* }$ and
\begin{equation*}
\int_\Omega p~u~dx = \int_\Omega \mydiv^2( {\cal A}^*\psi)~u~dx = R(u).
\end{equation*}
In general, subgradients will be elements of a larger set, namely a closure of ${\cal B}_1^* $ in $L^\infty(\Omega)$ with the restriction that 
$\mydiv( {\cal A}^*\psi)$ can be integrated with respect to the measure $\nabla u$.

The domain of $R$ and the topological properties introduced are unclear at first glance and depend on the specific choice of ${\cal A}$. However, we can give a general result bounding $R$ by the total variation. 
\begin{lemma} \label{Rseminormlemma}
The functional $R$ is a seminorm on $BV(\Omega)$ and satisfies
$ R(u) \leq \textup{TV}(u) = \vert u  \vert_{BV}$
for all $u \in BV(\Omega)$. 
\end{lemma}
\begin{proof}
The fact that $R$ is a seminorm is apparent from the dual definition \eqref{eq:Rdefinition1}. From the primal definition \eqref{eq:Rdefinition0} we see that the infimum over all $w$ is less than or equal to the value at $w=0$, which is just $|u|_{BV}$. 
\end{proof}

\subsection{Combination of Natural Differential Operators}

As an alternative to the above form we can provide a matrix formulation when writing the gradient as a vector
\begin{equation*}
\nabla_V w := \left(\frac{\partial w_1}{\partial x_1},\frac{\partial w_1}{\partial x_2}, \frac{\partial w_2}{\partial x_1}, \frac{\partial w_2}{\partial x_2}\right)^T.
\end{equation*}
Then the operator ${\cal A}$ is represented by an $m \times 4$ matrix $\bm{A}$, and we have ${\cal A} \nabla w = \bm{A} \nabla_V w$. For the four scalar operators used above we obtain
\begin{align*}
\bm{A}_{\text{curl}} &= \sqrt{\beta_1}(0,-1,1,0),\qquad
\bm{A}_{\text{div}} = \sqrt{\beta_2}(1,0,0,1), \\
\bm{A}_{\text{sh}_1} &= \sqrt{\beta_3}(-1,0,0,1),\qquad
\bm{A}_{\text{sh}_2} = \sqrt{\beta_4}(0,1,1,0).
\end{align*}
Using the vector of natural differential operators
\begin{equation*}
\nabla_N w := (\mycurl(w), \mydiv(w), \mysheara(w), \myshearb(w))^{\top}
\end{equation*}
we can also write
\begin{align*}
\bm{A} \nabla_V w = \bm{B} \nabla_N w, \qquad
\bm{A} = \bm{B} \left( \begin{array}{cccc} 0 & -1 & 1 & 0 \\ 1 & 0 & 0 & 1 \\ -1 & 0 & 0 & 1 \\ 0 & 1 & 1 & 0 \end{array}\right) .
\end{align*}
We mention that due to the fact that we use the Frobenius norm, which has the property $\Vert z \Vert = \Vert {\bf Q} z \Vert$ for every orthogonal matrix ${\bf Q}$, two regularisations represented by matrices ${\bf A}_1$ and ${\bf A}_2$ will be equivalent if there exists an orthogonal matrix  ${\bf Q}$ with ${\bf A}_2 = {\bf Q} {\bf A}_1$. 

The question we would like to investigate in detail in the following paragraphs is whether enforcement of joint sparsity of some or all of the four natural differential operators \eqref{eq:defCurl} - \eqref{eq:defShear2} applied to the vector field $w$ can improve the reconstruction results. Moreover, we shall characterise a variety of models in the literature as special cases. This is not surprising, as we can always choose a suitable matrix $\bm{A}$ for any of those, but interestingly they can all be described by a diagonal matrix
\begin{equation*}
\bm{B}=\text{diag}(\sqrt{\beta_1},\sqrt{\beta_2},\sqrt{\beta_3},\sqrt{\beta_4}).
\end{equation*}
We will thus describe the regularisation functional solely in terms of the vector
\begin{equation*}
{\bm \beta}=(\sqrt{\beta_1},\sqrt{\beta_2},\sqrt{\beta_3},\sqrt{\beta_4})
\end{equation*}
as
\begin{align}  \label{eq:unifiedRegulariser}
R_{\bm \beta}(u) = \sup_{\varphi \in {\cal C}_{\bm \beta}} \int_\Omega u \mydiv(\varphi) ~dx = \inf_{w \in \mathcal{M}(\Omega,\mathbb{R}^2)} \Vert \nabla u - w \Vert_{\mathcal{M}(\Omega,\mathbb{R}^2)} + \Vert \text{diag}({\bm \beta})\nabla_N w\Vert_{\mathcal{M}(\Omega,\mathbb{R}^4)}
\end{align}
with 
\begin{align*}
{\cal C}_{\bm \beta} = \{ \varphi \in C_c^\infty(\Omega,\mathbb{R}^2): \varphi = \nabla_N^*( \text{diag}({\bm \beta}) \psi) \text{ for some } \psi \in C_c^\infty(\Omega,\mathbb{R}^4), \Vert \varphi\Vert_\infty \leq 1, \Vert \psi\Vert_\infty \leq 1\},
\end{align*}
where
\begin{align*}
\nabla^*_N \begin{pmatrix} \psi_1\\ \psi_2\\ \psi_3\\ \psi_4 \end{pmatrix} = \mycurl^* \psi_1 + \mydiv^* \psi_2 + \mysheara^* \psi_3 + \myshearb^* \psi_4
\end{align*}
and
\begin{align*}
\mycurl^* \psi &= \left(\frac{\partial \psi}{\partial x_2},-\frac{\partial \psi}{\partial x_1}\right)^T, \qquad \mydiv^* \psi = \left(-\frac{\partial \psi}{\partial x_1},-\frac{\partial \psi}{\partial x_2}\right)^T,\\
\mysheara^* \psi &= \left(\frac{\partial \psi}{\partial x_1},-\frac{\partial \psi}{\partial x_2}\right)^T, \qquad \myshearb^* \psi = \left(-\frac{\partial \psi}{\partial x_2},-\frac{\partial \psi}{\partial x_1}\right)^T.
\end{align*}

Based on this regularisation we will study the model problem
\begin{equation}
\label{eq:UnifiedModel}
\frac{1}2 \int_\Omega (u-f)^2~dx + \alpha R_{\bm \beta}(u)	\rightarrow \min_{u \in BV(\Omega)} 
\end{equation}
for $f\in L^2(\Omega)$, which of course can be extended directly to more general inverse problems and data terms. Note that $\alpha$ is a regularisation parameter in the classical sense, while the parameters $\beta_i$ are rather characterising the specific form of the regularisation functionals.

In Section \ref{sec:diffop}, we have presented reconstruction results for the denoising problem \eqref{eq:SparseDiffOp} and the effect of regularisation incorporating one of the four scalar-valued vector operations \eqref{eq:defCurl}-\eqref{eq:defShear2}, i.e.\ for only one of the weights $\beta_i$ being non-zero. In the following, we demonstrate how our model behaves when two, three or all $\beta_i$ are non-zero.

\begin{figure}
\captionsetup[subfigure]{labelformat=empty}
\centering
\subfloat[Test image]{\includegraphics[height=3.75cm]{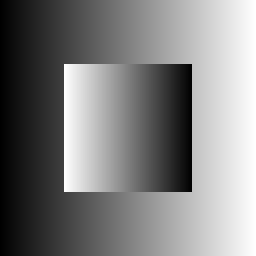}}\hspace{1cm}
\subfloat[Noisy image]{\includegraphics[height=3.75cm]{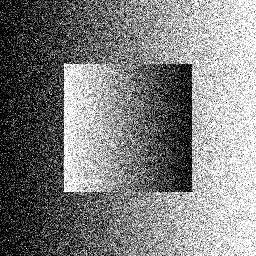}}\hspace{1cm}
\subfloat[Sparse curl, div, sh$_1$, sh$_2$]{\includegraphics[height=3.75cm]{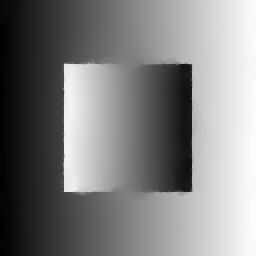}}\\
\subfloat[Sparse curl, div]{\includegraphics[height=3.75cm]{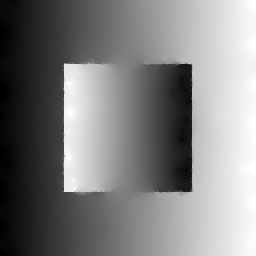}}\hspace{1cm}
\subfloat[Sparse curl, sh$_1$]{\includegraphics[height=3.75cm]{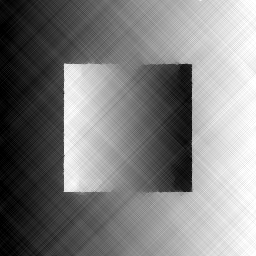}}\hspace{1cm}
\subfloat[Sparse curl, sh$_2$]{\includegraphics[height=3.75cm]{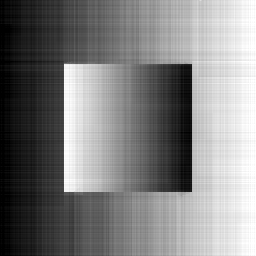}}\\
\subfloat[Sparse div, sh$_1$]{\includegraphics[height=3.75cm]{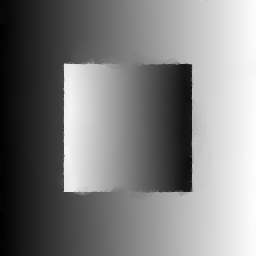}}\hspace{1cm}
\subfloat[Sparse div, sh$_2$]{\includegraphics[height=3.75cm]{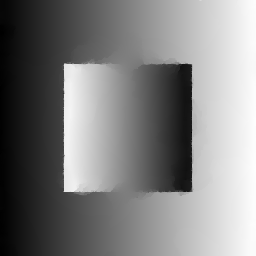}}\hspace{1cm}
\subfloat[Sparse sh$_1$, sh$_2$]{\includegraphics[height=3.75cm]{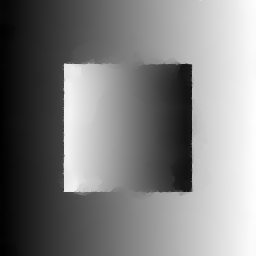}}\\
\subfloat[Sparse curl, div, sh$_1$]{\includegraphics[height=3.75cm]{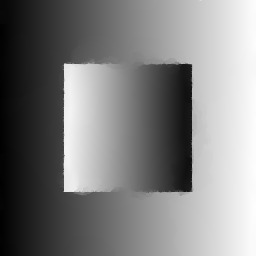}}\hfill
\subfloat[Sparse curl, div, sh$_2$]{\includegraphics[height=3.75cm]{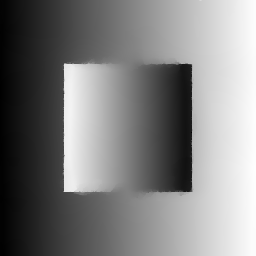}}\hfill
\subfloat[Sparse curl, sh$_1$, sh$_2$]{\includegraphics[height=3.75cm]{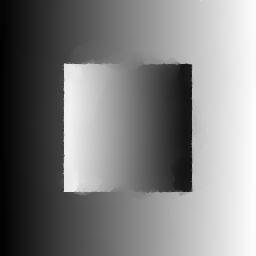}}\hfill
\subfloat[Sparse div, sh$_1$, sh$_2$]{\includegraphics[height=3.75cm]{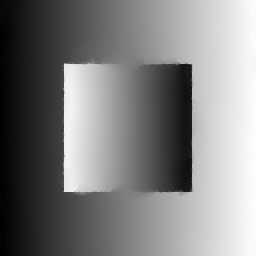}}
\caption{Reconstruction of a piecewise affine test image adding Gaussian noise with zero mean and variance $\sigma^2 = 0.05$ using \eqref{eq:UnifiedModel} for different parameter combinations}
\label{fig:reconTest}
\end{figure}

In Figure \ref{fig:reconTest}, we are given a piecewise affine test image and add Gaussian noise with zero mean and variance $\sigma^2 = 0.05$. For the task of denoising, we solve \eqref{eq:UnifiedModel} and vary the weights $\beta_i$. We optimise the parameters such that the structure similarity (SSIM) index is maximal. We can observe that setting two weights in our novel regulariser to zero still yields some artefacts in the reconstruction, especially in the case of enforcing a sparse curl in combination with one of the two components of the shear. As soon as we only set one of the four weights to zero, we obtain very good results, as can be seen in the bottom row of Figure \ref{fig:reconTest}. On the top right, the reconstruction with all weights being non-zero is presented, which yields a comparably good result.

In the following, we demonstrate that we can resemble special cases of already existing TV-type reconstruction models by modifying the weights in our regulariser \eqref{eq:unifiedRegulariser}. In particular, we show that we are able to retrieve \eqref{eq:TV}, \eqref{eq:CEP}, \eqref{eq:TGV} and \eqref{eq:ICTV}. However, before we discuss the relation of our proposed model to these existing regularisers in detail and even demonstrate that we can interpolate between the latter two by adapting one single weight, we shall elaborate on nullspaces and the existence of $BV$ solutions for our unified model given in \eqref{eq:UnifiedModel}.

\subsection{Nullspaces and Existence of $BV$ Solutions}

Our numerical results indicate that we obtain a real denoising resembling at least the regularity of a $BV$ solution if at least three of the $\beta_i$ are not vanishing. It is thus interesting to further study the nullspace ${\cal N}(R_{\bm \beta})$ of the regularisation functional $R_{\bm \beta}$ in such cases and check whether it is finite-dimensional. Subsequently, a similar argument to \cite{HigherOrderTV} can be made showing that the regularisation functional is equivalent to the norm in $BV(\Omega)$ on a subspace that does not include the nullspace. If the nullspace components are sufficiently regular, Lemma \ref{artefactlemma0} yields that minimisers of a variational model for denoising are indeed in $BV(\Omega)$.
In the following, we thus aim at characterising the set of all $u \in L^2(\Omega)$ for which $R_{\bm \beta}(u) = 0$ holds. Note that we provide further details on the derivation of the subsequent results in Appendix \ref{app:nullspaces}.
First of all, we directly see that $\beta_1$ plays a special role, since curl($\nabla u) = 0$. Thus, the case $\beta_1=0$ will yield the same nullspace as $\beta_1 > 0$. Hence, we only distinguish cases based on the other parameters:
\begin{itemize}

\item $\beta_i > 0$, $i=2,3,4$. In this case we have $\nabla u = w$ and $\nabla w =0$, the nullspace simply consists of affinely linear functions (see also \cite{HigherOrderTV}).
 
\item $\beta_2=0$, $\beta_3,\beta_4 > 0$. In this case we can argue similarly to Section \ref{sec:diffop} and see that $u=U(x_1+x_2) + V(x_1-x_2)=U_1(x_1)+U_2(x_2)$. Computation of second derivatives with respect to $x_1$ and $x_2$, respectively, yields the identity $U''(x_1+x_2) + V''(x_1-x_2)=U_1''(x_1)=U_2''(x_2).$ Thus, $U_1''$ and $U_2''$ are equal and constant. Integrating those with the constraint that $U_1$ and $U_2$ can only depend on one variable yields that the nullspace can only be a linear combination of $x_1^2+x_2^2$, $x_1$, $x_2$, $1$. One easily checks that these functions are indeed elements of the nullspace.

\item $\beta_3=0$, $\beta_2,\beta_4 > 0$. Now we see that $u$ is harmonic and on the other hand
$u=U_1(x_1)+U_2(x_2)$, which yields $U_1''(x_1) + U_2''(x_2) = 0$. The latter can only be true if $U_1''$ and $U_2''$ are constant, with constants summing to zero. Integrating those shows that the nullspace consists exactly of linear combinations of $x_1^2-x_2^2$, $x_1$, $x_2$, $1$.

\item $\beta_4=0$, $\beta_2,\beta_3 > 0$. A similar argument as above now yields $u=U(x_1+x_2) + V(x_1-x_2)$ and $U''(x_1+x_2)+V''(x_1-x_2)=0$. Again we obtain that $U''$ and $V''$ are constant, after integration we see that the nullspace consists exactly of linear combinations of $x_1 x_2$, $x_1$, $x_2$, $1$.

\end{itemize}

\begin{sloppypar}
This leads us to the following result characterising further the topological properties of the regularisation functionals, based on a Sobolev-Korn type inequality, which we state first.

\begin{lemma}
\label{lemma:Korn}
Let $\beta_i \geq 0$ for $i=1,\ldots,4$ and assume that at most one of the parameters $\beta_i$ vanishes. Then the Korn-type inequality \begin{equation}\label{eq:Korn}
\Vert w - P_{\bm{B}} w\Vert_{\mathcal{M}(\Omega,\mathbb{R}^2)} \leq C_{\bm{B}} \Vert \bm{B} \nabla_N w \Vert_{\mathcal{M}(\Omega,\mathbb{R}^4)}
\end{equation}
holds, 
where $P_{\bm{B}}$ is the projection onto the finite-dimensional nullspace ${\cal N}(\bm{B} \nabla_N w)$ of the differential operator $\bm{B} \nabla_N w$ and $C_{\bm{B}}$ is a constant depending on $\bm{B}$ only.
\end{lemma}
\begin{proof}
We will use the Korn inequality in measure spaces (cf. \cite{bredies2013symmetric}, Corollary 4.20), stating that for vector fields of bounded deformation the inequality
$$ \Vert w - P  w\Vert_{L^2(\Omega,\mathbb{R}^2)} \leq  C_S \Vert {\cal E}_S(w)\Vert_{\mathcal{M}(\Omega,\mathbb{R}^{2 \times 2})} $$
holds, where ${\cal E}_S(w)$ is the symmetric gradient and $P$ a projector onto its nullspace. We can equivalently write the inequality as 
$$ \Vert w - P  w\Vert_{L^2(\Omega,\mathbb{R}^2)} \leq  C \Vert {\cal E}(w)\Vert_{\mathcal{M}(\Omega,\mathbb{R}^{4})}, $$
where $\mathcal{E}(w)$ is the vectorised symmetric gradient
\begin{equation*}
{\cal E}(w) = \left(\frac{\partial w_1}{\partial x_1}, \frac{\frac{\partial w_1}{\partial x_2} + \frac{\partial w_2}{\partial x_1}}{2}, \frac{\frac{\partial w_1}{\partial x_2} + \frac{\partial w_2}{\partial x_1}}{2}, \frac{\partial w_2}{\partial x_2}\right)^\top.
\end{equation*}
Since on a bounded domain the total variation of a measure is a weaker norm than the $L^2$ norm of its Lebesgue density, we find
$$ \Vert w - P  w\Vert_{\mathcal{M}(\Omega,\mathbb{R}^2)} \leq  \tilde C \Vert {\cal E}(w)\Vert_{\mathcal{M}(\Omega,\mathbb{R}^4)}. $$
In order to verify the Korn-type inequality it is crucial to have three coefficients $\beta_i$ different from zero. 
In this case an elementary computation shows that there exists an invertible matrix
$\tilde{\bm{B}} \in \mathbb{R}^{2\times 2}$ and an orthogonal Matrix $\bm{Q} \in \mathbb{R}^{4 \times 4}$ such that
$$ \bm{Q} \bm{B} \nabla_N w = {\cal E}(\tilde{\bm{B}} w). $$
Thus, the Korn inequality applied to $\tilde w = \tilde{\bm{B}} w$ implies
\begin{align*} \Vert \tilde w - P  \tilde w\Vert_{\mathcal{M}(\Omega,\mathbb{R}^2)} & \; \leq  \tilde C \Vert {\cal E}(\tilde{\bm{B}} w)\Vert_{\mathcal{M}(\Omega,\mathbb{R}^{4})} \\  
& \; = \tilde C \Vert \bm{Q} \bm{B} \nabla_N w\Vert_{\mathcal{M}(\Omega,\mathbb{R}^{4})}  \\
& \; = \tilde C \Vert \bm{B} \nabla_N w\Vert_{\mathcal{M}(\Omega,\mathbb{R}^{4})} . 
\end{align*}
Since $P=\tilde{\bm{B}} P_N (\tilde{\bm{B}})^{-1}$ is a projector on the nullspace of ${\cal E}$, we obtain 
\begin{align*} 
\Vert   w - P_N  w\Vert_{\mathcal{M}(\Omega,\mathbb{R}^2)} & \; \leq \Vert \tilde{\bm{B}}^{-1} \Vert~\Vert \tilde w - P  \tilde w\Vert_{\mathcal{M}(\Omega,\mathbb{R}^2)} \\ 
& \; \leq   \Vert \tilde{\bm{B}}^{-1} \Vert~\tilde C \Vert \bm{B} \nabla_N w\Vert_{\mathcal{M}(\Omega,\mathbb{R}^{4})} . 
\end{align*}

If all $\beta_i$ are positive, we can use an analogous argument with $\bm{B} \nabla_N w = \nabla \tilde{\bm{B}} w$ and the Poincar{\'e}-Wirtinger inequality in spaces of bounded variation \cite{Bergounioux_Poincare}. 
\end{proof}

\begin{lemma}
Let $\beta_i \geq 0$ for $i=1,\ldots,4$. Then for $R_{\bm \beta}$ defined in \eqref{eq:unifiedRegulariser} the estimate
$
R_{\bm \beta}(u) \leq |u|_{BV} 
$ 
holds for all $ u \in BV(\Omega)$. Moreover, assume that at most one of the parameters $\beta_i$ vanishes and let ${\cal U} \subset BV(\Omega)$ be the subspace of all $BV$ functions orthogonal to ${\cal N}(R_{\bm \beta})$ in the $L^2$ scalar product. Then there exists a constant $c \in (0,1)$ depending only on ${\bm \beta}$ and $\Omega$ such that 
$
R_{\bm \beta}(u) \geq c |u|_{BV} 
$ for all $u \in {\cal U} $.
\end{lemma}
\begin{proof}
The first estimate is a special case of Lemma \ref{Rseminormlemma}. In order to verify the second inequality we proceed as in \cite{bredies2014regularization}. The key idea is to use the Korn-type inequality defined in Lemma \ref{lemma:Korn}. Given \eqref{eq:Korn}, we have
\begin{align*}
&\Vert \nabla u - w\Vert_{\mathcal{M}(\Omega,\mathbb{R}^2)} +  \Vert \bm{B} \nabla_N w \Vert_{\mathcal{M}(\Omega,\mathbb{R}^4)} \\
& \; \geq \Vert \nabla u - w\Vert_{\mathcal{M}(\Omega,\mathbb{R}^2)} + \frac{1}{C_{\bm{B}}}  \Vert w - P_{\bm{B}} w\Vert_{\mathcal{M}(\Omega,\mathbb{R}^2)} \\
& \; \geq \min\{1,\frac{1}{C_{\bm{B}}}\} ( \Vert \nabla u - w\Vert_{\mathcal{M}(\Omega,\mathbb{R}^2)} + \Vert w - P_{\bm{B}} w\Vert_{\mathcal{M}(\Omega,\mathbb{R}^2)}) \\ 
& \; \geq \min\{1,\frac{1}{C_{\bm{B}}}\}  \Vert \nabla u - P_{\bm{B}} w\Vert_{\mathcal{M}(\Omega,\mathbb{R}^2)}.
\end{align*}
Thus, taking the infimum over all $w$ yields
\begin{align*}
R_{\bm \beta}(u) & \; \geq  \min\{1,\frac{1}{C_{\bm{B}}}\}   \inf_{w \in \mathcal{M}(\Omega,\mathbb{R}^2)} \Vert \nabla u - P_{\bm{B}} w\Vert_{\mathcal{M}(\Omega,\mathbb{R}^2)}\\
& \; =  \min\{1,\frac{1}{C_{\bm{B}}}\}   \inf_{\tilde{w} \in {\cal N}(\bm{B} \nabla_N w)} \Vert \nabla u - \tilde{w}\Vert_{\mathcal{M}(\Omega,\mathbb{R}^2)},
\end{align*}
where the last equality results from the definition of the projection $P_{\bm B}$. 
It is then easy to see that for $u \in {\cal U}$ the optimal value is $\tilde{w}=0$. This implies the desired estimate.
\end{proof}
\begin{rem}
In the above analysis of the nullspaces we figured out that due to our choice of the first term $\Vert \nabla u - w \Vert_{\mathcal{M}(\Omega,\mathbb{R}^2)}$ of the regulariser $R_{\bm \beta}$, penalisation of the curl is irrelevant for the characterisation of the nullspaces of $R_{\bm \beta}$. Accordingly, the assertion of the above lemma can easily be extended to the cases $\beta_1 = \beta_2 = 0$ and $\beta_3,\beta_4 > 0$, $\beta_1 = \beta_3 = 0$ and $\beta_2,\beta_4 > 0$ as well as $\beta_1 = \beta_4 = 0$ and $\beta_2,\beta_3 > 0$, where the line of argument follows exactly the proof given above. For all remaining cases the above proof fails however, since in these cases the resulting nullspaces are not finite-dimensional.
\end{rem}
\begin{thm}
Let $f\in L^2(\Omega)$ and $\alpha > 0$. Moreover, let $\beta_i \geq 0$ for $i=1,\ldots,4$ and let at most one of the parameters $\beta_1,\ldots,\beta_4$ vanish. Then there exists a unique solution $\hat u \in BV(\Omega)$ of \eqref{eq:UnifiedModel}.
\end{thm}
\begin{proof}
We decompose $u=u_0+(u-u_0)$ and $f=f_0 +(f-f_0)$, where $u_0$ respectively $f_0$ are the $L^2$-projection on the nullspace of $R_{\bm \beta}$. Then
\begin{align*}
&\frac{1}2 \int_\Omega(u-f)^2~dx + \alpha R_{\bm \beta}(u)\\
& \; =  \frac{1}2 \int_\Omega(u_0-f_0)^2~dx +  \alpha R_{\bm \beta}(u-u_0) + \frac{1}2 \int_\Omega(u-u_0-f+f_0)^2~dx.
\end{align*}
Since $f_0  \in BV(\Omega)$, it is easy to see that the optimal solution is given by $u=f_0 + v$, where $v$ is a minimiser of
\begin{equation*}
\frac{1}2 \int_\Omega(v-f+f_0)^2~dx +  \alpha R_{\bm \beta}(v)  \rightarrow \min_{ v \in {\cal U}}
\end{equation*}
according to Lemma \ref{artefactlemma0}.

Since $R_{\bm \beta}$ is coercive on ${\cal U}$ and the functionals are lower semicontinuous in the weak-star topology on bounded sets, we conclude the existence of a minimiser by standard arguments. Uniqueness follows from strict convexity of the first term and convexity of $R_{\bm \beta}$. 
\end{proof}
\end{sloppypar}

\subsection{Special Cases}
\label{Subsec:special cases}
\begin{sloppypar}
In the following, we discuss several special cases of second-order functionals in the literature, which arise either by a special choice of the vector ${\bm \beta}$ or by letting elements in ${\bm \beta}$, in particular $\beta_1$, tend to infinity. For the sake of readability we will in all cases consider all models with additional parameters equal to one, the case of other values follows by simple scaling arguments. Throughout this section, for simplicity we denote by $\mathcal{E}(w)$ and $\nabla(w)$ the respective vectorised versions, i.e.
\begin{equation*}
{\cal E}(w) = \left(\frac{\partial w_1}{\partial x_1}, \frac{\frac{\partial w_1}{\partial x_2} + \frac{\partial w_2}{\partial x_1}}{2}, \frac{\frac{\partial w_1}{\partial x_2} + \frac{\partial w_2}{\partial x_1}}{2}, \frac{\partial w_2}{\partial x_2}\right)^\top
\end{equation*}
and
\begin{equation*}
\nabla(w) = \left(\frac{\partial w_1}{\partial x_1}, \frac{\partial w_1}{\partial x_2}, \frac{\partial w_2}{\partial x_1}, \frac{\partial w_2}{\partial x_2}\right)^\top.
\end{equation*} 
\end{sloppypar}
\subsubsection*{TGV}

The second-order TGV model \eqref{eq:TGV} in a notation corresponding to our approach is given by
\begin{align*}
&\text{TGV}(u) = \inf_{w \in \mathcal{M}(\Omega,\mathbb{R}^2)} \Vert \nabla u - w \Vert_{\mathcal{M}(\Omega,\mathbb{R}^2)} + \Vert {\cal E}(w) \Vert_{\mathcal{M}(\Omega,\mathbb{R}^4)}
\end{align*}
with ${\cal E}(w)$ being the symmetric gradient, encoded via the matrix 
\begin{equation*}
{\bf A}_{\text{TGV}} = \left( \begin{array}{llll} 1 & 0 & 0 & 0 \\0 & \frac{1}{2} & \frac{1}{2} & 0 \\ 0 & \frac{1}{2} & \frac{1}{2} & 0 \\ 0 & 0 & 0 & 1\end{array}\right).
\end{equation*}
Now let ${\bf B}=$diag$(0,\frac{1}{\sqrt{2}},\frac{1}{\sqrt{2}},\frac{1}{\sqrt{2}})$ and
\begin{equation*}
{\bf A}_1 =   {\bf B} \left( \begin{array}{cccc} 0 & -1 & 1 & 0 \\ 1 & 0 & 0 & 1 \\ -1 & 0 & 0 & 1 \\ 0 & 1 & 1 & 0 \end{array}\right) = \frac{1}{\sqrt{2}} \left(  \begin{array}{cccc} 0 & 0 & 0 & 0 \\ 1 & 0 & 0 & 1 \\ -1 & 0 & 0 & 1 \\ 0 & 1 & 1 & 0 \end{array}\right) .
\end{equation*}
We see that $ {\bf A}_{\text{TGV}} = {\bf Q} {\bf A}_1$ with the orthogonal matrix 
\begin{equation*}
{\bf Q} = \frac{1}{\sqrt{2}} \left(  \begin{array}{cccc} 0 & 1 & -1 & 0 \\ 1 & 0 & 0 & 1 \\ -1 & 0 & 0 & 1 \\ 0 & 1 & 1 & 0 \end{array}\right).
\end{equation*}
Hence, the TGV functional can be considered as a special case of \eqref{eq:unifiedRegulariser} with ${\bm \beta}=(0,\frac{1}{\sqrt{2}},\frac{1}{\sqrt{2}},\frac{1}{\sqrt{2}})$.

%

\subsubsection*{TGV with full gradient matrix}

A variant of the second-order TGV model is given by using the full gradient instead of the symmetric gradient, i.e.
\begin{align*}
\text{TGVF}(u) = \inf_{w \in \mathcal{M}(\Omega,\mathbb{R}^2)} \Vert \nabla u - w \Vert_{\mathcal{M}(\Omega,\mathbb{R}^2)} + \Vert \nabla w \Vert_{\mathcal{M}(\Omega,\mathbb{R}^4)}.
\end{align*}
 This can simply be encoded via ${\bf A}_{\text{TGVF}} = {\bf I}$ being the unit matrix in $\R^{4\times 4}$. Choosing ${\bf B} = \frac{1}{\sqrt{2}} {\bf I}$ we immediately see the equivalence, since
\begin{equation*}
{\bf A}_1 =   {\bf B} \left( \begin{array}{cccc} 0 & -1 & 1 & 0 \\ 1 & 0 & 0 & 1 \\ -1 & 0 & 0 & 1 \\ 0 & 1 & 1 & 0 \end{array}\right) = \frac{1}{\sqrt{2}} \left( \begin{array}{cccc} 0 & -1 & 1 & 0 \\ 1 & 0 & 0 & 1 \\ -1 & 0 & 0 & 1 \\ 0 & 1 & 1 & 0 \end{array}\right)
\end{equation*}
is already an orthogonal matrix and we obtain ${\bf A}_{\text{TGVF}} = {\bf A}_1^{\top} {\bf A}_1$. Hence, the TGV functional with full gradient matrix can be considered as a special case of \eqref{eq:unifiedRegulariser} with ${\bm \beta}=(\frac{1}{\sqrt{2}},\frac{1}{\sqrt{2}},\frac{1}{\sqrt{2}},\frac{1}{\sqrt{2}})$.

\subsubsection*{ICTV}

Let us now examine the relation to \eqref{eq:ICTV}, which rewritten in our notation becomes
\begin{align*}
\text{ICTV}(u) = \inf_{\substack{w \in \mathcal{M}(\Omega,\mathbb{R}^2)\\\mycurl(w) = 0}} \Vert \nabla u - w \Vert_{\mathcal{M}(\Omega,\mathbb{R}^2)} + \Vert \nabla w \Vert_{\mathcal{M}(\Omega,\mathbb{R}^4)}.
\end{align*}
In this case we do not need to distinguish between the gradient of $w$ and the symmetric gradient, since they are equal due to the vanishing curl.
Note that we have replaced the assumption of $w$ being a gradient by the equivalent assumption of vanishing curl, which corresponds better to our approach and indicates that we will need to consider the limit $\beta_1 \rightarrow \infty$.  Not surprisingly we will choose $\beta_2=\beta_3=\beta_4=\frac{1}{2} $ as in the TGV case. Thus, we will study the limit of $\beta_1 \rightarrow \infty$, using the notion of $\Gamma$-convergence (\cite{braides2002gamma,dal2012introduction}):

\begin{thm} \label{ictvthm}
Let $\beta_2=\beta_3=\beta_4=\frac{1}{2}$. We define ${\bm{\beta}^t}:=(t,\frac{1}{\sqrt{2}}, \frac{1}{\sqrt{2}},\frac{1}{\sqrt{2}})$, $t > 0$.  Then $R_{{\bm \beta}^t}$ $\Gamma$-converges to \textup{ICTV} strongly in $L^p(\Omega)$ for any $p < 2$ as $t\rightarrow \infty$, where we extend both functionals by infinity on $L^p(\Omega) \setminus BV(\Omega)$.
\end{thm} 
\begin{proof}
Let $t > 0$, $u_t \in BV(\Omega)$ and let $w_t \in \mathcal{M}(\Omega,\mathbb{R}^2)$ be a minimiser of 
\begin{equation*}
\Vert \nabla u_t - w \Vert_{\mathcal{M}(\Omega,\mathbb{R}^2)} + 
\Vert {\bf B}^t \nabla_N w \Vert_{\mathcal{M}(\Omega,\mathbb{R}^4)}
\end{equation*}
with ${\bf B}^t$ being the diagonal matrix with diagonal ${\bm \beta}^t$.
First, we consider the lower bound inequality.
To this end, we assume $u_t \rightarrow u$ strongly in $L^p(\Omega)$. Then we either have $\lim \inf_t 
R_{{\bm \beta}^t}(u_t) =  \infty$, which makes the lower bound inequality trivial, or $R_{{\bm \beta}^t}(u_t) $ bounded. If  $\lim\inf  R_{{\bm \beta}^t}(u_t)$ is finite, we immediately see from the norm equivalence and lower semicontinuity of the total variation that the limit $u$ has finite norm in $BV(\Omega)$. Hence, for $u \in L^p(\Omega) \setminus BV(\Omega)$ the lower bound inequality holds.
Thus, let us consider the remaining case of the limit inferior being finite and $u \in BV(\Omega)$. Then we see that 
\begin{align*}
{t} \Vert \mycurl (w_t) \Vert_{\mathcal{M}(\Omega)} \leq \Vert {\bf B}^t \nabla_N w \Vert_{\mathcal{M}(\Omega,\mathbb{R}^4)} \leq R_{{\bm \beta}^t}(u_t),
\end{align*}
which implies that $ \mycurl (w_t)$ strongly converges to zero in ${\cal M}(\Omega)$. 
Since $$ \Vert {\bf B}^t \nabla_N w \Vert_{\mathcal{M}(\Omega,\mathbb{R}^4)} \geq
\Vert {\bf B}^0 \nabla_N w  \Vert_{\mathcal{M}(\Omega,\mathbb{R}^4)}$$
for all $w$, 
we have
\begin{align*}
R_{{\bm \beta}^t}(u_t) & \; = \Vert \nabla u_t - w_t \Vert_{\mathcal{M}(\Omega,\mathbb{R}^2)} + 
\Vert {\bf B}^t \nabla_N w_t \Vert_{\mathcal{M}(\Omega,\mathbb{R}^4)}\\
& \; \geq 
\Vert \nabla u_t - w_t \Vert_{\mathcal{M}(\Omega,\mathbb{R}^2)} + 
\Vert {\bf B}^0 \nabla_N w_t \Vert_{\mathcal{M}(\Omega,\mathbb{R}^4)}.
\end{align*}
Due to the lower semicontinuity of the last term we see
\begin{align*}
\lim\inf_t  R_{{\bm \beta}^t}(u_t) & \; \geq \lim\inf_t \Vert \nabla u_t - w_t \Vert_{\mathcal{M}(\Omega,\mathbb{R}^2)} + \Vert {\bf B}^0 \nabla_N w_t \Vert_{\mathcal{M}(\Omega,\mathbb{R}^4)}\\
& \; \geq
\Vert \nabla u - w \Vert_{\mathcal{M}(\Omega,\mathbb{R}^2)} + 
\Vert {\bf B}^0 \nabla_N w \Vert_{\mathcal{M}(\Omega,\mathbb{R}^4)}, 
\end{align*}
where $w$ is a weak-star limit of an appropriate subsequence of $w_t$. The latter exists due to the boundedness of $w_t$ and satisfies $\mycurl(w) = 0$. 
Since the infimum over all curl-free $w$ is at most as large, we obtain the lower bound inequality
\begin{equation*}
\lim\inf_t  R_{{\bm \beta}^t}(u_t)  \geq \text{ICTV}(u).
\end{equation*}
Next, we consider the upper bound inequality. 
For $u \in L^p(\Omega) \setminus BV(\Omega)$ the upper bound inequality follows trivially with the sequence $u_t = u$ for all $t$. The upper bound inequality for $u \in BV(\Omega)$ is also easy to verify since for such $u$ we have
$ R_{{\bm \beta}^t}(u) \leq \text{ICTV}(u) $
due to the fact that we obtain exactly ICTV$(u)$ when we restrict the infimum in the definition of $ R_{{\bm \beta}^t}(u)$ to the subset of curl-free $w$. 
\end{proof}

An interesting observation is that we interpolate the two TGV models as well as the ICTV model solely by the parameter $\beta_1$, from the TGV model with the symmetric gradient ($\beta_1=0$) over the one with the full gradient $(\beta_1=\frac{1}{2})$ to the ICTV model in the limit $\beta_1 \rightarrow \infty$. 

\subsubsection*{Interpolation between TGV and ICTV}

\begin{sloppypar}
In this paragraph, we illustrate the previously described ability of our approach to interpolate between the ICTV and TGV model by means of a numerical test case. To this end, we corrupted an image section of the parrot test image from the Kodak image database\footnote{\url{http://r0k.us/graphics/kodak/}} by Gaussian noise with zero mean and a variance of $0.05$. Next, we applied the proposed denoising model \eqref{eq:UnifiedModel} to the noisy image data, where we always chose $\alpha = \frac{1}{4}$ and $\beta_2 = \beta_3 = \beta_4 =\frac{1}{2}$ and varied $\beta_1$ as follows: In order to obtain a second-order \enquote{TGV-type} reconstruction, we set $\beta_1$ equal to zero. For the \enquote{ICTV-type} model recovery that we obtain if $\beta_1$ tends to infinity, we chose $\beta_1 = 10^{10}$. The corresponding denoising results are depicted in the left and right column of the first row of Figure \ref{fig:interpolation}. Additionally, we calculated the respective interpolated denoising result for $\beta_1 \in \lbrace 10^{-4}, 10^{-3},10^{-2},10^{-1}, 0.25, 1, 2, 4, 25, 100, 2500, 10^{4}, 10^{6}\rbrace$, where $\beta_1=25$ yielded the best result with respect to the quality measure SSIM. The corresponding denoised image is shown in the middle of the top row of Figure \ref{fig:interpolation}. It is a well-known fact that the TGV and the ICTV model yield results of comparable quality and thus it is not surprising that all three denoising results as well as the error images in the second row of Figure \ref{fig:interpolation} look very similar. This visual inspection is further confirmed by the quality measure SSIM, since the differences are only in the range of $10^{-3}$. To point out that there are indeed slight differences between these denoising results, we also provide difference images between the result for $\beta_1=25$ and the TGV respectively the ICTV result in the third row of Figure \ref{fig:interpolation}.
While in the first rows of Figure \ref{fig:interpolation} we can hardly recognise any visual difference between the results of the three methods under consideration, the lower four rows of Figure \ref{fig:interpolation} reveal that in some sense the interpolated model is indeed in between the TGV and the ICTV model. 
In these rows, we plot the four different operators \eqref{eq:defCurl}-\eqref{eq:defShear2} applied to the vector field $w$ corresponding to the \enquote{TGV-type}, \enquote{interpolated} and \enquote{ICTV-type} reconstruction in the top row, respectively. Looking at these results, we can observe that the plots of the divergence and the two components of the shear apparently are rather similar and exhibit the same structures. On the other hand, the plot of the curl of the ICTV-type model seems to be very close to zero in the whole image domain, while in the curl of the interpolated model slight structures become visible, which are even more evident in the respective plot of the TGV model, exactly as expected.
\end{sloppypar}


\begin{figure}
\centering
\setlength{\tabcolsep}{1mm}
\begin{tabular}{B{0.55cm}B{0.25\textwidth}B{0.25\textwidth}B{0.25\textwidth}}
& {\it {\bf TGV}} & {\it {\bf Interpolated}} & {\it {\bf ICTV}} \\
\rotatebox[origin=c]{90}{{\bf recon}}&
\includegraphics[height=0.125\textheight]{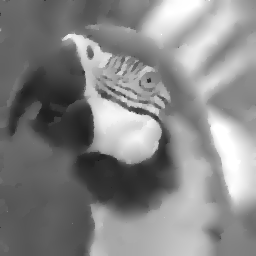}&
\includegraphics[height=0.125\textheight]{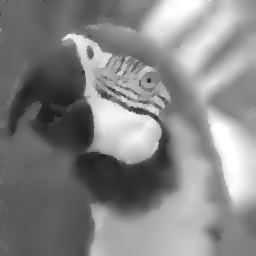}&
\includegraphics[height=0.125\textheight]{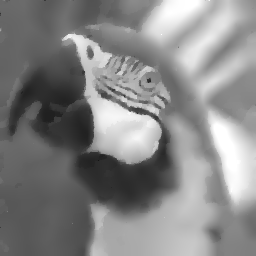}\\
& $\text{SSIM} = 0.78918$ & $\text{SSIM} = 0.79032$ & $\text{SSIM} = 0.78792$\\			
\rotatebox[origin=c]{90}{{\bf error img}}&
\includegraphics[height=0.12\textheight]{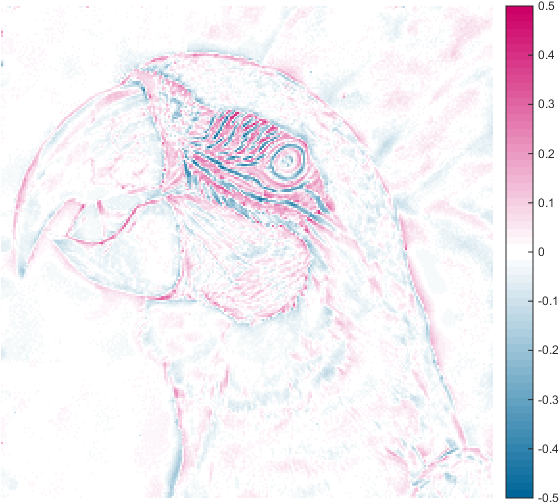}&
\includegraphics[height=0.12\textheight]{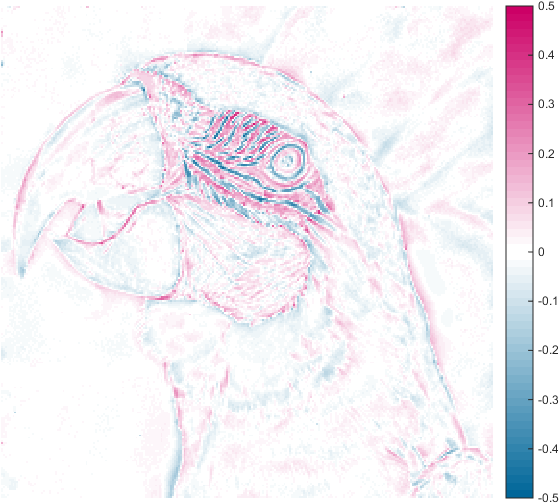}&
\includegraphics[height=0.12\textheight]{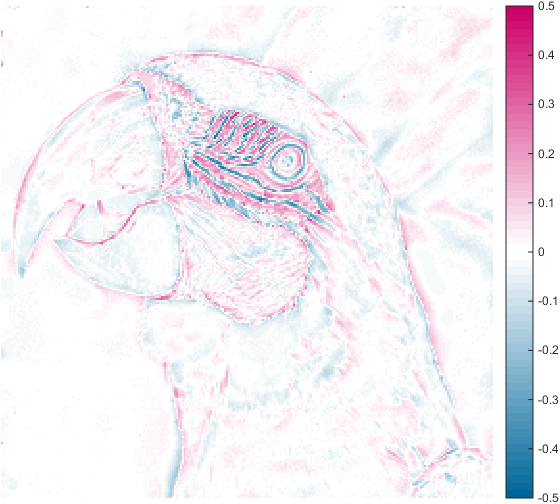}\\
& & & \\			
\rotatebox[origin=c]{90}{{\bf difference img}}&
\includegraphics[height=0.12\textheight]{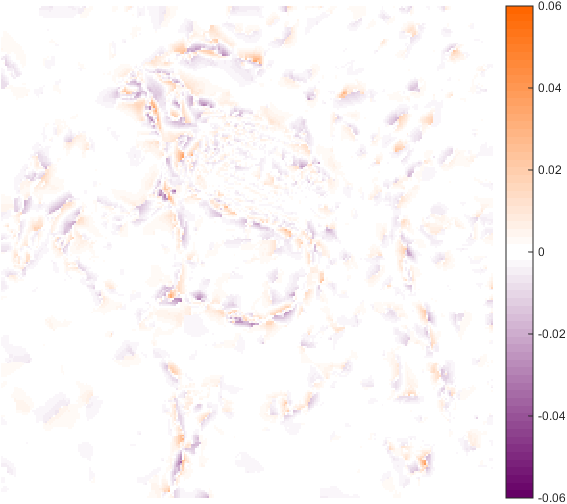}&
&
\includegraphics[height=0.12\textheight]{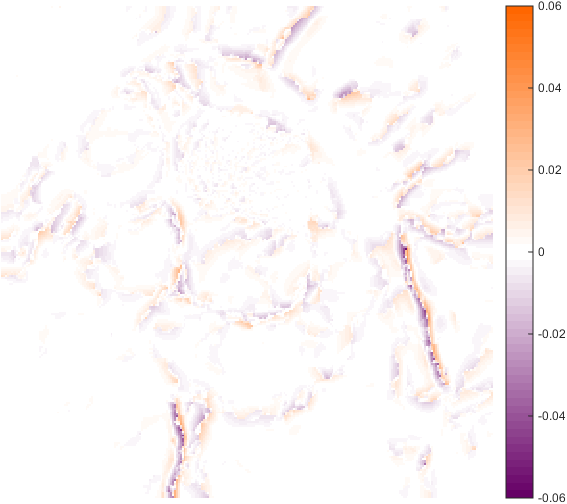}\\
& Interpolated - TGV & & Interpolated - ICTV\\
\\			
\rotatebox[origin=c]{90}{{\bf $\text{curl}(\boldsymbol{w})$}}&
\includegraphics[height=0.12\textheight]{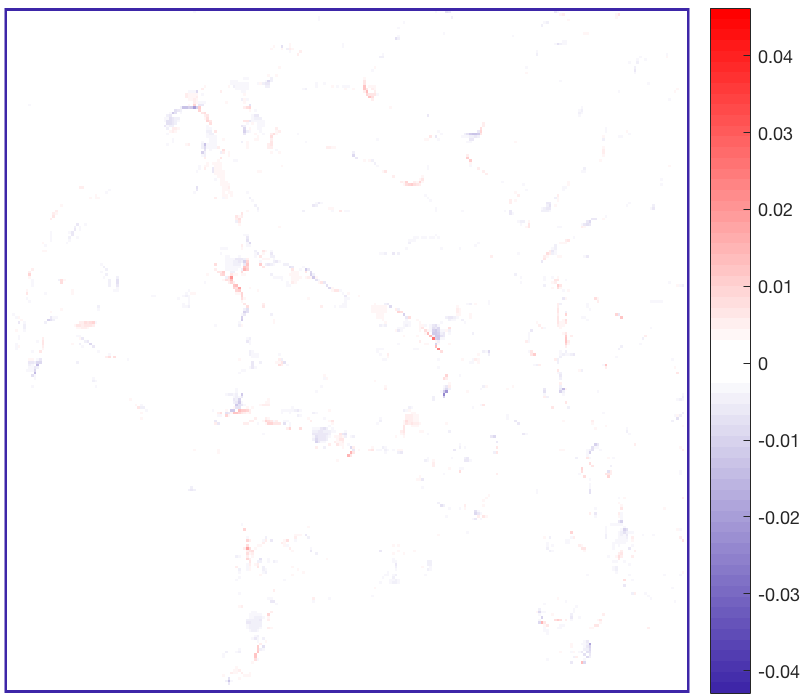}&
\includegraphics[height=0.12\textheight]{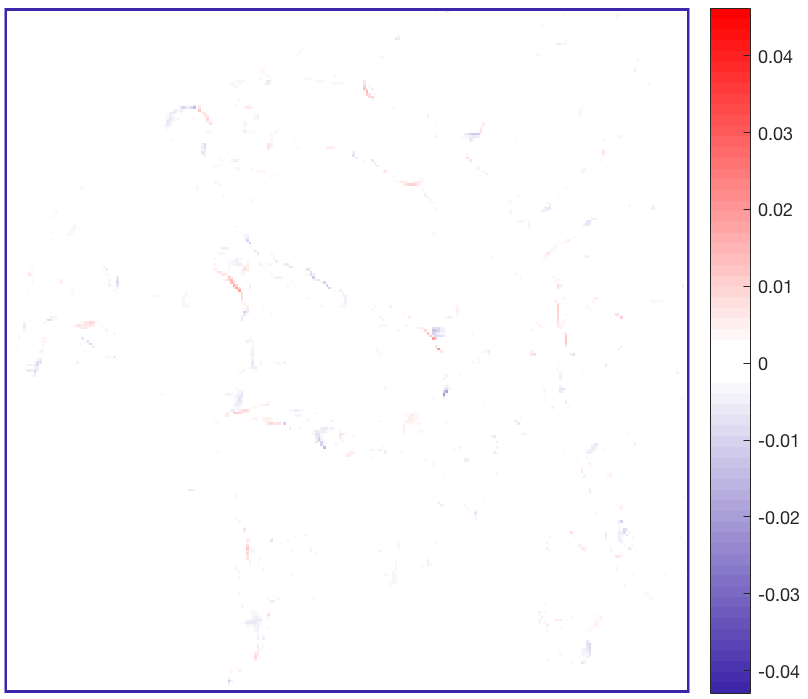}&
\includegraphics[height=0.12\textheight]{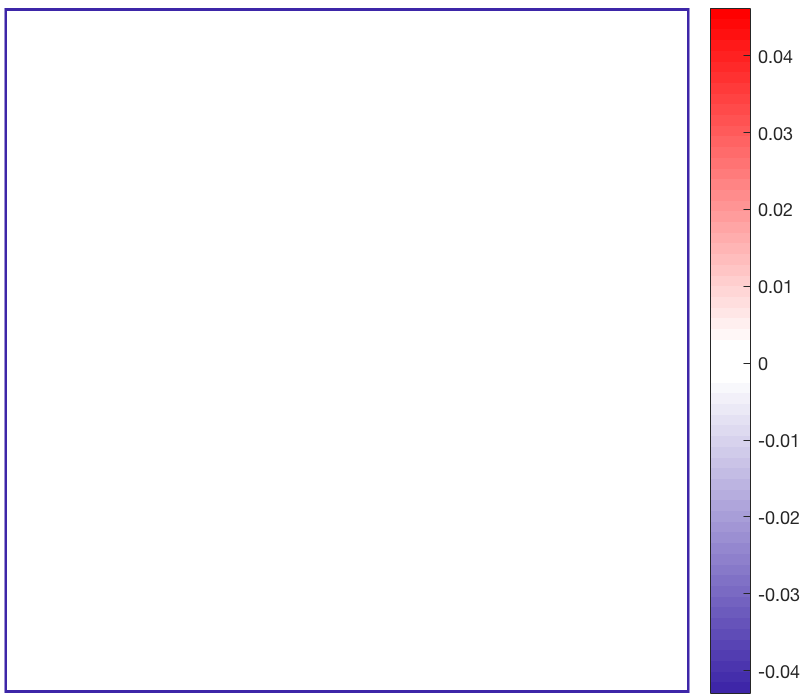}\\
\rotatebox[origin=c]{90}{{\bf $\text{div}(\boldsymbol{w})$}}&
\includegraphics[height=0.12\textheight]{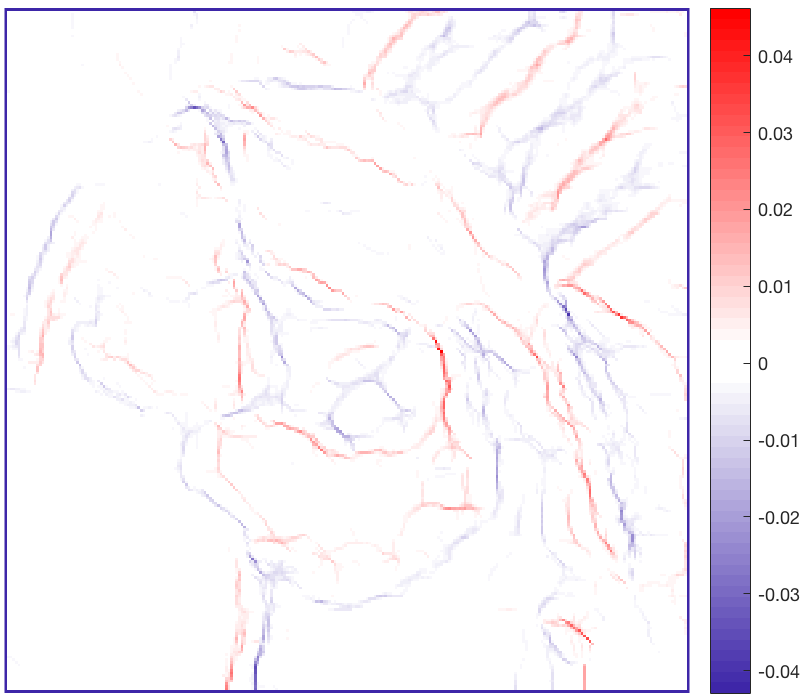}&
\includegraphics[height=0.12\textheight]{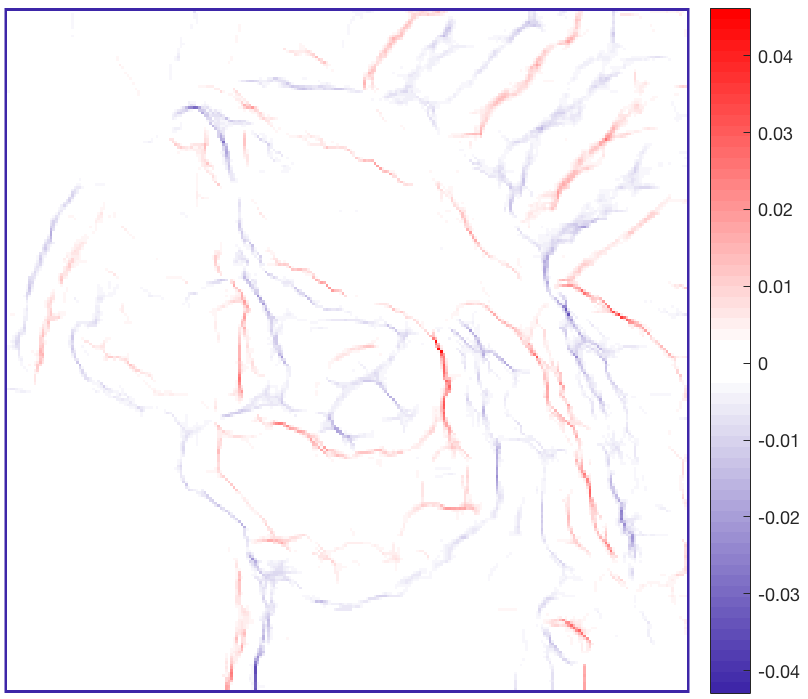}&
\includegraphics[height=0.12\textheight]{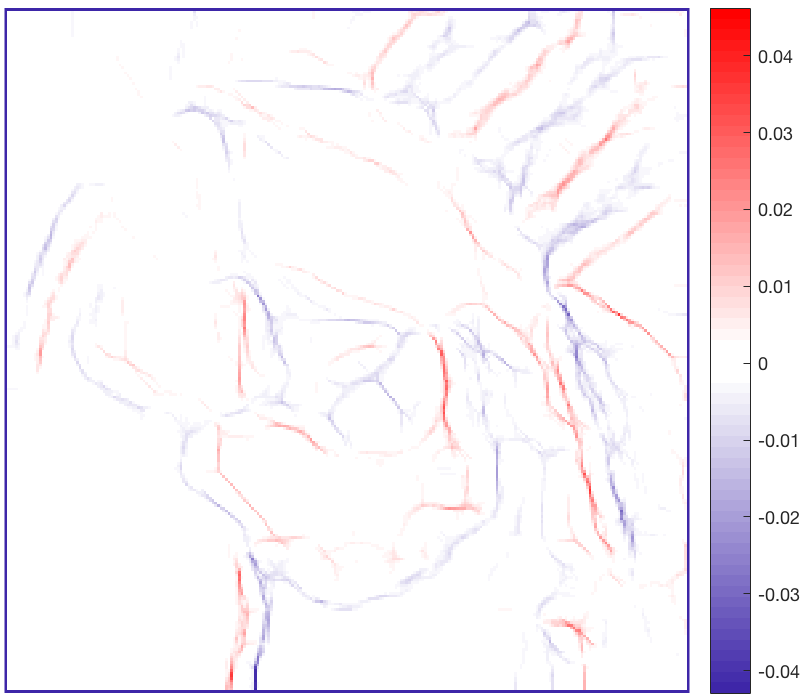}\\
\rotatebox[origin=c]{90}{{\bf $\text{sh}_1(\boldsymbol{w})$}}&
\includegraphics[height=0.12\textheight]{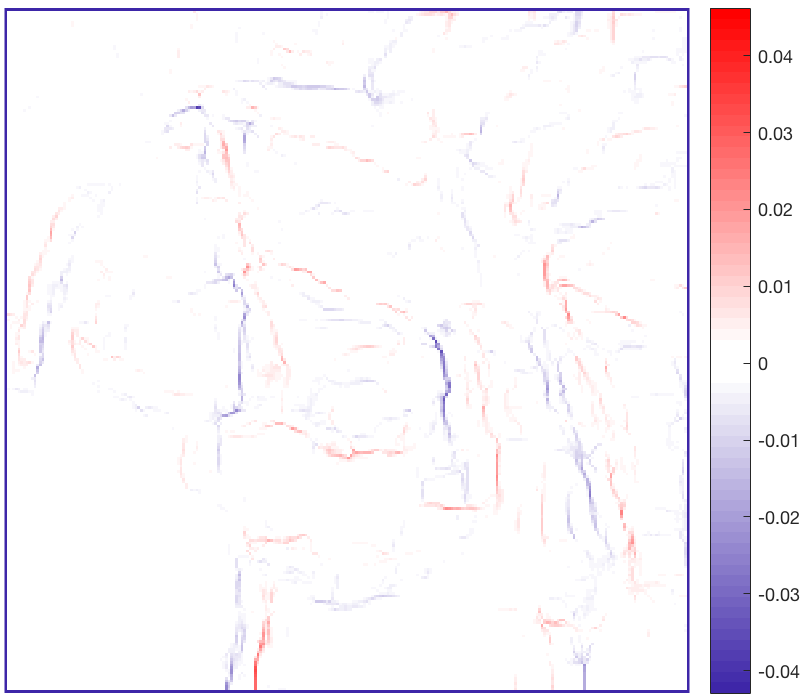}&
\includegraphics[height=0.12\textheight]{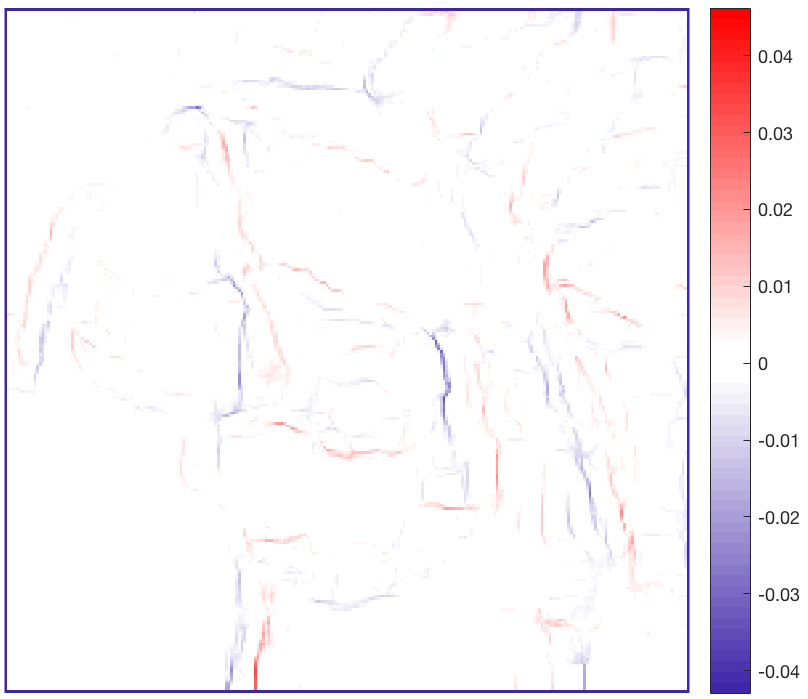}&
\includegraphics[height=0.12\textheight]{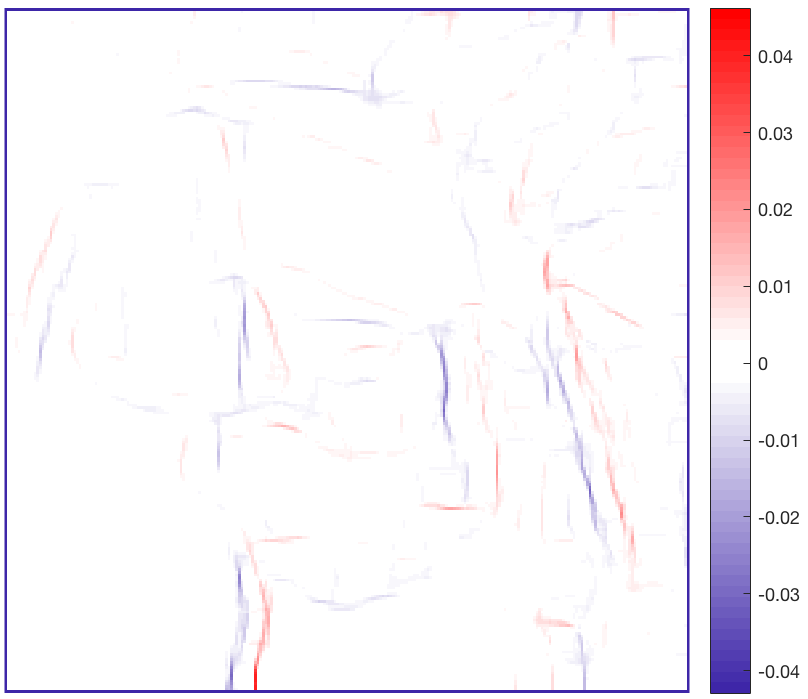}\\
\rotatebox[origin=c]{90}{{\bf $\text{sh}_2(\boldsymbol{w})$}}&
\includegraphics[height=0.12\textheight]{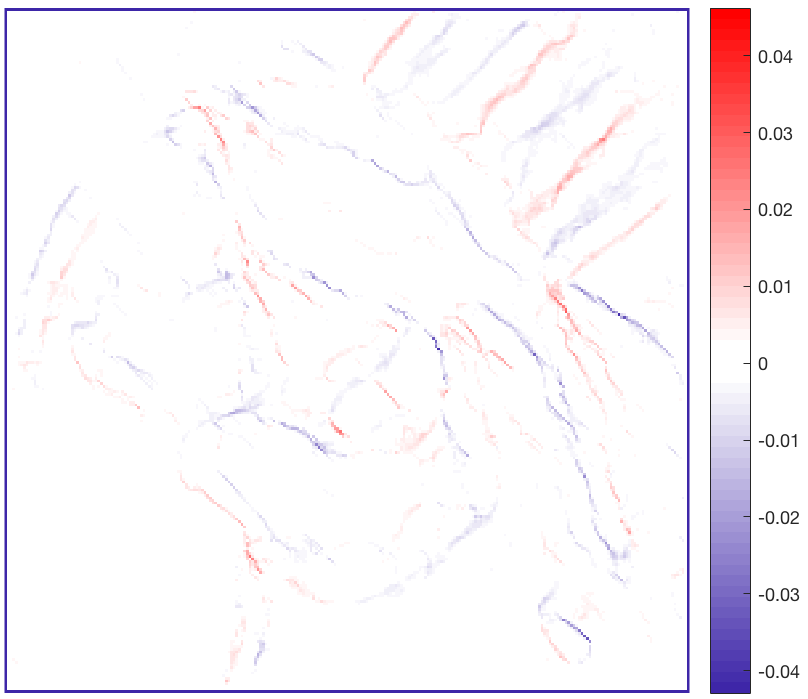}&
\includegraphics[height=0.12\textheight]{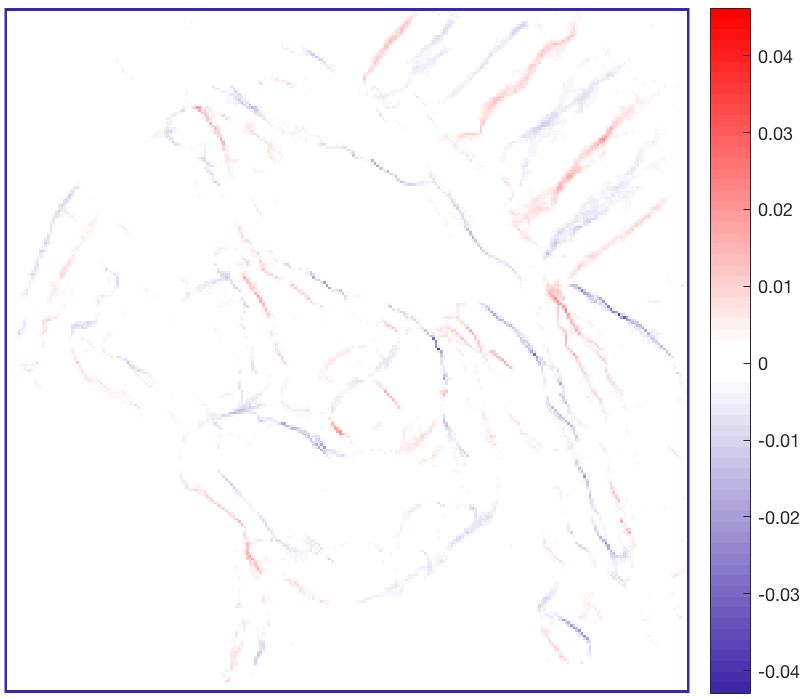}&
\includegraphics[height=0.12\textheight]{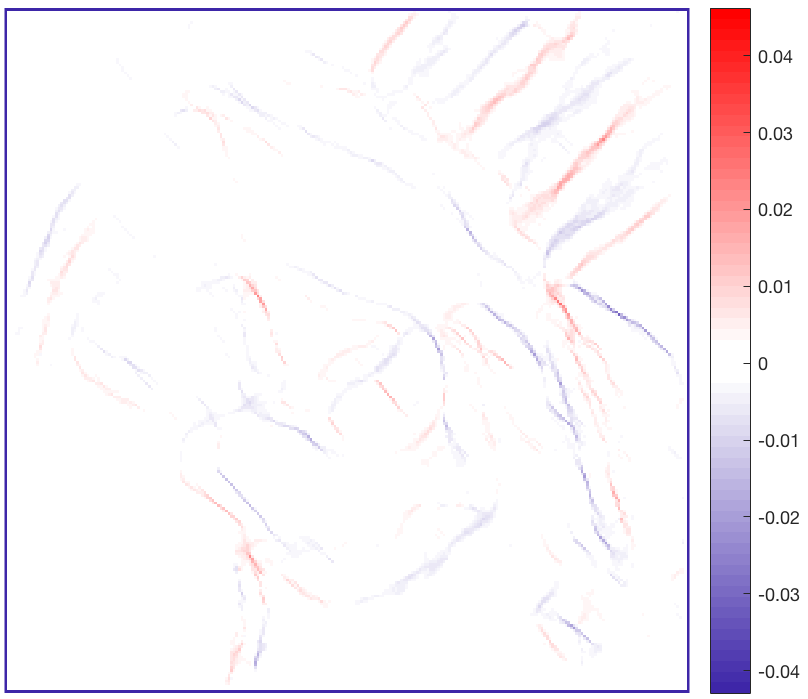}
\end{tabular}
\caption{Reconstruction of parrot test image adding Gaussian noise with zero mean and variance $\sigma^2 = 0.05$ using \eqref{eq:UnifiedModel} demonstrating the ability to interpolate between \eqref{eq:TGV} and \eqref{eq:ICTV}. Top three rows: denoised images $u$, error image showing the difference to the ground truth, difference image to the interpolated result. Lower four rows: different differential operators applied to vector field $w$.}
\label{fig:interpolation}
\end{figure}

\subsubsection*{CEP - Model}

The CEP model   \eqref{eq:CEP} can be rewritten in our notation as
\begin{align*}
\text{CEP}(u) = \inf_{\substack{w \in \mathcal{M}(\Omega,\mathbb{R}^2)\\ \mycurl(w) = 0}} \Vert \nabla u - w \Vert_{\mathcal{M}(\Omega,\mathbb{R}^2)} + \Vert \mydiv (w) \Vert_{\mathcal{M}(\Omega)}.
\end{align*}
It is apparent in this case to choose ${\beta}^t:=(t,1,0,0)$ and to again consider the limit $t \rightarrow \infty$ to recover CEP as a limit of $R_{{\bm \beta}^t}$. However, here we are in a situation where more than one of the parameters $\beta_i$ vanishes, thus we cannot guarantee the existence of a minimiser for such a model and consequently we cannot perform a rigorous analysis of the limit in $BV(\Omega)$. In the denoising case \eqref{eq:UnifiedModel} one could still perform a convergence analysis for the functional including the data term with respect to weak $L^2$ convergence, which is however not in the scope of our approach. 

From the issues in the analysis and our previous discussion of artefacts when only using $\mydiv$ and $\mycurl$ in the regularisation functional it is also to be expected that the CEP model produces some kind of point artefacts. Indeed those can be seen by close inspection of the results in \cite{CEP}, in particular Figure 4.

\subsubsection*{TV and Variants}
 
We finally verify the relation of our model to the original total variation, which is of course to be expected as the parameters $\beta_i$ converge to infinity. This is made precise by the following theorem, from which we see the $\Gamma$-convergence except on the finite-dimensional nullspace. The proof is analogous to Theorem \ref{ictvthm} and omitted here.
 
\begin{thm}
Let $\beta_i \geq 0$ for $i=1,\ldots,4$ and let at most one of them vanish. Set ${\bf B}=\text{diag}(\sqrt{\beta_1},\ldots,\sqrt{\beta_4})$ and ${\bm \beta}^t = t {\bm \beta}$, ${\bf B}^t = t {\bf B}$. Then $R_{{\bm \beta}^t}$ $\Gamma$-converges to $\textup{TV}_{\textup{B}}$ strongly in $L^p(\Omega)$ for any $p < 2$ as $t\rightarrow \infty$, where
\begin{equation*}
\textup{TV}_{\textup{B}}= \inf_{\overline{w}, B \nabla_N \overline{w} = 0} \Vert \nabla u - \overline{w} \Vert_{\mathcal{M}(\Omega,\mathbb{R}^2)}
\end{equation*}
and we extend both functionals by infinity on $L^p(\Omega) \setminus BV(\Omega)$. 
\end{thm}

\subsection{Rotational Invariance}
\label{Subsec:rotational invariance} 
At the end of this section we show that by imposing a simple condition on the choice of the weighting parameters $\beta_1,\dots,\beta_4$ we can control the rotational invariance of the regulariser $R_{\bm \beta}(u)$.
\begin{thm}\label{rotinvthm}
Let $\beta_i \geq 0$ for $i = 1,\dots,4$ and let $\beta_3 = \beta_4$. Then the regulariser $R_{\bm \beta}(u)$ is rotationally invariant, i.e., for an orthonormal rotation matrix $\bm{Q} \in \mathbb{R}^{2\times2}$ with
\begin{equation*}
\bm{Q}(\theta) = \begin{pmatrix}
\cos(\theta) & -\sin(\theta)\\
\sin(\theta) & \cos(\theta)
\end{pmatrix} \quad \text{ for } \theta \in \left[0,2\pi\right)
\end{equation*}
and for $u \in BV(\Omega)$ it holds that $\check{u} \in BV(\Omega)$, where $\check{u}=u\circ \bm{Q}$, i.e. $\check{u}(x) = u(\bm{Q}x)$ for a.e. $x \in \Omega$, and
\begin{equation*}
R_{\bm \beta}(\check{u}) = R_{\bm \beta}(u).
\end{equation*}
\end{thm}
\begin{proof}
See Appendix \ref{app:proof of Theorem 4}.

\end{proof}


\section{Discretisation}
\label{sec:discretisation}
 
We devote a separate section of this paper to the discretisation of our novel approach and the contained natural vector operators as a basis for any numerical implementation. 
This seemed necessary, since we aim at complying not only with the standard requirement that it should hold that
\begin{equation}
\nabla^* (u) = - \text{div} (u),
\tag{adjG}
\label{eq:adjoint_structure_of_grad_div}
\end{equation}
but also with natural conservation laws such as
\begin{equation}
\text{curl} \left( \nabla u \right) = 0 \quad \text{and} \quad \text{div} \left( \text{curl}^* \left( u \right) \right) = 0
\tag{conservLaws}
\label{eq:conservation_laws}
\end{equation}
imposing additional constraints upon the choice of discretisation. 
We will use the finite differences-based discretisation proposed in the context of the congeneric second-order TGV-model \cite{TGV} as a starting point for our considerations in this section. However, as we shall see, the approach taken there fails to fulfil the aforementioned conservation laws \eqref{eq:conservation_laws}. 
As a consequence, we suggest a similar, yet in several places adjusted and thus different discretisation, which all numerical results of our unified model \eqref{eq:VOS} presented in this paper are based upon. 
We will compare solutions of the TGV-model \eqref{eq:TGV} obtained by means of the discretisation we suggest with images resulting from the discretisation proposed in \cite{TGV}.
Eventually, we will comment on chances and challenges of other discretisation strategies using staggered grids or finite element methods in the context of our unified model. 

Abusing notation, we denote the involved functions and operators in the same way as in the continuous setting before, but from now on, we are thereby referring to their discretised versions. For the sake of simplicity, we assume the normalised images to be quadratic, i.e.\ $f,u \in [0,1]^{N \times N}$. Then we discretise the image domain by a two-dimensional regular Cartesian grid of size $N \times N$, i.e.
$\Omega = \lbrace (ih,jh): 1 \leq i,j \leq N\rbrace$,
where $h$ denotes the spacing size and $(i,j)$ denote the discrete pixels at location $(ih,jh)$ in the image domain.
Similarly as in \cite{TGV} and as it is fairly customary in image processing, we use forward differences with Neumann boundary conditions to discretise the gradient $(\nabla)_{i,j}: \mathbb{R} \to \mathbb{R}^2$ of a scalar-valued function $u$, i.e.
\begin{equation}
(\nabla u)_{i,j} = \begin{pmatrix}(\nabla u)^1_{i,j}\\ (\nabla u)^2_{i,j}\end{pmatrix} = \begin{pmatrix}(\delta_{x+} u)_{i,j} \\ (\delta_{y+} u)_{i,j}\end{pmatrix},
\tag{discreteG}
\label{eq:discrete_gradient}
\end{equation}
where
\begin{align}
\begin{split}
	(\delta_{x+} u)_{i,j} = u_{i+1,j} - u_{i,j},\\
	(\delta_{y+} u)_{i,j} = u_{i,j+1} - u_{i,j},
\end{split}
\tag{forwDiff}
\label{eq:forward_differences}
\end{align}
and where we extend the definition by zero if $i=N$ respectively $j=N$.
To avoid asymmetries and to preserve the adjoint structure \eqref{eq:adjoint_structure_of_grad_div}, we discretise the first-order divergence operator $(\text{div})_{i,j}:\mathbb{R}^2 \to \mathbb{R}$ of a two-dimensional vector field $w_{i,j} = (w_{i,j}^1,w_{i,j}^2)^{\mathrm{T}}$ using backward finite differences with homogeneous Dirichlet boundary conditions, i.e.
\begin{equation}
(\text{div} (w))_{i,j} = (\delta_{x-} w^1)_{i,j} + (\delta_{y-} w^2)_{i,j},
\tag{discreteDiv}
\label{eq:discrete_divergence}
\end{equation}
where 
\begin{align}
(\delta_{x-} w^1)_{i,j} &=
\begin{cases}
w^1_{i,j} - w^1_{i-1,j}, & \text{if}\ 1 < i < N,\\
w^1_{i,j}, & \text{if}\ i = 1,\\
-w^1_{i-1,j}, & \text{if}\ i = N,
\end{cases}\notag \\
(\delta_{y-} w^2)_{i,j} &=
\begin{cases}
w^2_{i,j} - w^2_{i,j-1}, & \text{if}\ 1 < j < N,\\
w^2_{i,j}, & \text{if}\ j = 1,\\
-w^2_{i,j-1}, & \text{if}\ j = N.
\end{cases}
\tag{backwDiff}
\label{eq:backward_differences}
\end{align}

In \cite{TGV} the authors moreover proposed to recursively apply forward and backward differences to the divergence operator of higher order such that the outermost divergence operator is based on backward differences with homogeneous Dirichlet boundary conditions. For the second-order divergence operator $(\text{div}^2)_{i,j}: \mathbb{R}^{2 \times 2} \to \mathbb{R}$ of a symmetric $2 \times 2$-matrix $(g)_{i,j}$ at every pixel location $(i,j)$ (cf.\ \eqref{eq:first_and_second_order_divergence}) this means:
\begin{align*}
(\text{div}^2(g))_{i,j} & \; = (\delta_{x-}\delta_{x+} g_{11})_{i,j} + (\delta_{y-}\delta_{y+} g_{22})_{i,j} +  (\delta_{x-}\delta_{y+} g_{12})_{i,j} + (\delta_{y-}\delta_{x+} g_{21})_{i,j}\\
& \; = (\delta_{x-}\delta_{x+} g_{11})_{i,j} + (\delta_{y-}\delta_{y+} g_{22})_{i,j} +  ((\delta_{x-}\delta_{y+} + \delta_{y-}\delta_{x+}) g_{12})_{i,j}.
\end{align*}
\begin{sloppypar}
Further following the reasoning in \cite{TGV}, the discrete second-order derivative and discrete second-order divergence should also satisfy an adjointness condition. 
Consequently, we calculate the adjoint of $\text{div}^2$ in order to obtain the Hessian matrix of a scalar-valued function $u$. Symmetrisation of the Hessian then yields the following discretisation of the symmetrised second-order derivative $(\mathcal{E}^2)_{i,j} : \mathbb{R} \to \mathbb{R}^{2 \times 2}$:
\begin{align*}
(\mathcal{E}^2(u))_{i,j} = (\mathcal{E}(\nabla u))_{i,j} = \begin{pmatrix} (\delta_{x-} \delta_{x+} u)_{i,j} & \frac{\left( (\delta_{y-}\delta_{x+} + \delta_{x-}\delta_{y+}) u \right)_{i,j}}{2} \\ \frac{\left( (\delta_{x-}\delta_{y+} + \delta_{y-}\delta_{x+}) u \right)_{i,j}}{2} & (\delta_{y-} \delta_{y+} u)_{i,j} \end{pmatrix}, 
\end{align*}
where for the first equality we used that since $(\nabla u)_{i,j}$ is a (1,0)-tensor, or in other words a vector, the symmetrised gradient $\mathcal{E}$ of $u$ is just equal to the gradient.
To stay consistent, the symmetrised derivative $(\mathcal{E})_{i,j}:\mathbb{R}^2 \to \mathbb{R}^{2 \times 2}$ of a two-dimensional vector field $w_{i,j} = (w_{i,j}^1,w_{i,j}^2)^{\mathrm{T}}$ should thus be discretised in the following way:
\begin{align*}
(\mathcal{E}(w))_{i,j} = \begin{pmatrix} (\delta_{x-} w^1)_{i,j} & \frac{\left( \delta_{y-} w^1 + \delta_{x-} w^2\right)_{i,j}}{2} \\ \frac{ \left( \delta_{x-} w^1 + \delta_{y-} w^2 \right)_{i,j}}{2} & (\delta_{y-} w^2)_{i,j} \end{pmatrix}, 
\end{align*}
where $(\delta_{x-} w^2)_{i,j}$ and $(\delta_{y-} w^1)_{i,j}$ are defined analogously to \eqref{eq:backward_differences} with $w^1$ and $w^2$ being interchanged. 
We have thus recalled the choice of discretisation of the second-order divergence and hence of the symmetrised derivative as suggested in \cite{TGV}. 
\end{sloppypar}

\begin{sloppypar}
With regard to Section \ref{Subsec:special cases} we conclude that in this setting the most natural discretisations of the curl operator $(\text{curl})_{i,j}: \mathbb{R}^2 \to \mathbb{R}$ of a two-dimensional vector field $w_{i,j} = (w_{i,j}^1,w_{i,j}^2)^{\mathrm{T}}$ as well as of the two components of the shear $(\text{sh}_1)_{i,j}: \mathbb{R}^2 \to \mathbb{R}$ and $(\text{sh}_2)_{i,j}: \mathbb{R}^2 \to \mathbb{R}$ would all be based on backward finite differences with homogeneous Dirichlet boundary conditions, i.e.
\begin{alignat*}{3}
&(\text{curl} (w))_{i,j} &&= (\delta_{x-} w^2)_{i,j} - (\delta_{y-} w^1)_{i,j},\\
&(\text{sh}_1 (w))_{i,j} &&= (\delta_{y-} w^2)_{i,j} - (\delta_{x-} w^1)_{i,j},\\
&(\text{sh}_2 (w))_{i,j} &&= (\delta_{y-} w^1)_{i,j} + (\delta_{x-} w^2)_{i,j}.
\end{alignat*}
However, this discretisation of the curl operator fails to comply with the conservation laws given in \eqref{eq:conservation_laws}, since  
\begin{alignat*}{3}
&(\text{curl}(\nabla u))_{i,j} &&= (\delta_{x-} \delta_{y+} u )_{i,j} - (\delta_{y-} \delta_{x+} u)_{i,j},\\
&(\text{div}(\text{curl}^* (u)))_{i,j} &&= (\delta_{x-} \delta_{y+} u )_{i,j} - (\delta_{y-} \delta_{x+} u)_{i,j}
\end{alignat*}
can in general each be non-zero. To resolve this issue, we decided to instead discretise the curl operator 
$(\text{curl})_{i,j}: \mathbb{R}^2 \to \mathbb{R}$ of a two-dimensional vector field $w_{i,j} = (w_{i,j}^1,w_{i,j}^2)^{\mathrm{T}}$ with forward finite differences, i.e.
\begin{equation}
(\text{curl} (w))_{i,j} = (\delta_{x+} w^2)_{i,j} - (\delta_{y+} w^1)_{i,j}.
\tag{discreteCurl}
\label{eq:discrete_curl}
\end{equation}
In order to meet the theory derived for the continuous setting in Section \ref{Subsec:special cases} the discretisation of the curl operator by forward finite differences in combination with the discretisation of the divergence operator by backward finite differences requires that the first component of the shear $(\text{sh}_1)_{i,j}: \mathbb{R}^2 \to \mathbb{R}$ shall be discretised using backward finite differences while the second component $(\text{sh}_2)_{i,j}: \mathbb{R}^2 \to \mathbb{R}$ shall be discretised by means of forward finite differences, i.e.
\begin{alignat}{3}
&(\text{sh}_1 (w))_{i,j} &&= (\delta_{y-} w^2)_{i,j} - (\delta_{x-} w^1)_{i,j}, \tag{discreteSh1}\\
&(\text{sh}_2 (w))_{i,j} &&= (\delta_{y+} w^1)_{i,j} + (\delta_{x+} w^2)_{i,j} \tag{discreteSh2}.
\end{alignat}
As a side benefit of this choice of discretisation we additionally obtain the identities 
\begin{equation}
\text{sh}_1 \left( \text{sh}_2^* \left( u \right) \right) = 0 \quad \text{and} \quad \text{sh}_2 \left( \text{sh}_1^* \left( u \right) \right) = 0.
\tag{conservLaws2}
\label{eq:convervation_laws_2}
\end{equation}
Vice versa, this approch leads to the following discretisation of the symmetrised derivative $(\mathcal{E})_{i,j}:\mathbb{R}^2 \to \mathbb{R}^{2 \times 2}$ of a  vector field $w_{i,j} = (w_{i,j}^1,w_{i,j}^2)^{\mathrm{T}}$:
\begin{align}
(\mathcal{E}(w))_{i,j} = \begin{pmatrix} (\delta_{x-} w^1)_{i,j} & \frac{ \left( \delta_{y+} w^1 + \delta_{x+} w^2 \right)_{i,j}}{2} \\ \frac{ \left( \delta_{x+ }w^1 + \delta_{y+} w^2 \right)_{i,j}}{2} & (\delta_{y-} w^2)_{i,j} \end{pmatrix},
\tag{discreteSymG}
\label{eq:discrete_symmetrised_gradient}
\end{align}
that is we discretise the mixed derivatives differently than proposed in \cite{TGV}. 
Further following the line of argument brought forward in this section, the corresponding discrete second-order divergence operator $(\text{div}^2)_{i,j}: \mathbb{R}^{2 \times 2} \to \mathbb{R}$ of a symmetric $2 \times 2$-matrix $(g)_{i,j}$ at every pixel location $(i,j)$ (cf.\ \eqref{eq:first_and_second_order_divergence}) would be given by:
\begin{align*}
(\text{div}^2(g))_{i,j} = (\delta_{x-}\delta_{x+} g_{11})_{i,j} + (\delta_{y-}\delta_{y+} g_{22})_{i,j} +  ((\delta_{x+}\delta_{y+} + \delta_{y+}\delta_{x+}) g_{12})_{i,j}.
\end{align*}
Paraphrasing this discretisation, one could say that with respect to the pure second partial derivatives, i.e.\ the diagonal entries of the Hessian, we stick to the idea of recursively applying forward and backward differences as proposed by Bredies and coworkers \cite{TGV}, while in regard to the mixed partial derivatives we repeatedly use forward differences. 
Being aware that this discretisation of the second-order divergence might seem a little less intuitive than the one proposed in \cite{TGV}, we nevertheless decided to adhere to the discretisation that we introduced in this section. 
This is because in the context of our unified model it seems crucial to find a discretisation that preserves the nullspaces of the continuous model and complies with natural conservation laws such that for example choosing $\beta_1 > 0$ and $\beta_2 = \beta_3 = \beta_4 = 0$ indeed returns the noisy image $f$ as predicted by the theory for the continuous model. 
\end{sloppypar}

\begin{sloppypar}
To compare the effect of the two different discretisation schemes on the reconstructed images, we corrupted a test image from the Mc Master Dataset \cite{McMDataset} by Gaussian noise of mean 0 and variance 0.05 and calculated the denoising results obtained by means of the TGV$^2$ model \eqref{eq:TGV} with both discretisation approaches discussed in this section so far, the one proposed by Bredies and coworkers in \cite{TGV} as well as our alternative satisfying the natural conservation laws. 
\begin{figure}[t]
\centering
\setlength{\tabcolsep}{1mm}
\begin{tabular}{B{0.25cm}B{0.25\textwidth}B{0.25\textwidth}B{0.25\textwidth}}
& {\bf Proposed discretisation} & {\bf Discretisation in \cite{TGV}} & {\it {\bf Difference image}} \\
& & & \\
\rotatebox[origin=c]{90}{{\bf $500 \times 500$ pixels}}&
\includegraphics[height=4cm]{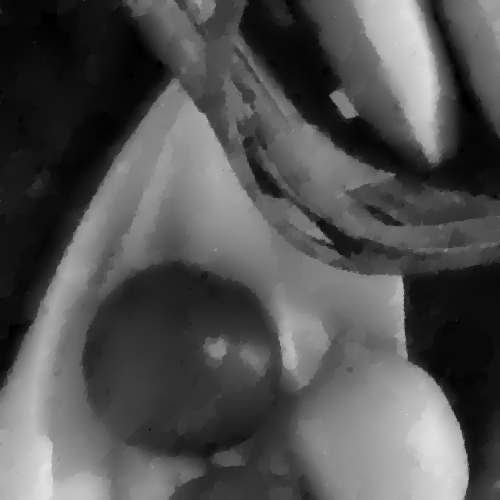}&
\includegraphics[height=4cm]{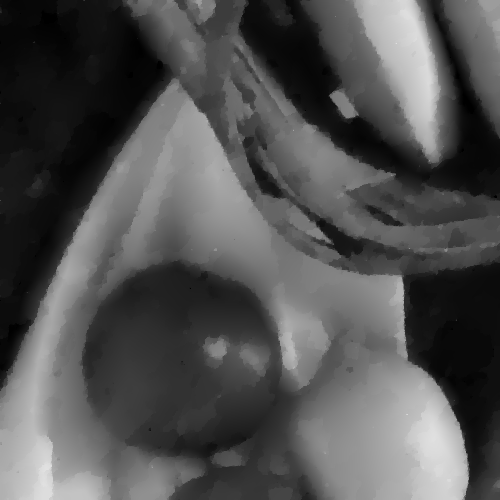}&
\includegraphics[height=4cm]{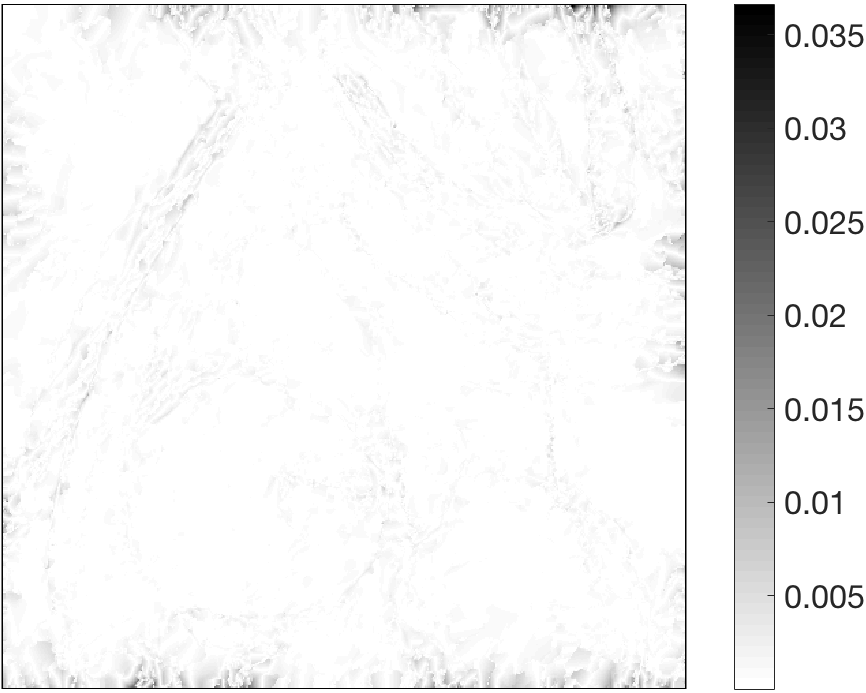}\\
& $\text{SSIM} = 0.8180$ & $\text{SSIM} = 0.8180$ & \\			
\rotatebox[origin=c]{90}{{\bf $250 \times 250$ pixels}}&
\includegraphics[height=4cm]{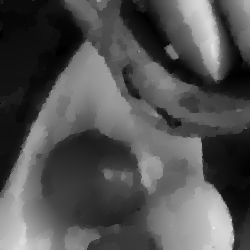}&
\includegraphics[height=4cm]{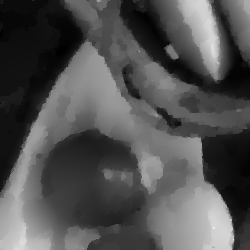}&
\includegraphics[height=4cm]{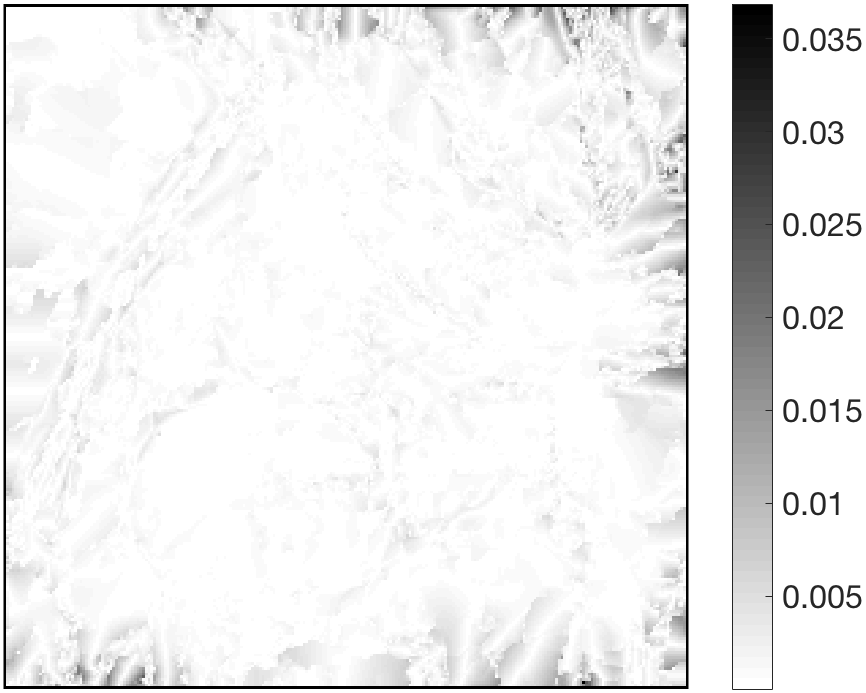}\\
& $\text{SSIM} = 0.7822$ & $\text{SSIM} = 0.7840$ & \\ 
\end{tabular}
\caption{Comparison of our proposed discretisation (reconstructions in the left column) with the one in \cite{TGV} (reconstructions in the middle column) and absolute difference of the two reconstructions (right column) for two different image sizes}
\label{fig:discretisation}
\end{figure}
The outcome of this comparison is shown in Figure \ref{fig:discretisation}.
Looking at the denoised images, we can conclude that both discretisation approaches provide very similar results, since visually there is hardly any difference between the corresponding images detectable. Thus, we included the difference images in the figure to illustrate that the reconstructions based on the two different discretisations are not identical, but indeed differ slightly especially close to some of the edges and near the boundary of the image domain. Also, with respect to the quality measure SSIM the results for both discretisations are in a similar range, however the differences seem to become more significant with decreasing image resolution. This makes sense since the proportion of pixels depicting an edge in relation to the overall number of pixels of the image increases with decreasing resolution and this is where most of the differences due to the different discretisation schemes occur. In light of the bottom row of Figure \ref{fig:discretisation} we can conclude that at a relatively low image resolution our proposed discretisation apparently performs slightly inferior to the one proposed by Bredies and coworkers, however we decided to nevertheless adhere to the proposed discretisation scheme since this way we can guarantee that the conservation laws valid in the continuous setting also apply for the discretised model.
\end{sloppypar}

At the end of this section we shall also briefly comment on alternative discretisation schemes in the context of our unified model \eqref{eq:VOS} that do not rely on finite differences. 
One option for such a discretisation would be based on staggered grids, i.e.\ on two grids, often referred to as primal and dual grid, that are shifted with respect to each other by half a pixel. 
Following for example \cite{staggeredGrids}, one could define a discrete gradient operator of a scalar function mapping from the cell centres of the primal grid to the vertical and horizontal faces (normal to the sides) of the primal grid, which can be identified with the vertical and horizontal edges (tangential to the sides) of the dual grid. 
In this setting one could then also define discrete versions of the natural vector field operators contained in our model: the curl would map from a vector field defined on the edges of the dual grid to a scalar function defined on the cell centres of the dual grid, which can be identified with the nodes of the primal grid. 
The same would apply to the second component of the shear. 
The divergence operator and the first component of the shear on the contrary would map from a vector field defined on the faces of the primal grid to a scalar function defined on the cell centres of the primal grid. 
However, now one had to face the question of how to add up the values of these different natural vector operators of a given vector field, since their codomains do not coincide. 
Of course, one may consider introducing averaging operators such that in the end one obtains values of the respective operators at the same locations \cite{staggeredGrids} or one might try to resolve this issue by defining inner products and norms in a suitable way (cf.\ e.g.\ \cite{staggeredGrids2,staggeredGrids3,staggeredGrids5}), however again it seems less obvious which is the best way to go. Another alternative would be Raviart-Thomas-based finite element methods \cite{RaviartThomasFEM}, where it would be quite straightforward to define the gradient and the divergence operator, however here, too, it would be less clear how to define the curl operator and the two components of the shear in the most natural way. 

Summing up, there seems to be no straightforward solution to the discretisation of our unified model \eqref{eq:VOS} that meets all our demands and we thus, despite the known demerits, decided to stick to the simple discretisation based on forward finite differences introduced earlier in this section.
An extensive investigation of the most natural discretisation in the context of higher-order TV methods and the Hessian taking into account the connection to the natural vector field operators and the related conservation laws is beyond the scope of this paper and left to future research.

\section{Results}
\label{sec:results}

In this section, we report on numerical denoising results obtained for two different greyscale test images: Trui ($257 \times 257$ pixels), cf.\ Figure \ref{fig:SVF_compression}, and the piecewise affine test image considered in Figures \ref{fig:sparseDiffOp} and \ref{fig:reconTest} ($256 \times 256$ pixels). We choose the denoising framework because of its straightforward  implementation and simple interpretability but would like to stress that our novel joint regulariser could be employed in any variational imaging model. First, we compare the best denoising result with respect to the structure similarity (SSIM) index obtained by using our unified model \eqref{eq:VOS} with denoising models using TV, ICTV and second-order TGV regularisers and the same standard $L^2$ data term. In addition, we present results of a large-scale parameter test solving our model \eqref{eq:VOS} and examining how various parameter combinations lead to reconstructions of different quality.

\begin{sloppypar}
In all experiments, we use the first-order primal-dual algorithm by Chambolle and Pock \cite{ChambollePock} for the convex optimisation. Moreover, we make use of both the step size adaptation and the stopping criterion presented in \cite{GoldsteinLiYuanEsserBaraniuk}. In order to solve our model \eqref{eq:VOS}, analogous to the  implementation described in detail in \cite{SVF}, we define
\begin{equation*}
x = 
\begin{pmatrix}
u, w
\end{pmatrix}^T,\quad
y = 
\begin{pmatrix}
y_1,
y_2
\end{pmatrix}^T,\quad
K = 
\begin{pmatrix}
\nabla & 0 & 0 & 0 & 0 \\ -I & \text{curl} & \text{div} & \text{sh}_1 & \text{sh}_2
\end{pmatrix}^T,
\end{equation*}
where the image $u$ and the vector field $w$ are defined as above, $y_1$ has the same size as $w$, $y_2$ is a vector with four components, each of which has the same size as $u$, and $I$ denotes the identity matrix. Using this notation we can now write our energy functional as a sum $G(x) + F(Kx)$ according to \cite{ChambollePock} by defining
\begin{equation*}
G(x) = \frac{1}{2} \Vert u - f \Vert_2^2,\quad
F(Kx) = \alpha R_{\bm{\beta}}(u),
\end{equation*}
and apply the modified primal-dual algorithm in \cite{GoldsteinLiYuanEsserBaraniuk}. For the implementation of the TV, ICTV and TGV models, we employ the corresponding standard primal-dual implementations, using the discretisation proposed in the respective papers if applicable.
\end{sloppypar}

\subsection{Comparison of Best VOS Result to State-Of-The-Art Methods}

In the following, we compare the best result of our \eqref{eq:VOS} model employing the discretisation described in Section \ref{sec:discretisation} with denoising results obtained by using TV, ICTV and second-order TGV regularisation. We measure optimality with respect to SSIM.

\begin{figure}[h]
\captionsetup[subfigure]{labelformat=empty}
\centering
\subfloat[Ground truth]{\includegraphics[width=0.25\textwidth]{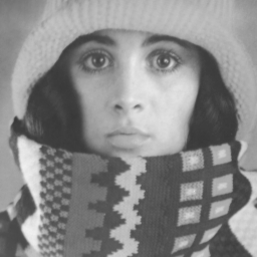}}\hfill
\subfloat[Noisy image]{\includegraphics[width=0.25\textwidth]{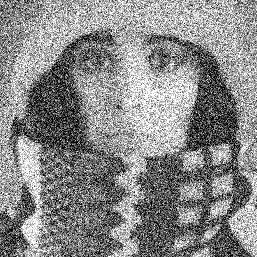}}\hfill
\subfloat[TV denoised]{\includegraphics[width=0.25\textwidth]{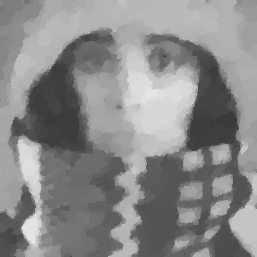}}\\
\subfloat[ICTV denoised]{\includegraphics[width=0.25\textwidth]{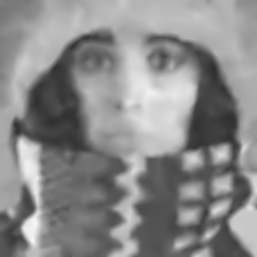}}\hfill
\subfloat[TGV denoised]{\includegraphics[width=0.25\textwidth]{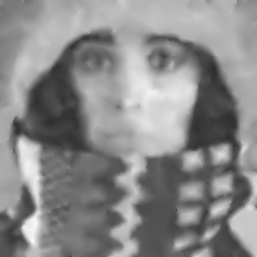}}\hfill
\subfloat[Ours denoised]{\includegraphics[width=0.25\textwidth]{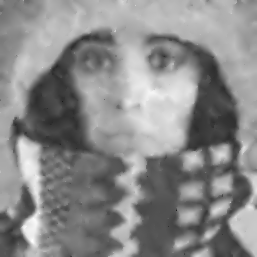}}
\caption{Best results with respect to SSIM for Trui test image}
\label{fig:resultsComparisonTrui05}
\end{figure}

In Figure \ref{fig:resultsComparisonTrui05}, we demonstrate that by using our unified model \eqref{eq:VOS} we are able to obtain a reconstruction of the noisy Trui image superior to TV and comparable to ICTV and second-order TGV with respect to the quality measure SSIM. The task is to reconstruct the image on the top left, which has been corrupted by additive Gaussian noise with zero mean and variance $\sigma^2 = 0.05$ (top centre). We would like to stress that this noise level is relatively high compared with most publications on denoising but we chose it in order to better highlight the visual differences in the reconstructions. In the TV-regularised reconstruction (top right), we choose $\alpha = \frac{1}{4}$ and obtain an SSIM value of 0.7995. In the ICTV case (bottom left), we select $\alpha_1 = \frac{1}{2}$ and $\alpha_0 = \frac{1}{4}$, where SSIM = 0.8121. For the second-order TGV-type reconstruction, we set $\alpha_1 = \alpha_0 = \frac{1}{4}$. Here, we obtain an SSIM value of 0.8141. For better comparison with the ICTV result and the result of our unified model we mention that the corresponding TGV-result with our discretisation on this image resolution yields an SSIM of 0.8131. The best result using our model is shown on the bottom right, choosing $\alpha = \frac{1}{4.5}, \beta_1 = 0, \beta_2 = \frac{1}{8}, \beta_3 = 1$ and $\beta_4 = \frac{1}{2}$ and achieving an SSIM value of 0.8136. We would like to especially draw attention to the enhanced reconstruction of the chessboard-like pattern in the scarf as well as the regions around the eyes and the mouth by using our model.

\begin{figure}[h]
\captionsetup[subfigure]{labelformat=empty}
\centering
\subfloat[Ground truth]{\includegraphics[width=0.25\textwidth]{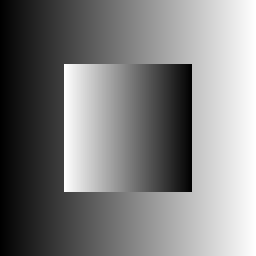}}\hfill
\subfloat[Noisy image]{\includegraphics[width=0.25\textwidth]{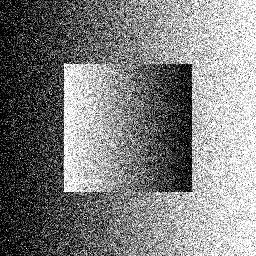}}\hfill
\subfloat[TV denoised]{\includegraphics[width=0.25\textwidth]{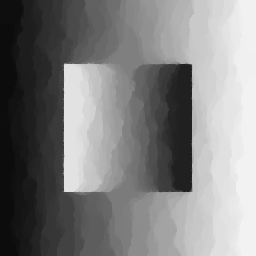}}\\
\subfloat[ICTV denoised]{\includegraphics[width=0.25\textwidth]{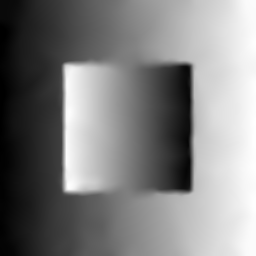}}\hfill
\subfloat[TGV denoised]{\includegraphics[width=0.25\textwidth]{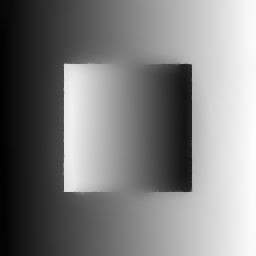}}\hfill
\subfloat[Ours denoised]{\includegraphics[width=0.25\textwidth]{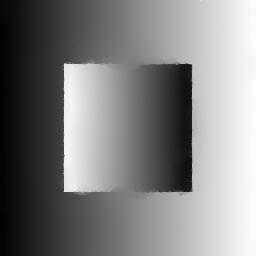}}
\caption{Best results with respect to SSIM for piecewise affine test image}
\label{fig:resultsComparisonTest05}
\end{figure}

Now we present similar results obtained by solving the denoising problem for the piecewise affine square test image in Figure \ref{fig:resultsComparisonTest05}, again considering a noise variance of $\sigma^2 = 0.05$. In the case of TV denoising (top right), we choose $\alpha = \frac{1}{2}$, yielding SSIM = 0.9153. On the bottom left, ICTV regularisation selecting $\alpha_1 = 1$ and $\alpha_0 = \frac{1}{2}$ leads to an SSIM value of 0.9509. The parameters for the second-order TGV reconstruction (bottom centre) are $\alpha_1 = \frac{1}{2}$ and $\alpha_0 = 2$. Here, we obtain an SSIM value of 0.9775. The best result using our model is obtained by setting $\alpha = \frac{1}{3}$, $\beta_1 = 4.5$, $\beta_2 = 90$ and $\beta_3 = \beta_4 = 9$. We achieve the best SSIM index of 0.9844.

\subsection{Practical Study of Parameter Combinations}

\begin{sloppypar}
In order to get a better understanding of our novel regulariser and how zero and non-zero values of the different parameters in our model affect the denoising reconstructions, we set up large-scale parameter tests for both the Trui and the piecewise affine test image. We use the discretisation described in Section \ref{sec:discretisation} for all experiments, solving \eqref{eq:VOS} numerically as described at the beginning of this section. For the Trui image we select $\alpha \in \{ \frac{1}{5}, \frac{1}{4.5}, \frac{1}{4} \}$ and $\beta_i \in \{ 0, \frac{1}{8}, \frac{1}{4}, \frac{1}{2}, 1, 2, 5, 20 \}$, $i = 1,\dots,4$, which leads to 12288 different combinations, and for the piecewise affine test image we choose $\alpha \in \{ \frac{1}{4.5}, \frac{1}{4}, \frac{1}{3.5}, \frac{1}{3} \}$ and $\beta_i = \frac{b}{\alpha^2}$, $b \in \{ 0, \frac{1}{8}, \frac{1}{4}, \frac{1}{2}, 1, 10 \}$, $i = 1,\dots,4$, which leads to 5184 different combinations. We use different parameter sets, as our images differ quite significantly in structure and we naturally need a stronger overall regularisation for the less textured and more homogeneous piecewise affine test image. Again, we consider the denoising problem explained above and corrupt the original image by additive Gaussian noise with zero mean and variance $\sigma^2 = 0.05$.
\end{sloppypar}

\subsubsection*{Trui Test Image}

\begin{figure}[h]
\captionsetup[subfigure]{labelformat=empty}
\centering
\subfloat[SSIM]{\includegraphics[width=0.32\textwidth]{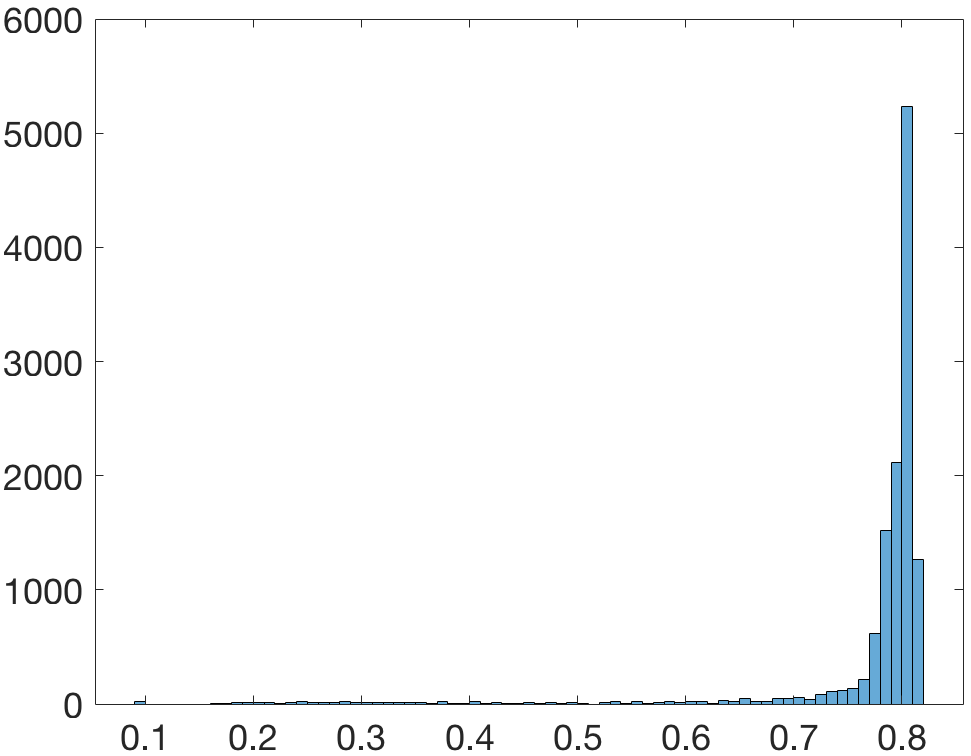}}\hfill
\subfloat[PSNR]{\includegraphics[width=0.32\textwidth]{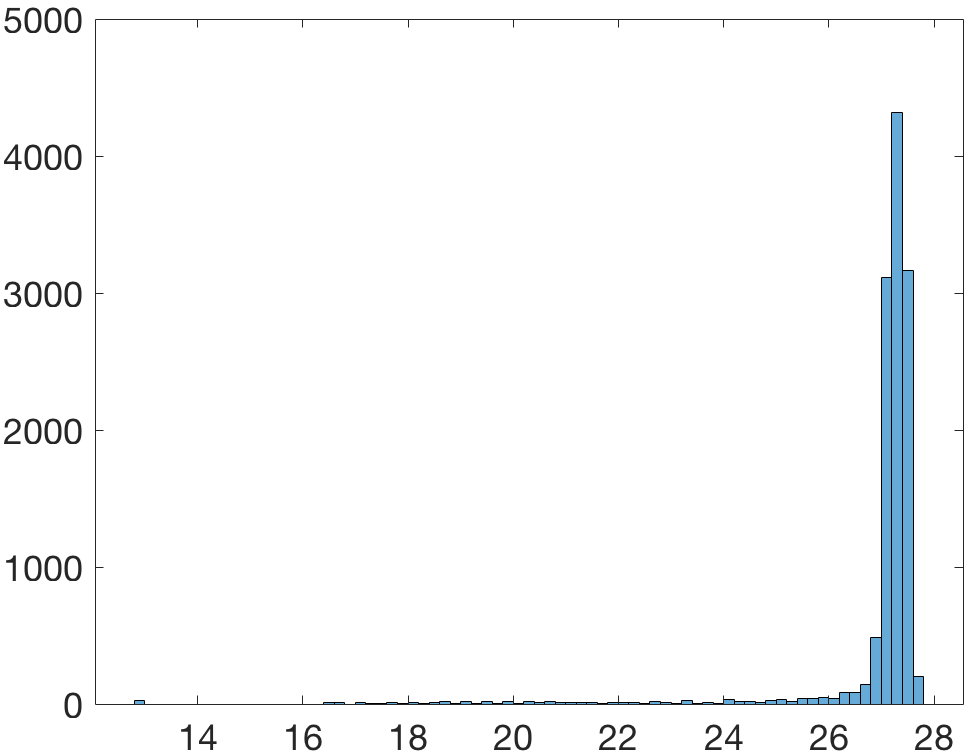}}\hfill
\subfloat[Relative Error]{\includegraphics[width=0.32\textwidth]{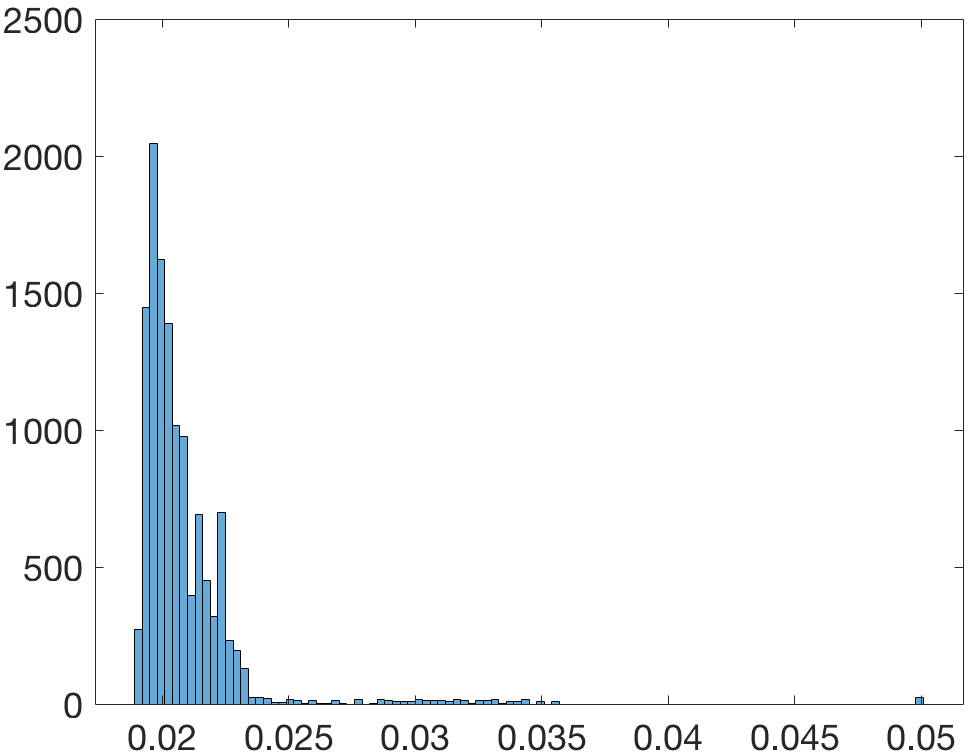}}
\caption{Histograms for Trui considering all tested parameter combinations}
\label{fig:histogramsTrui}
\end{figure}

Figure \ref{fig:histogramsTrui} shows histograms for three quality measures we calculated for all obtained reconstructions of Trui in our parameter test: SSIM, PSNR and relative error. It can be immediately observed that in the majority of cases, we get competitive values.

\begin{figure}[h]
\captionsetup[subfigure]{labelformat=empty}
\centering
\subfloat[SSIM]{\includegraphics[width=0.32\textwidth]{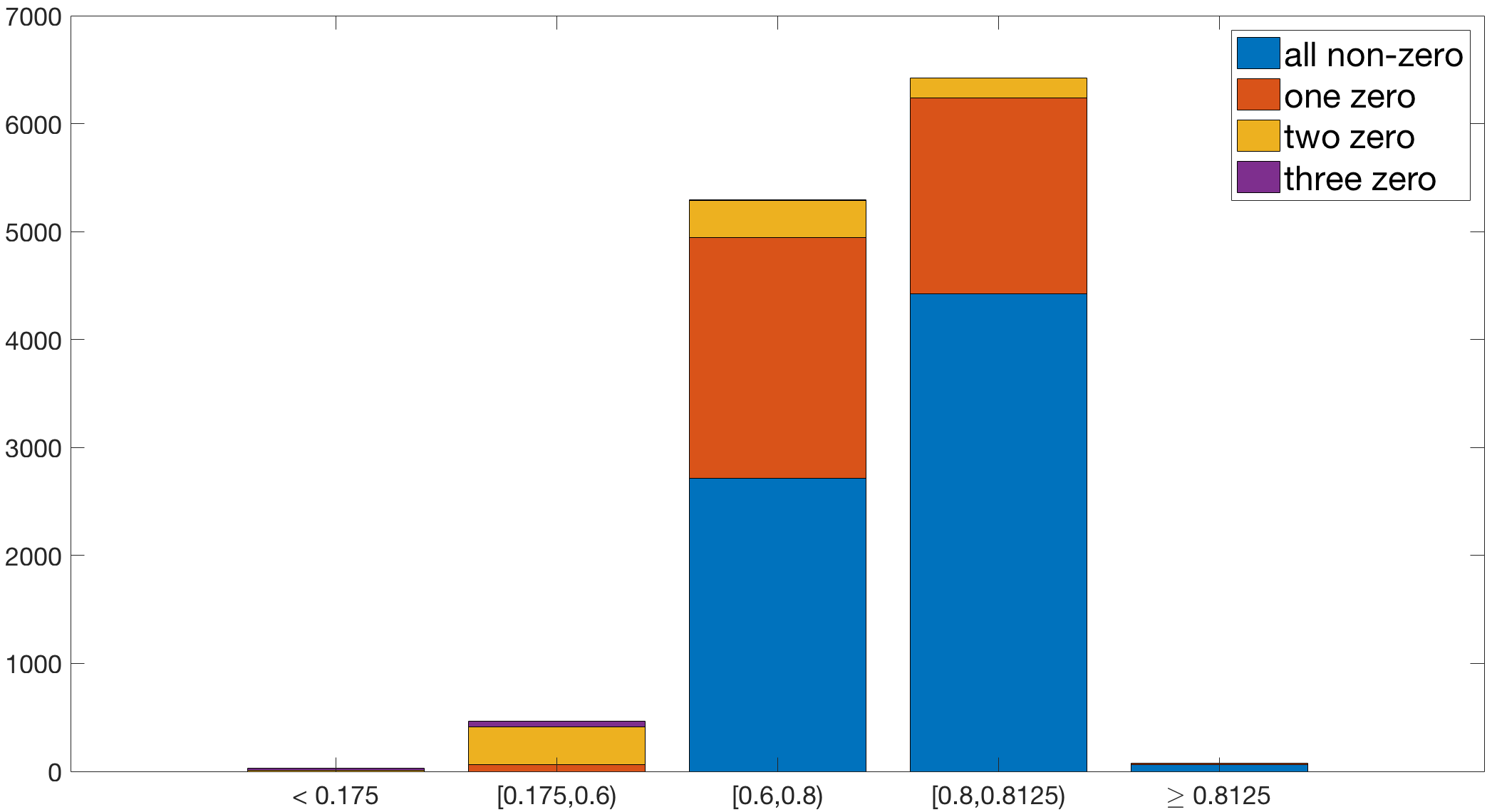}}\hfill
\subfloat[PSNR]{\includegraphics[width=0.32\textwidth]{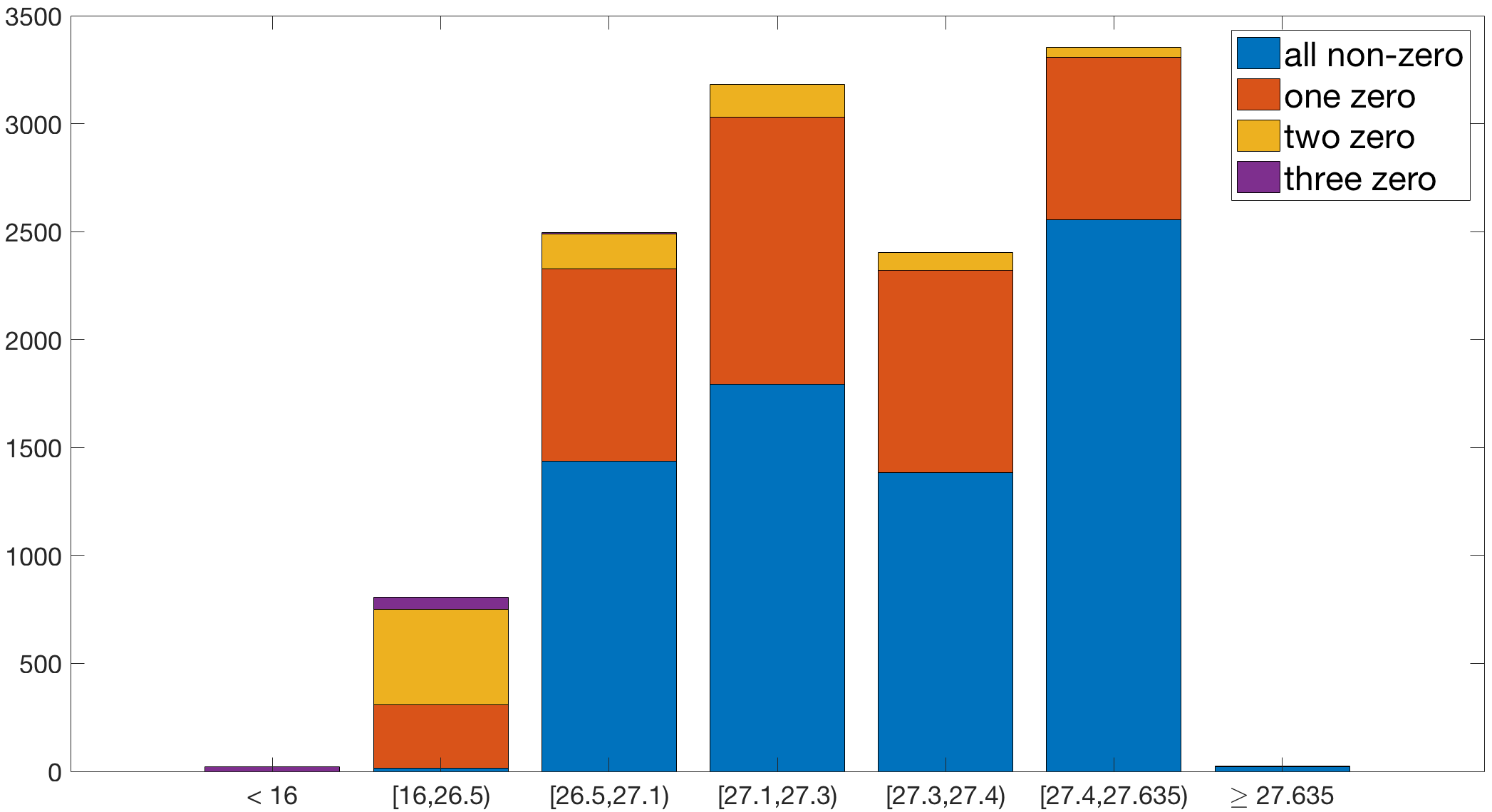}}\hfill
\subfloat[Relative Error]{\includegraphics[width=0.32\textwidth]{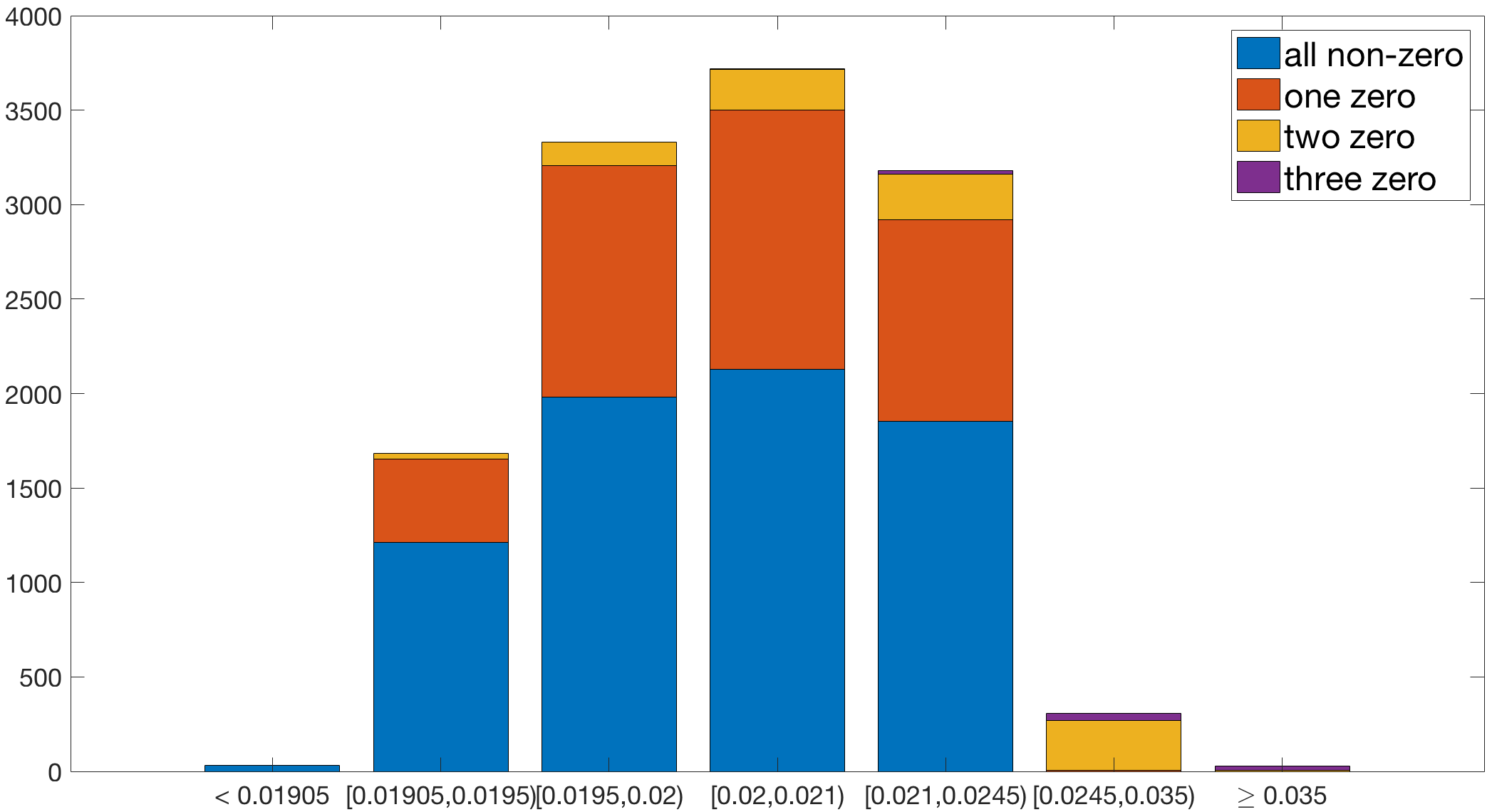}}
\caption{Histograms for Trui considering all tested parameter combinations, sub-divided into four cases: 1) all $\beta_i$ are non-zero (blue), 2) one $\beta_i$ is equal to zero (orange), 3) two $\beta_i$ are equal to zero (yellow), 4) three $\beta_i$ are equal to zero (purple). Note that the bars do not have equal width.}
\label{fig:barplotsZerosTrui}
\end{figure}

In Figure \ref{fig:barplotsZerosTrui}, we examine the occurrences of various quality measure values for different parameter combinations in more depth. More specifically, we sub-divide the results into four classes, dependent on how many $\beta_i$ are non-zero. From this analysis, we can already conclude that scenarios where only one $\beta_i$ is positive and hence only a single differential operator acts on the vector field $w$ in the joint vector operator sparsity regularisation term yield the worst results with respect to our selected measures. Setting two of the $\beta_i$ to zero seems to be the second-worst case. On the other hand, having all $\beta_i$ activated yields the best performing results, which confirms the usefulness and added value of our model and justifies the comparably large number of parameters.

Note at this point that for the multi-colour histograms throughout this section, we manually selected the very differently sized intervals for the bars and heavily customised them such that the different classes become well-separated. Consequently, if a bar still contains a variety of colours, they could not be separated further in a reasonable or meaningful manner.

\begin{figure}[h]
\captionsetup[subfigure]{labelformat=empty}
\centering
\subfloat[SSIM]{\includegraphics[width=0.32\textwidth]{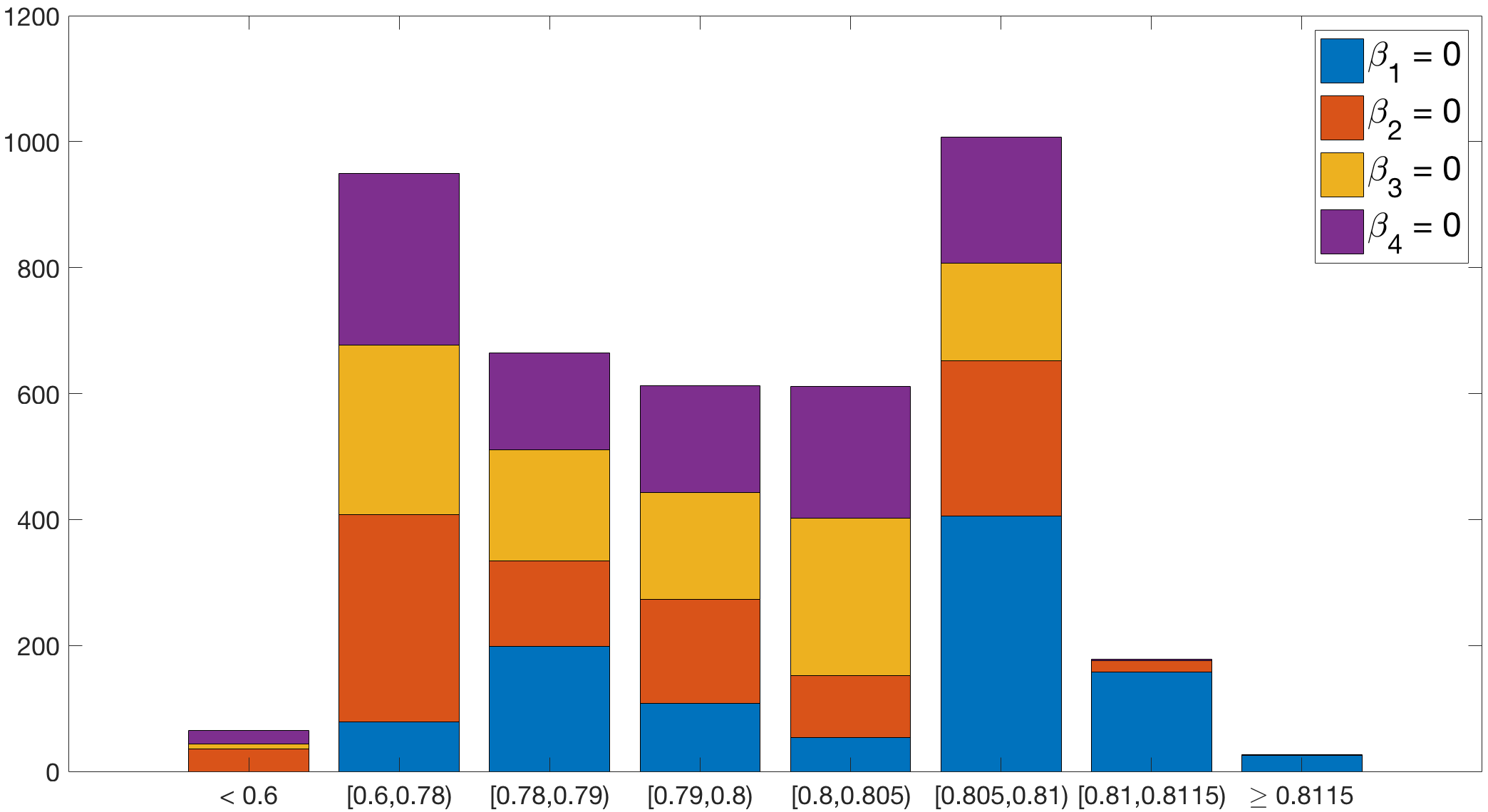}}\hfill
\subfloat[PSNR]{\includegraphics[width=0.32\textwidth]{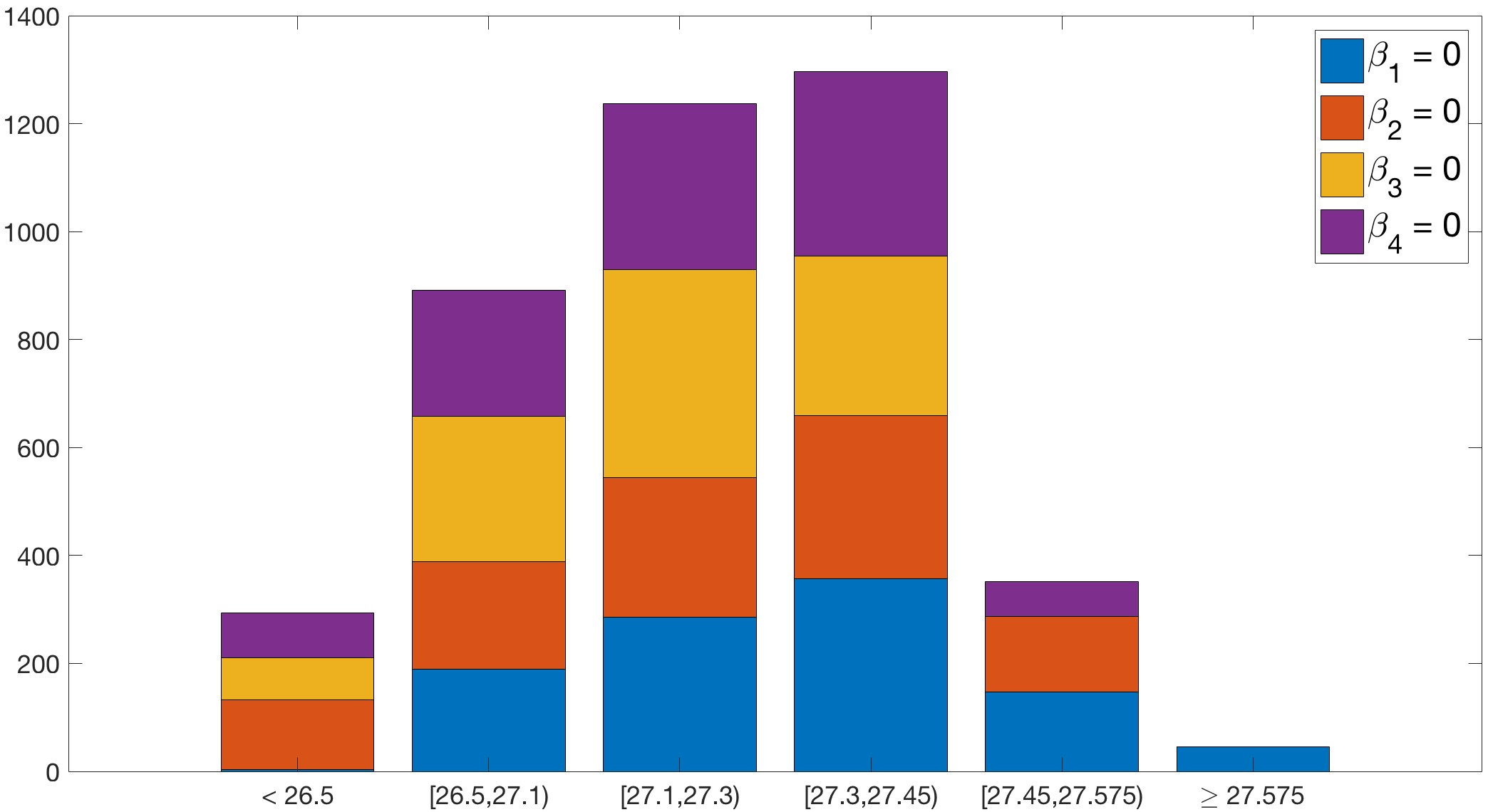}}\hfill
\subfloat[Relative Error]{\includegraphics[width=0.32\textwidth]{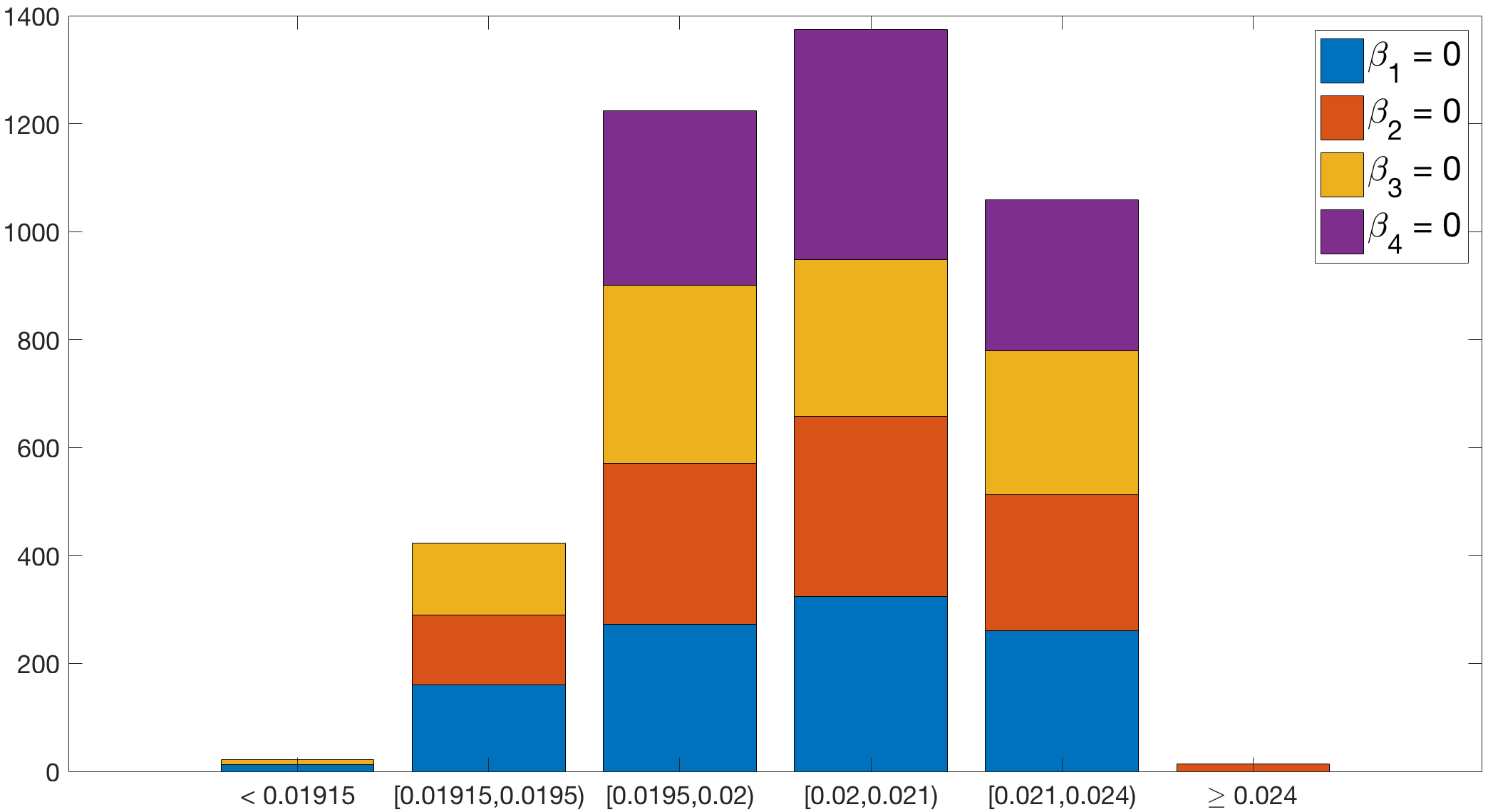}}
\caption{Histograms for Trui in the scenario that one $\beta_i$ is equal to zero: 1) $\beta_1 = 0$ (blue), 2) $\beta_2 = 0$ (orange), 3) $\beta_3 = 0$ (yellow), 4) $\beta_4 = 0$ (purple). Note that the bars do not have equal width.}
\label{fig:barplotsOneZeroTrui}
\end{figure}

In Figure \ref{fig:barplotsOneZeroTrui} we only consider a subset of our results and look at the case where one of the $\beta_i$ is set to zero, i.e.\ where three differential operators are active in our joint regulariser. Also in this scenario we recognise a certain trend. Considering the curl in the regularisation does not seem to be essential, since the best results are achieved in the case when it is set to zero. In contrast, the divergence appears to be of more crucial importance, as setting it to zero produces worse results in general. Of course, this is however strongly dependent on the combination of all five parameters including the overall regularisation weight $\alpha$, and in some cases zero divergence even yields very good results, especially with respect to the SSIM.

\begin{figure}[h]
\captionsetup[subfigure]{labelformat=empty}
\centering
\subfloat[SSIM]{\includegraphics[width=0.32\textwidth]{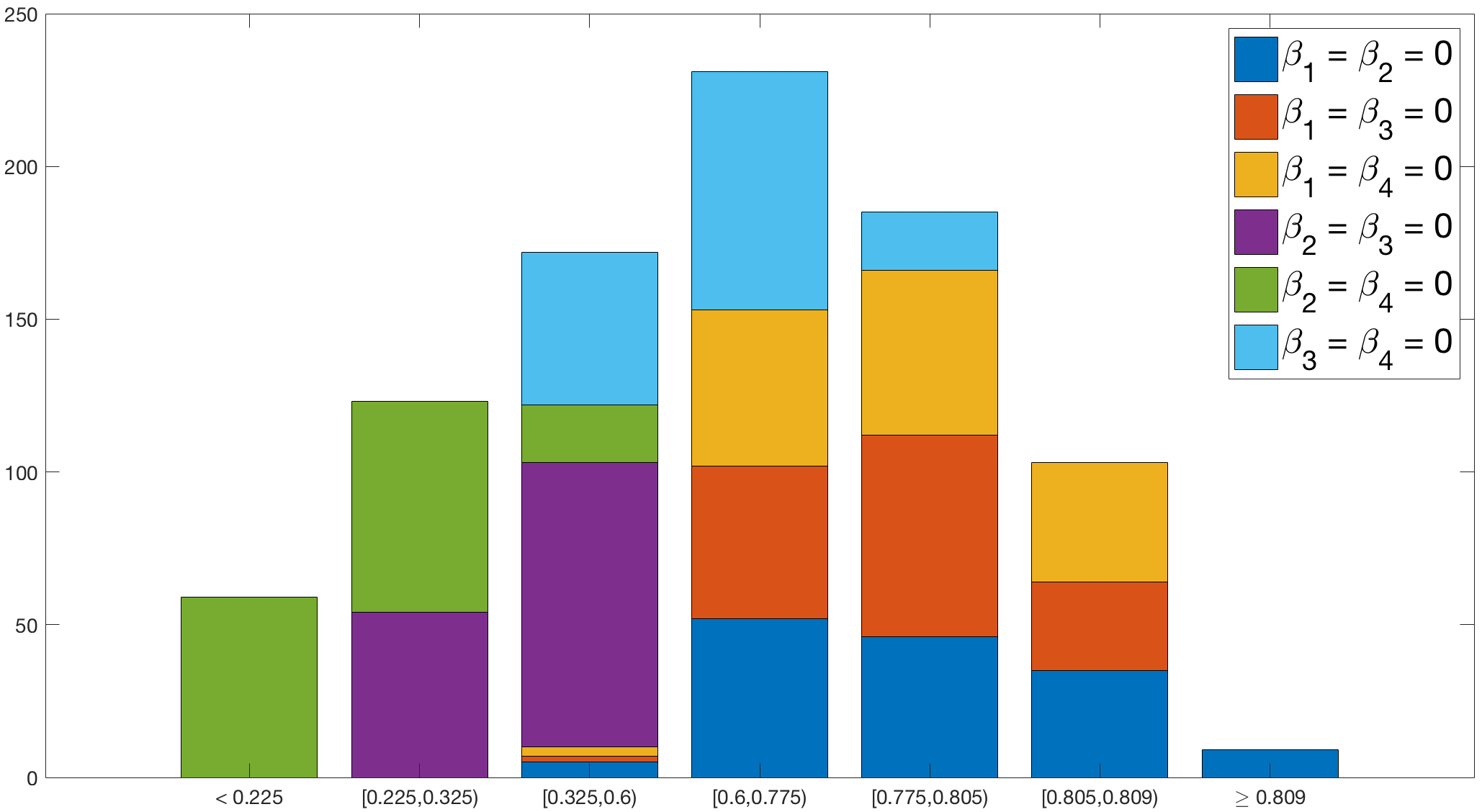}}\hfill
\subfloat[PSNR]{\includegraphics[width=0.32\textwidth]{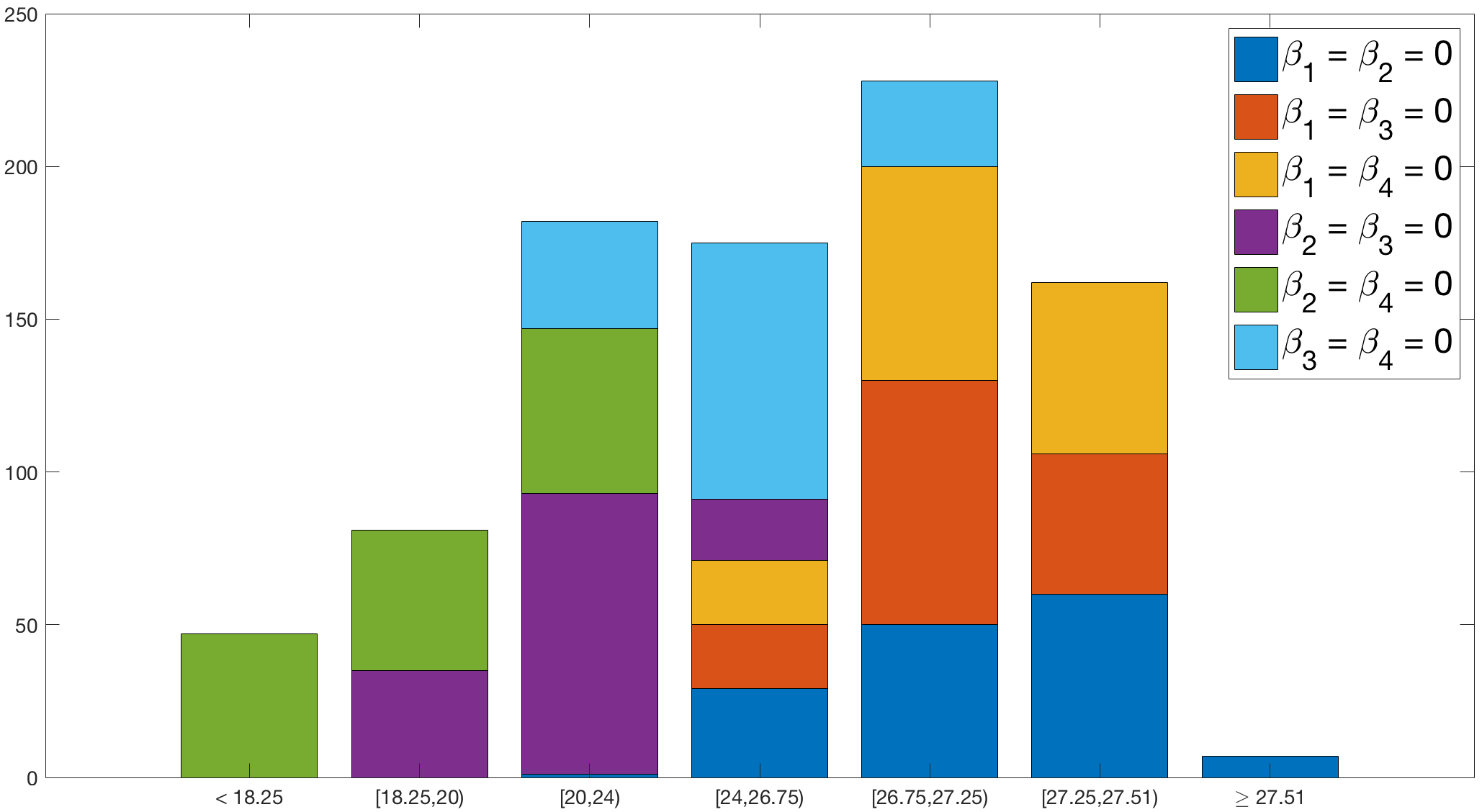}}\hfill
\subfloat[Relative Error]{\includegraphics[width=0.32\textwidth]{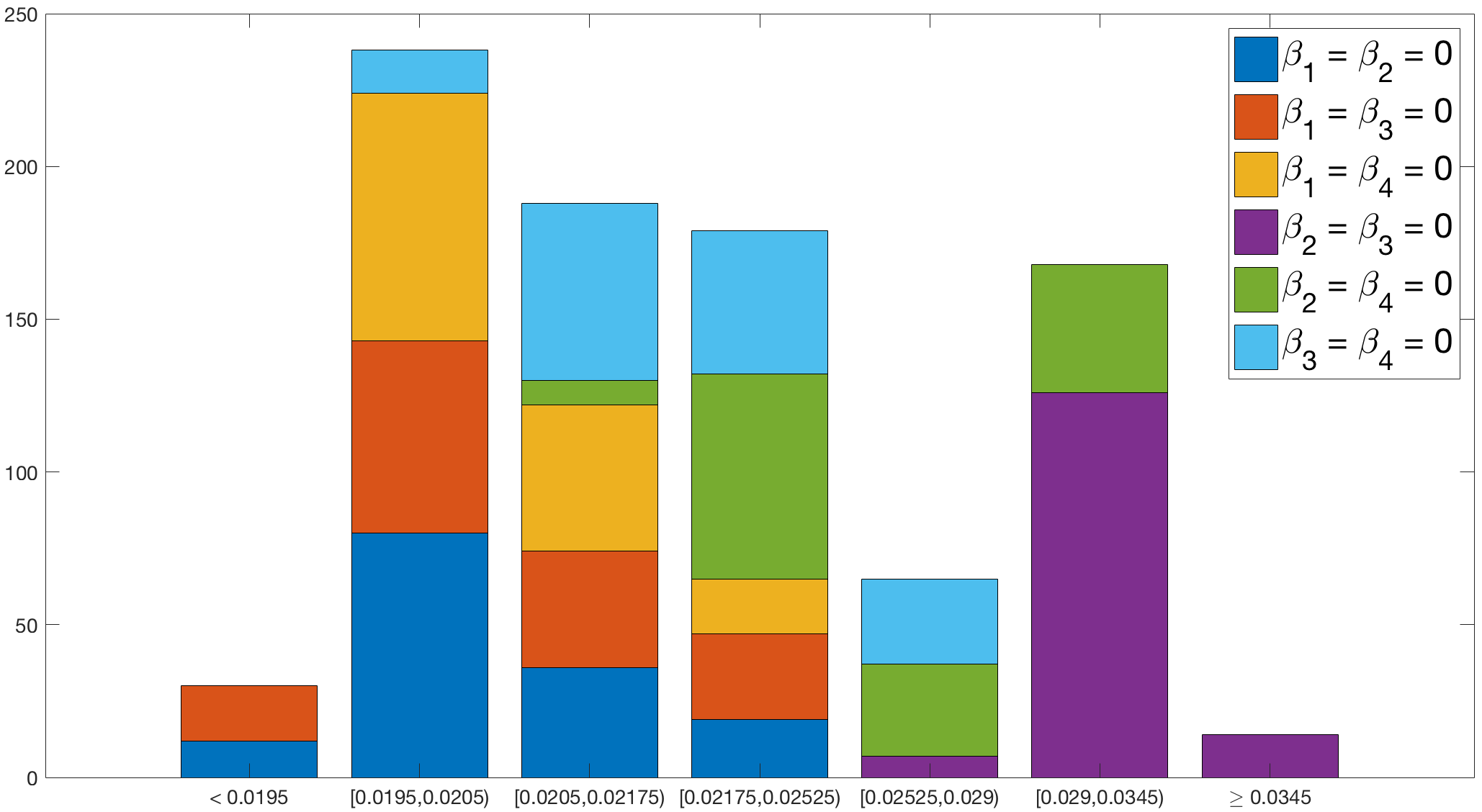}}
\caption{Histograms for Trui in the scenario that two $\beta_i$ are equal to zero: 1) $\beta_1 = \beta_2 = 0$ (blue), 2) $\beta_1 = \beta_3 = 0$ (orange), 3) $\beta_1 = \beta_4 = 0$ (yellow), 4) $\beta_2 = \beta_3 = 0$ (purple), 5) $\beta_2 = \beta_4 = 0$ (green), 6) $\beta_3 = \beta_4 = 0$ (cyan). Note that the bars do not have equal width.}
\label{fig:barplotsTwoZeroTrui}
\end{figure}

The histograms shown in Figure \ref{fig:barplotsTwoZeroTrui} correspond to the case where two of the $\beta_i$ are positive and the other two are set to zero. This yields six different combinations to consider. Interestingly, we again recognise some general trends throughout our data set. In a relatively consistent manner, setting both $\beta_2$, i.e.\ the divergence term, and $\beta_3$ or $\beta_4$, i.e.\ one component of the shear, to zero seems to be a bad idea, as this produces the worst results. This exactly coincides with our observations in Section \ref{sec:unifiedmodel} and more specifically in Figure \ref{fig:reconTest}, where the sparse curl/sh$_1$ and sparse curl/sh$_2$ reconstruction of the piecewise affine square test image contains diagonal and straight line artefacts, respectively. The third worst performing scenario in general is the combination of sparse curl and divergence. Setting $\beta_1$ and either component of the shear to zero results in the second-best reconstructions. In our test we obtain the best performance by only enforcing sparsity in the shear.

\subsubsection*{Piecewise Affine Test Image}

\begin{figure}[h]
\captionsetup[subfigure]{labelformat=empty}
\centering
\subfloat[SSIM]{\includegraphics[width=0.32\textwidth]{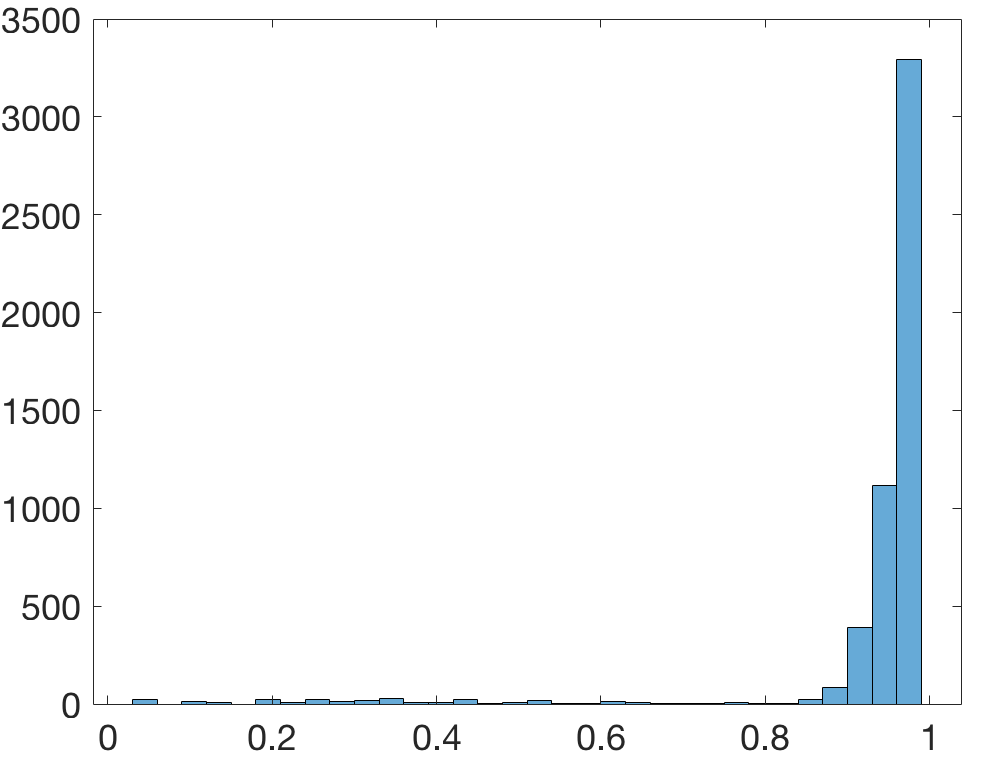}}\hfill
\subfloat[PSNR]{\includegraphics[width=0.32\textwidth]{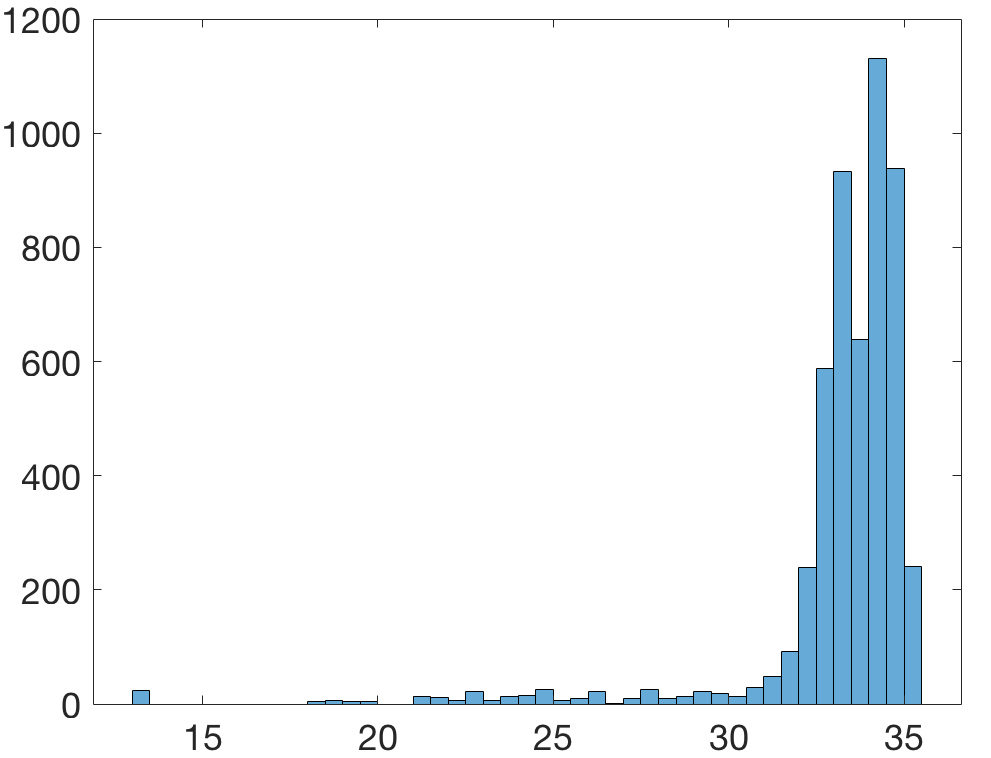}}\hfill
\subfloat[Relative Error]{\includegraphics[width=0.32\textwidth]{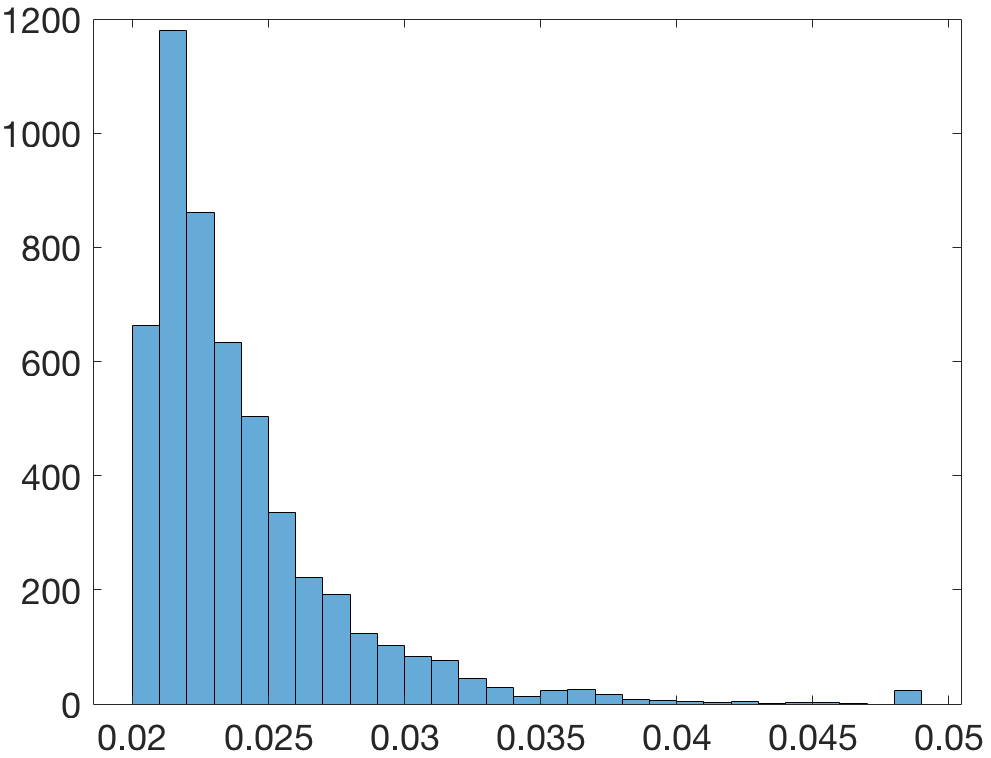}}
\caption{Histograms for piecewise affine image for all tested parameter combinations}
\label{fig:histogramsSquare}
\end{figure}

For the piecewise affine image in Figure \ref{fig:resultsComparisonTest05}, we generally obtain similar results. In Figure \ref{fig:histogramsSquare}, we can see that again, the histograms for the SSIM, PSNR and relative error are concentrated around desirable values, even better ones than for the Trui image. This is probably due to the simpler structure of the piecewise affine test image.

\begin{figure}[h]
\captionsetup[subfigure]{labelformat=empty}
\centering
\subfloat[SSIM]{\includegraphics[width=0.32\textwidth]{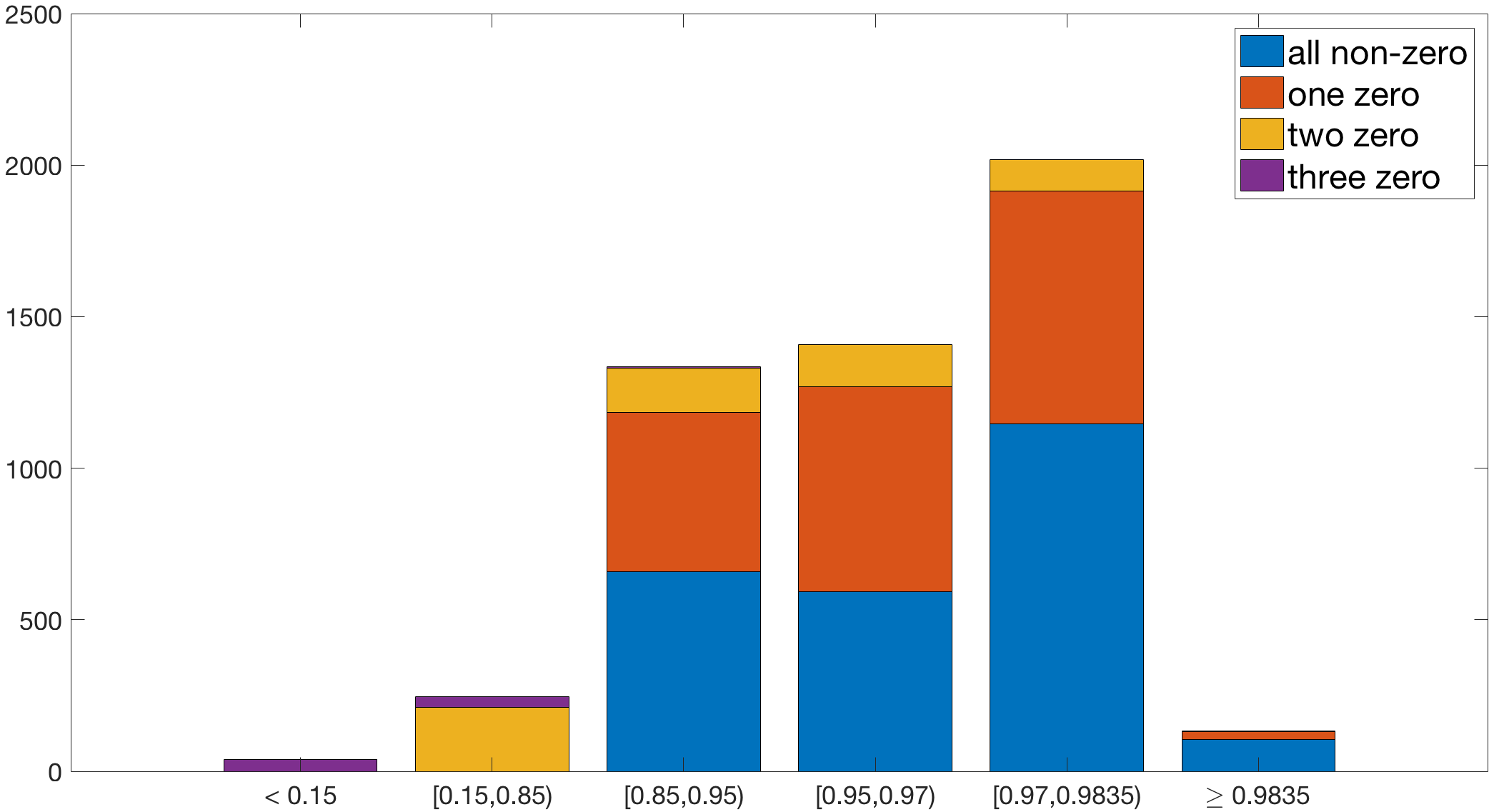}}\hfill
\subfloat[PSNR]{\includegraphics[width=0.32\textwidth]{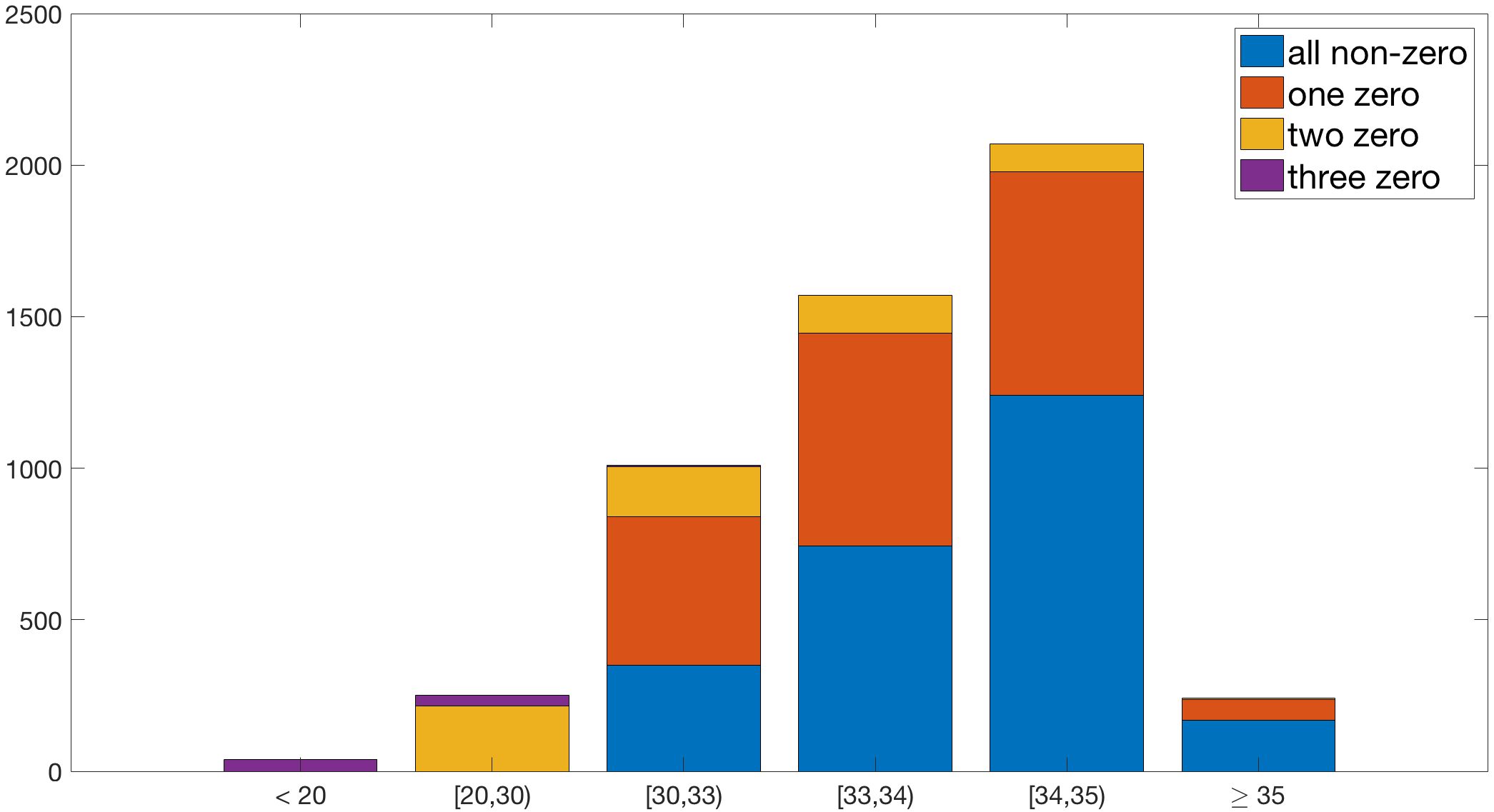}}\hfill
\subfloat[Relative Error]{\includegraphics[width=0.32\textwidth]{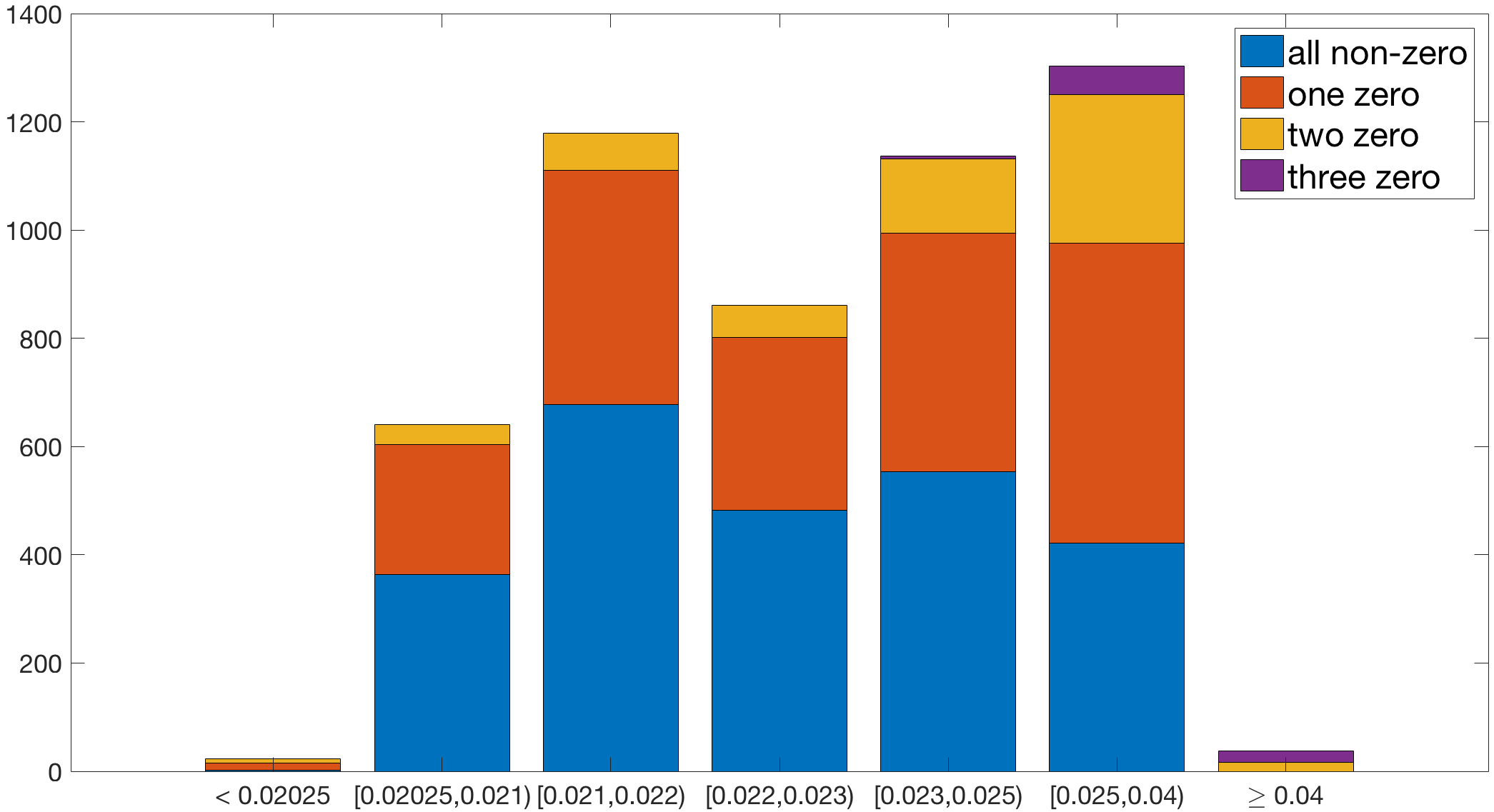}}
\caption{Histograms for piecewise affine image considering all tested parameter combinations, sub-divided into four cases: 1) all $\beta_i$ are non-zero (blue), 2) one $\beta_i$ is equal to zero (orange), 3) two $\beta_i$ are equal to zero (yellow), 4) three $\beta_i$ are equal to zero (purple). Note that the bars do not have equal width.}
\label{fig:barplotsZerosSquare}
\end{figure}

Figure \ref{fig:barplotsZerosSquare} confirms that the more $\beta_i$ are non-zero, the better the denoising reconstructions are in general. The worst and second-worst results are obtained when three or two $\beta_i$ are set to zero, respectively.

\begin{figure}[h]
\captionsetup[subfigure]{labelformat=empty}
\centering
\subfloat[SSIM]{\includegraphics[width=0.32\textwidth]{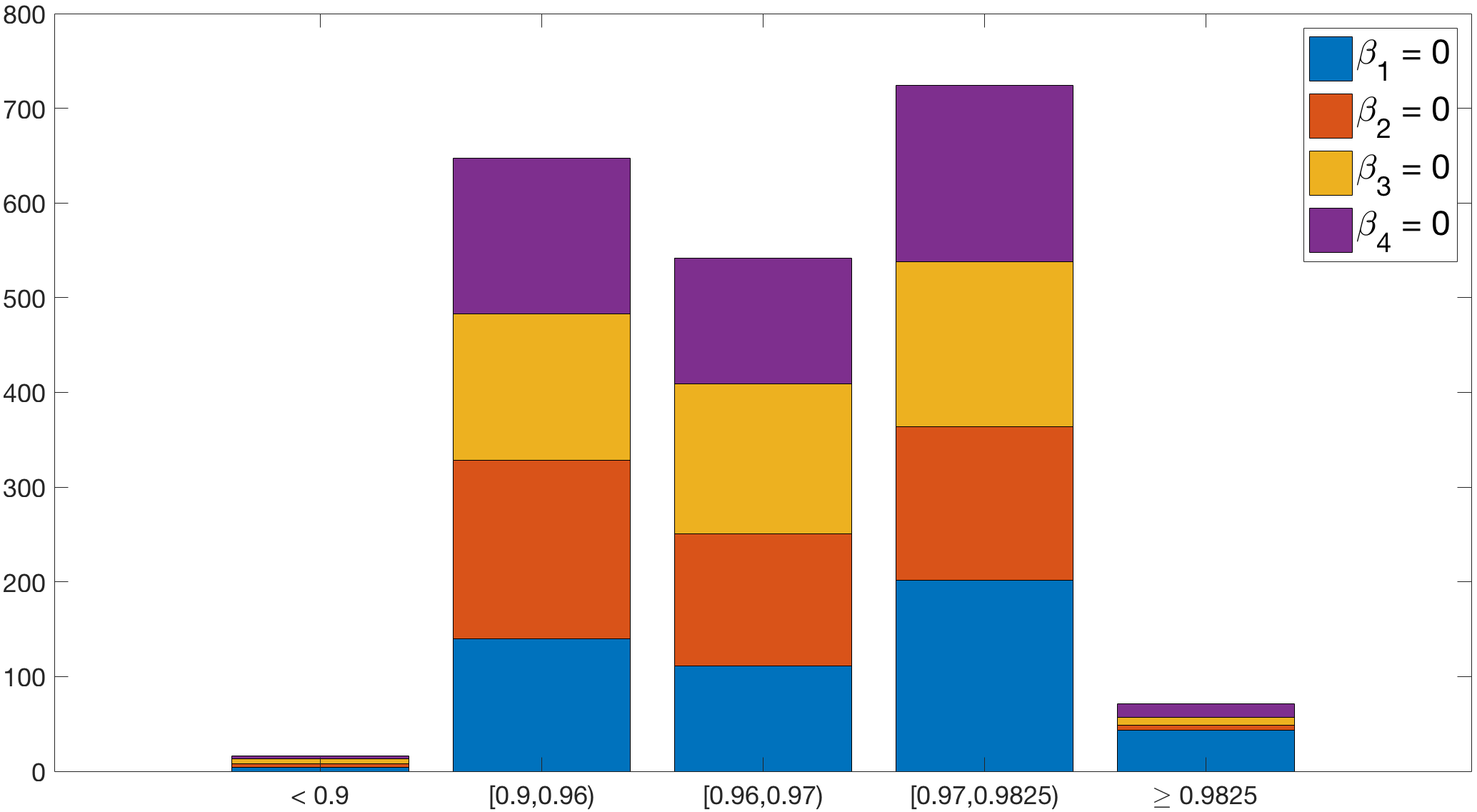}}\hfill
\subfloat[PSNR]{\includegraphics[width=0.32\textwidth]{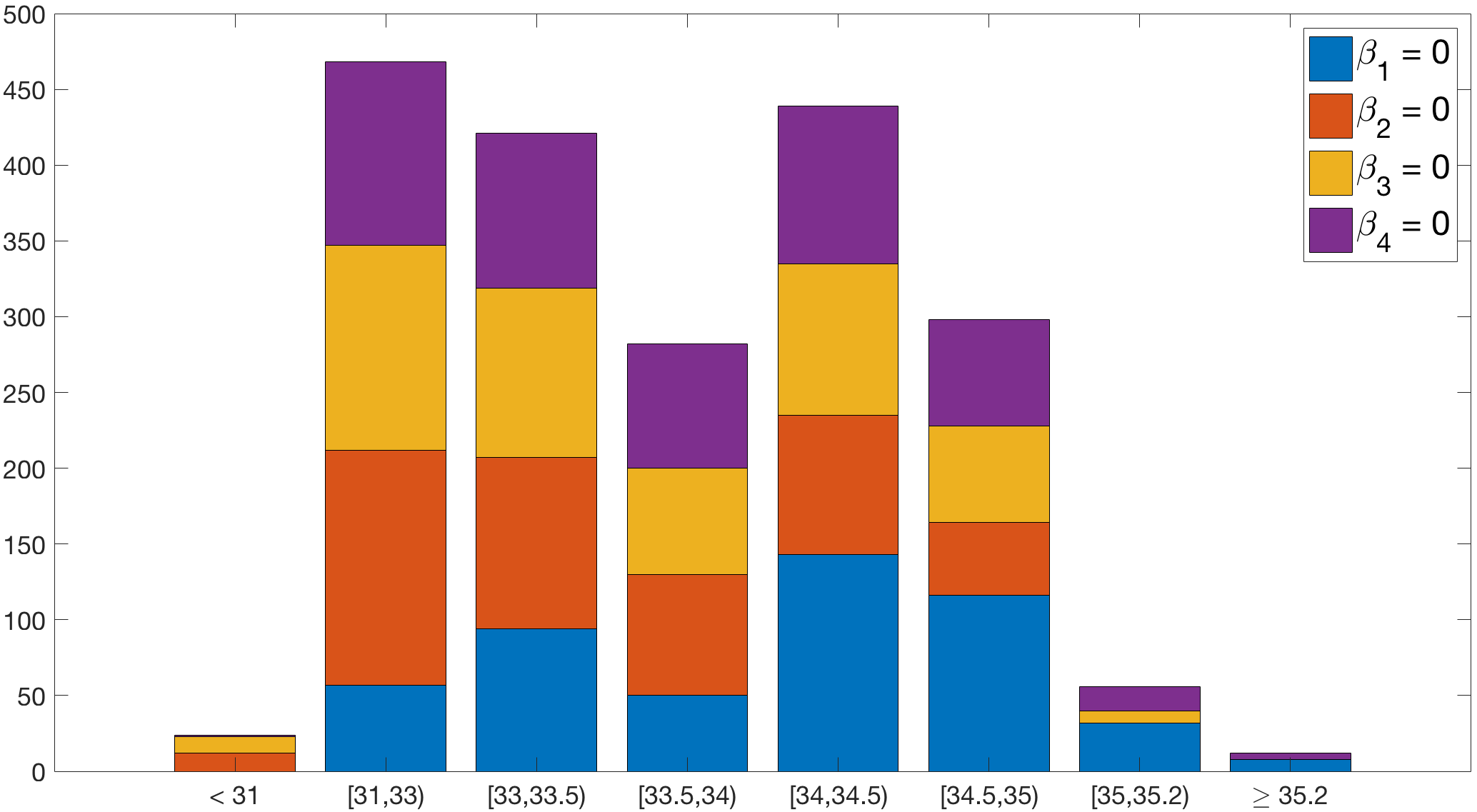}}\hfill
\subfloat[Relative Error]{\includegraphics[width=0.32\textwidth]{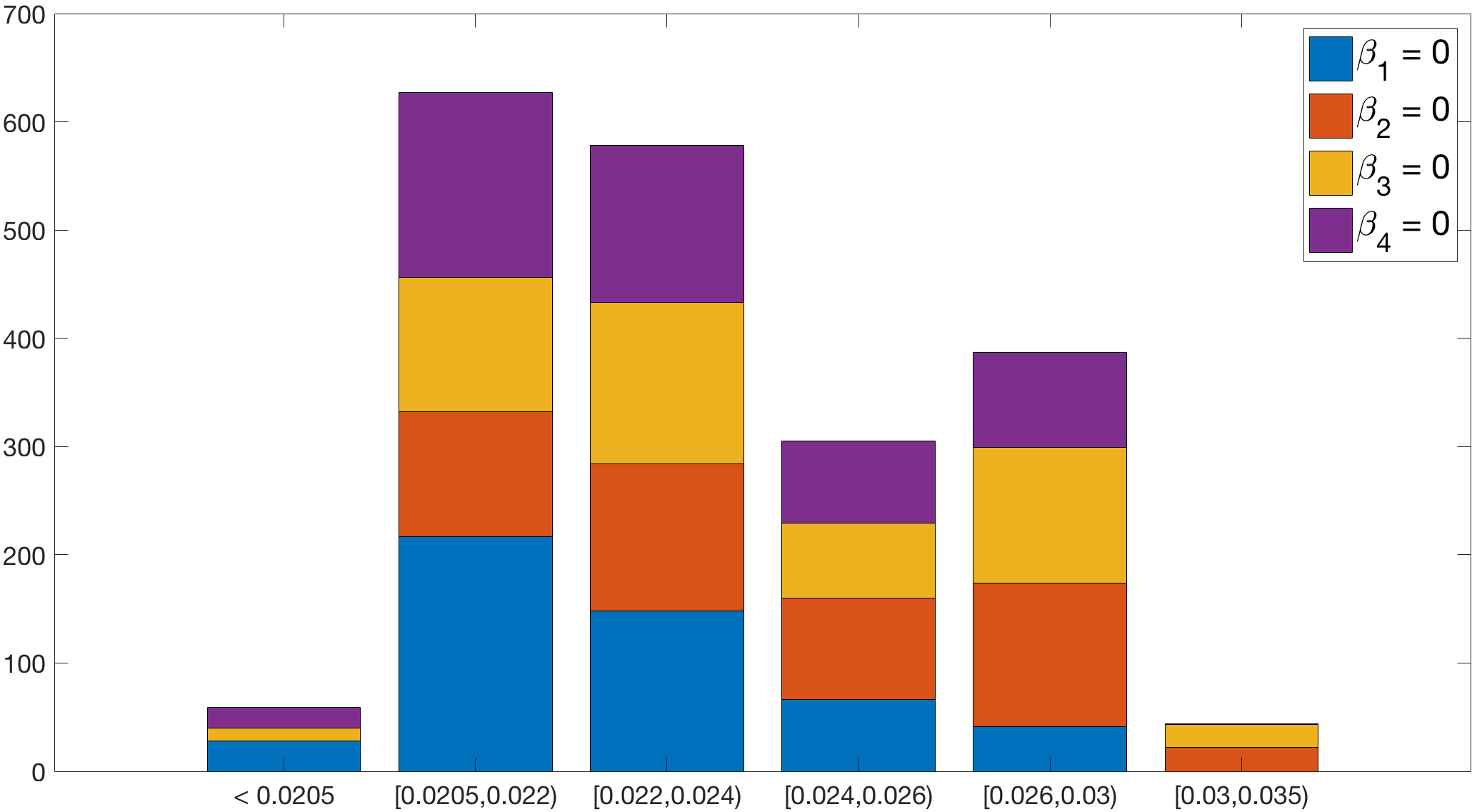}}
\caption{Histograms for piecewise affine image in the scenario that one $\beta_i$ is equal to zero: 1) $\beta_1 = 0$ (blue), 2) $\beta_2 = 0$ (orange), 3) $\beta_3 = 0$ (yellow), 4) $\beta_4 = 0$ (purple). Note that the bars do not have equal width.}
\label{fig:barplotsOneZeroSquare}
\end{figure}

Furthermore, the results in Figure \ref{fig:barplotsOneZeroSquare} reflect the ones in Figure \ref{fig:barplotsOneZeroTrui}. Setting the curl term to zero has a less negative effect compared to omitting the divergence term. However, we cannot make more general statements or draw conclusions regarding the shear terms, as the histograms are rather equally distributed with respect to the four parameter combination scenarios.

\begin{figure}[h]
\captionsetup[subfigure]{labelformat=empty}
\centering
\subfloat[SSIM]{\includegraphics[width=0.32\textwidth]{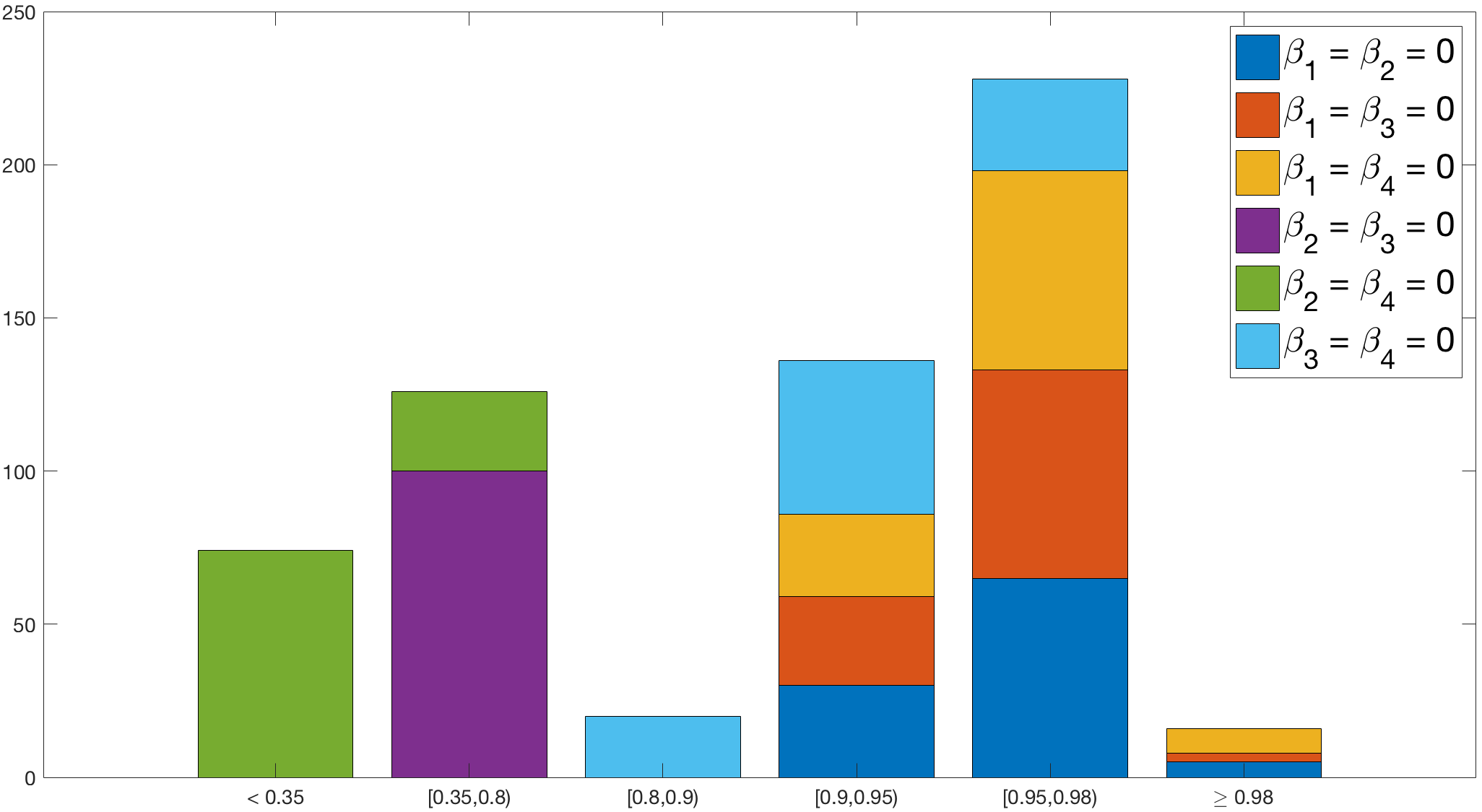}}\hfill
\subfloat[PSNR]{\includegraphics[width=0.32\textwidth]{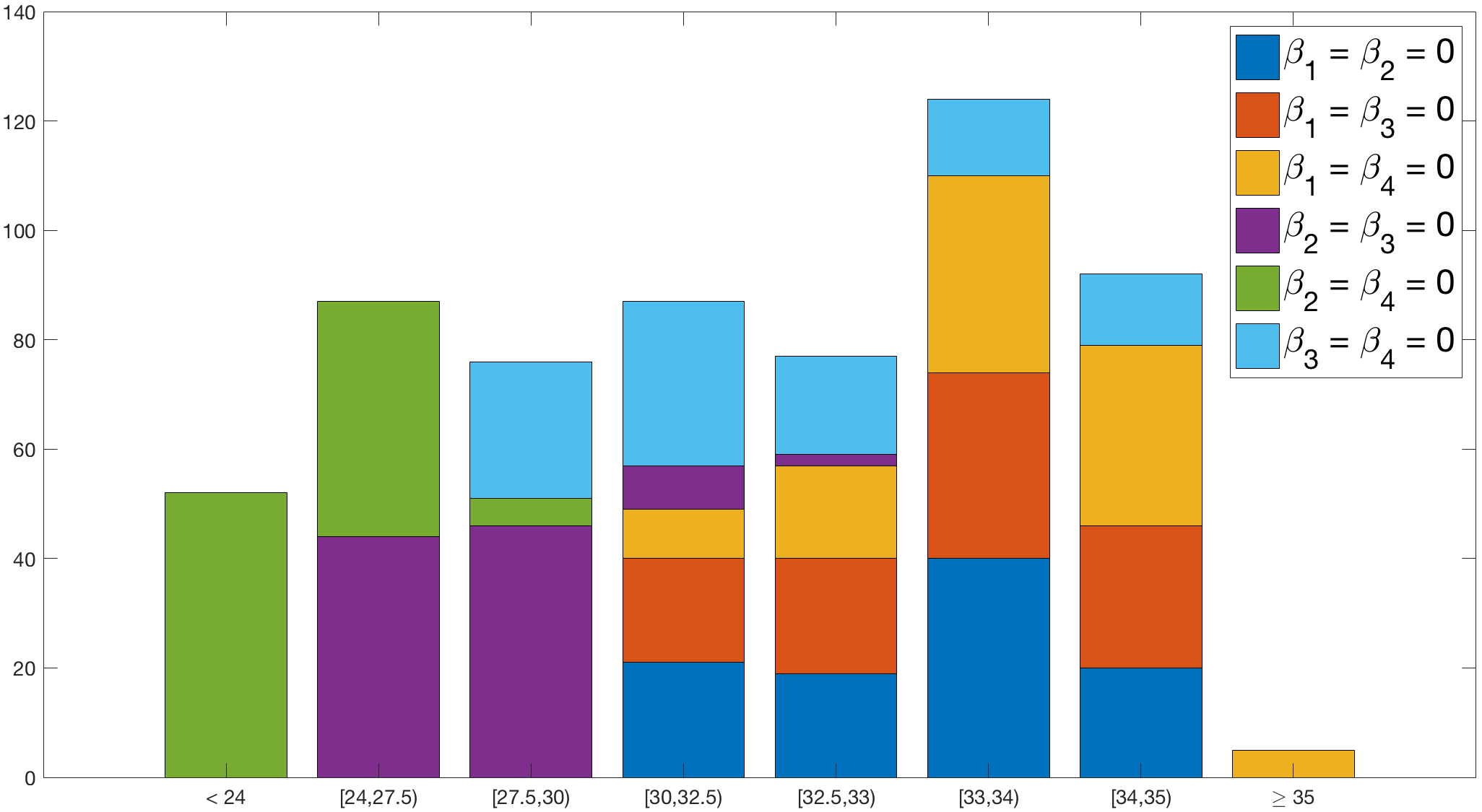}}\hfill
\subfloat[Relative Error]{\includegraphics[width=0.32\textwidth]{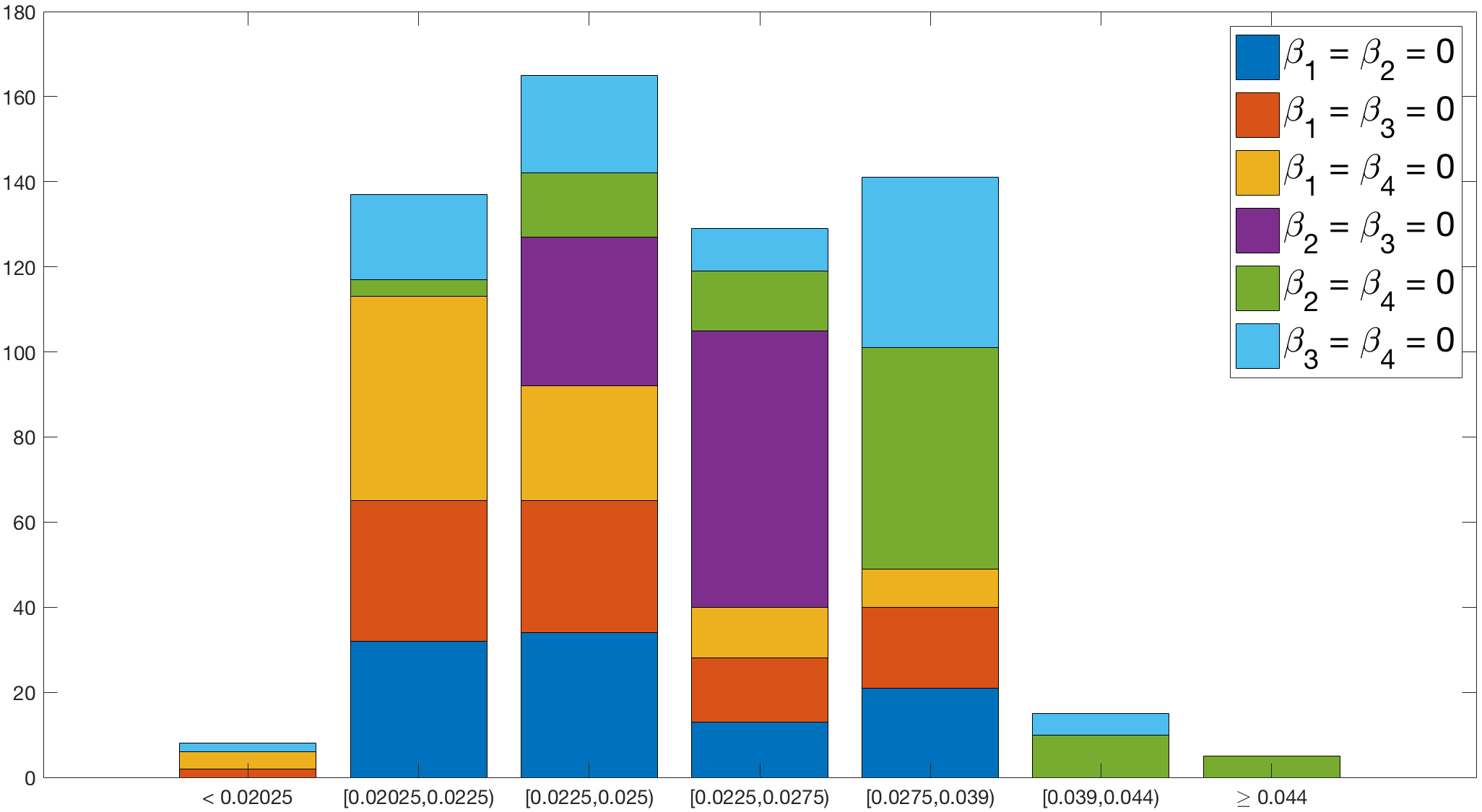}}
\caption{Histograms for piecewise affine image in the scenario that two $\beta_i$ are equal to zero: 1) $\beta_1 = \beta_2 = 0$ (blue), 2) $\beta_1 = \beta_3 = 0$ (orange), 3) $\beta_1 = \beta_4 = 0$ (yellow), 4) $\beta_2 = \beta_3 = 0$ (purple), 5) $\beta_2 = \beta_4 = 0$ (green), 6) $\beta_3 = \beta_4 = 0$ (cyan). Note that the bars do not have equal width.}
\label{fig:barplotsTwoZeroSquare}
\end{figure}

Figure \ref{fig:barplotsTwoZeroSquare} seems to reinforce the statements for Figure \ref{fig:barplotsTwoZeroTrui}. It can be clearly observed that the case where the divergence and the second component of the shear are equal to zero leads to the worst reconstructions with respect to the three quality measures. Also, similar to before, the combinations $\beta_2 = \beta_3 = 0$ and $\beta_3 = \beta_4 = 0$ perform rather poorly.

\section{Conclusion}
\label{sec:conclusion}

\begin{sloppypar}
Starting from our SVF model presented in \cite{SVF}, where we motivated sparsity enforcement of a vector field related to the gradient of the underlying image by an image compression framework using PDE-based diffusion inpainting methods, we extended \eqref{eq:SVF} further by introducing a novel regulariser penalising a joint $L^1$ norm incorporating differential vector field operators. More specifically, we promote sparsity in the curl, divergence and both components of the shear of the vector field at hand. We could dispose of the point artefacts observed in the denoising model in \cite{SVF}. Moreover, similar to well-established higher-order TV models, we avoid the staircasing effect while at the same time enabling piecewise affine reconstructions.

We showed that our unified regulariser can be viewed as a generalisation of a number of already existing frameworks: We can recover TV, our previously presented SVF model, CEP, second-order TGV and ICTV. Furthermore, we showed the capability of our model to interpolate between the latter two methods by changing the value of only one weighting parameter. We also saw that a wide range of parameters $\beta_i$ yields very similar results, confirming the robustness of the approach. In particular, this holds true if three of the $\beta_i$ are chosen to be non-zero (while not approaching infinity all at the same time) or if we pick two out of $\beta_2$, $\beta_3$ and $\beta_4$ to be positive weights, as we concluded that the curl has only marginal influence.
Our results also lead to the conjecture that visually more pleasing reconstructions are obtained if we indeed arrive at singularities along edges rather than in points, since the latter are visible as artefacts in the images. In view of this paper, it is hence recommended to either combine at least three natural vector field operators or the divergence and one component of the shear for the regularisation.

There are various interesting directions for future research. As we mentioned earlier, the denoising case was just an academic testbed for studying the regularisations; its use might become much more relevant in other inverse problems and image reconstruction frameworks. 
Moreover, our results could naturally be reconsidered in the regularisation of problems for vector fields such as motion estimation, where divergence, curl, and shear even have physical interpretations. 
In this context it is an often heard conjecture that in light of the Helmholtz decomposition divergence and curl are sufficient for regularisation. However, the combination of the two operators only yields satisfactory regularisation properties if their joint penalisation is combined with suitable boundary conditions as, for instance, accomplished in \cite{staggeredGrids5}.
Since the results presented in this paper indicate that a functional combining at least three suitably chosen differential operators is also capable of providing an equivalent regularisation in the space of bounded variation without the need to guarantee any boundary conditions, this might be an interesting alternative approach for the regularisation of vector fields that might require a less cumbersome numerical implementation.  
Furthermore, it would be interesting to reconsider higher-order regularisation on graphs, in particular to study variants of TGV on such structures. Since the divergence is the only natural differential operator for vector fields (edge functions) on graphs, our approach might be even more relevant in such a setting. 

Finally, we come to the issue of optimal parameter choice, since our approach yields quite some freedom in this respect. To overcome this, parameter learning using bi-level optimisation techniques might be particularly suited.
\end{sloppypar}


\section*{Acknowledgements}
The authors thank Kristian Bredies, Martin Holler (both University of Graz) and Christoph Schn\"orr (University of Heidelberg) for useful discussions and links to literature.

This work has been supported by ERC via Grant EU FP 7 - ERC Consolidator Grant 615216 LifeInverse. JSG acknowledges support by The Alan Turing Institute under the EPSRC grant EP/N510129/1 and by the NIHR Cambridge Biomedical Research Centre. The authors would like to thank the Isaac Newton Institute for Mathematical Sciences, Cambridge, for support and hospitality during the programme Variational Methods for Imaging and Vision, where work on this paper was undertaken, supported by EPSRC grant no EP/K032208/1 and the Simons Foundation.

\section*{Data Statement}
The corresponding MATLAB\textsuperscript{\textregistered} code (implemented and tested with R2018a) is publicly available on GitHub\footnote{Image denoising using the unified model in this work: \url{https://github.com/JoanaGrah/VectorOperatorSparsity}; image compression using the sparse vector fields model in \cite{SVF}: \url{https://github.com/JoanaGrah/SparseVectorFields}}.

\bibliographystyle{plain}
\bibliography{Manuscript}   


\appendix
\clearpage
\section{Derivation of Nullspaces}
\label{app:nullspaces}
In the following, we aim at characterising the set of all $u \in L^2(\Omega)$ for which $R_{\bm \beta}(u) = 0$ holds.

\begin{sloppypar}
At first we consider the case $\beta_2=0$ and $\beta_3,\beta_4 > 0$. Following the line of argument for the derivation of the nullspaces in Section \ref{sec:diffop}, it is clear that in order to be in the nullspace $u$ has to satisfy
\begin{equation*}
u(x) = U(x_1 + x_2) + V(x_1 - x_2) = U_1(x_1) + U_2(x_2).
\end{equation*}
Calculation of first- and second-order derivatives of $u$ then yields the following identities for the gradient and the Hessian of $u$:
\begin{align*}
\nabla u (x) &= \begin{pmatrix}
\frac{\partial}{\partial x_1}U(x_1 + x_2) + \frac{\partial}{\partial x_1}V(x_1 - x_2)\\
\frac{\partial}{\partial x_2}U(x_1 + x_2) - \frac{\partial}{\partial x_2}V(x_1 - x_2)
\end{pmatrix}\\
&= \begin{pmatrix}
\frac{\partial}{\partial x_1}U_1(x_1) + \frac{\partial}{\partial x_1}U_2(x_2)\\
\frac{\partial}{\partial x_2}U_1(x_1) + \frac{\partial}{\partial x_2}U_2(x_2)
\end{pmatrix}
= \begin{pmatrix}
\frac{\partial}{\partial x_1}U_1(x_1)\\
\frac{\partial}{\partial x_2}U_2(x_2)
\end{pmatrix}
\end{align*}
and 
\begin{equation*}
Hu = \begin{pmatrix}
(Hu)_{11} & (Hu)_{12}\\
(Hu)_{21} & (Hu)_{22}
\end{pmatrix},
\end{equation*}
where
\begin{align*}
(Hu)_{11}(x) &= \frac{\partial^2}{\partial x_1^2}U(x_1 + x_2) + \frac{\partial^2}{\partial x_1^2}V(x_1 - x_2)\\
&= \frac{\partial^2}{\partial x_1^2}U_1(x_1)\\
(Hu)_{12}(x) &= \frac{\partial^2}{\partial x_1\partial x_2}U(x_1 + x_2) - \frac{\partial^2}{\partial x_1\partial x_2}V(x_1 - x_2)\\
&= \frac{\partial^2}{\partial x_1 \partial x_2}U_1(x_1) + \frac{\partial^2}{\partial x_1 \partial x_2}U_2(x_2) = 0\\
(Hu)_{21}(x) &= \frac{\partial^2}{\partial x_1\partial x_2}U(x_1 + x_2) - \frac{\partial^2}{\partial x_1\partial x_2}V(x_1 - x_2)\\
&= \frac{\partial^2}{\partial x_1 \partial x_2}U_1(x_1) + \frac{\partial^2}{\partial x_1 \partial x_2}U_2(x_2) = 0\\
(Hu)_{22}(x) &= \frac{\partial^2}{\partial x_2^2}U(x_1 + x_2) + \frac{\partial^2}{\partial x_2^2}V(x_1 - x_2)\\
&= \frac{\partial^2}{\partial x_2^2}U_2(x_2).
\end{align*}
In particular, we observe:
\begin{equation*}
\frac{\partial^2}{\partial x_1^2} U_1(x_1) = \frac{\partial^2}{\partial x_2^2} U_2(x_2) \quad \text{ for all } x_1, x_2,
\end{equation*}
which can only be true if $\frac{\partial^2}{\partial x_1^2} U_1(x_1)$ and $\frac{\partial^2}{\partial x_2^2} U_2(x_2)$ are equal and constant, i.e. $\frac{\partial^2}{\partial x_1^2} U_1(x_1) = \frac{\partial^2}{\partial x_2^2} U_2(x_2) = c$.\\
Twofold integration of $\frac{\partial^2}{\partial x_1^2} U_1$ respectively $\frac{\partial^2}{\partial x_2^2} U_2$ on condition that the former only depends on $x_1$ while the latter only depends on $x_2$ yields:
\begin{align*}
\frac{\partial}{\partial x_1} U_1(x_1) = \int c \, dx_1 = cx_1 + d_1,\\
\frac{\partial}{\partial x_2} U_2(x_2) = \int c \, dx_2 = cx_2 + e_1
\end{align*}
and thus
\begin{align*}
&U_1(x_1) = \int cx_1 + d_1 \, dx_1 = cx_1^2 + d_1x_1 + d_0\\
&U_2(x_2) = \int cx_2 + e_1 \, dx_1 = cx_2^2 + e_1x_2 + e_0\\
&\Longrightarrow u = c(x_1^2 + x_2^2) + d_1 x_1 + e_1 x_2 + (d_0 + e_0).
\end{align*}
Consequently the nullspace only consists of functions that are linear combinations of $x_1^2 + x_2^2, x_1, x_2$ and $1$.
\end{sloppypar}

We continue with the case $\beta_3 = 0$ and $\beta_2,\beta_4 > 0$. By the discussion of the nullspaces in Section \ref{sec:diffop} $u$ has to be harmonic, i.e. 
\begin{equation*}
\frac{\partial^2}{\partial x_1^2}u(x) + \frac{\partial^2}{\partial x_2^2}u(x) = 0,
\end{equation*}
and moreover it has to be of the form $u(x) = U_1(x_1) + U_2(x_2)$.
Taking into account the calculations of the first- and second-order partial derivatives in the previous case, we easily see that the above equality is equivalent to 
\begin{equation*}
\frac{\partial^2}{\partial x_1^2}U_1(x_1) + \frac{\partial^2}{\partial x_2^2}U_2(x_2) = 0 \quad \text{ for all } x_1,x_2,
\end{equation*}
which obviously can only be true if $\frac{\partial^2}{\partial x_1^2}U_1(x_1)$ and $\frac{\partial^2}{\partial x_2^2}U_2(x_2)$ are constant with constants summing to zero. On this basis we analogously to the previous case integrate $\frac{\partial^2}{\partial x_1^2} U_1$ and $\frac{\partial^2}{\partial x_2^2} U_2$ twice on condition that the former only depends on $x_1$ and the latter only depends on $x_2$
\begin{align*}
\frac{\partial}{\partial x_1} U_1(x_1) &= \int c \, dx_1 = cx_1 + d_1,\\
\frac{\partial}{\partial x_2} U_2(x_2) &= \int -c \, dx_2 = -cx_2 + e_1
\end{align*}
and hence
\begin{align*}
&U_1(x_1) = \int cx_1 + d_1 \, dx_1 = cx_1^2 + d_1x_1 + d_0\\
&U_2(x_2) = \int -cx_2 + e_1 \, dx_1 = -cx_2^2 + e_1x_2 + e_0\\
&\Longrightarrow u = c(x_1^2 - x_2^2) + d_1 x_1 + e_1 x_2 + (d_0 + e_0).
\end{align*}
The nullspace thus only consists of functions that are linear combinations of $x_1^2 - x_2^2, x_1, x_2$ and $1$.

Finally, we study the case $\beta_4 = 0$ and $\beta_2, \beta_3 > 0$. Analogous to the previous case we argue that by the characterisation of the nullspaces in Section \ref{sec:diffop} $u$ is of the form $u(x) = U(x_1 + x_2) + V(x_1 - x_2)$ and again has to be harmonic, i.e.
\begin{equation*}
\frac{\partial^2}{\partial x_1^2}u(x) + \frac{\partial^2}{\partial x_2^2}u(x) = 0.
\end{equation*}
Again, we reconsider the first- and second-order partial derivatives from the first case and obtain for all $x_1,x_2$
\begin{align*}
&2 \left(\frac{\partial^2}{\partial x_1^2}U(x_1 + x_2) + \frac{\partial^2}{\partial x_2^2}V(x_1 - x_2)\right) = 0 
\end{align*}
which implies that $\frac{\partial^2}{\partial x_1^2}U$ and $\frac{\partial^2}{\partial x_2^2}V$ are constant with constants summing to zero. By twofold integration of $\frac{\partial^2}{\partial x_1^2}U$ and $\frac{\partial^2}{\partial x_2^2}V$ on condition that the former depends on $x_1 + x_2$ and the latter depends on $x_1 - x_2$ we thus obtain:
\begin{align*}
&\frac{\partial}{\partial x_1} U(x_1 + x_2) = \int c \, d(x_1 + x_2) = c (x_1 + x_2) + d_1,\\
&\frac{\partial}{\partial x_2} V(x_1 - x_2) = \int -c \, d(x_1 - x_2) = -c(x_1 - x_2) + e_1
\end{align*}
and hence
\begin{align*}
&U(x_1 + x_2) = \int c(x_1 + x_2) + d_1 \,d(x_1 + x_2) = c(x_1 + x_2)^2 + d_1(x_1 + x_2) + d_0\\
&V(x_1 - x_2) = \int -c(x_1 - x_2) + e_1 \,d(x_1 - x_2) = -c(x_1 - x_2)^2 + e_1(x_1 - x_2) + e_0\\
&\Longrightarrow u = 4c x_1 x_2 + (d_1 + e_1) x_1 + (d_1 - e_1) x_2 + (d_0 + e_0).
\end{align*}
As a result the nullspace contains all functions that are linear combinations of $x_1x_2,x_1,x_2$ and $1$.

\section{Proof of Theorem \ref{rotinvthm}}
\label{app:proof of Theorem 4}

\setcounter{thm}{3}

\begin{thm}
Let $\beta_i \geq 0$ for $i = 1,\dots,4$ and let $\beta_3 = \beta_4$. Then the regulariser $R_{\bm \beta}(u)$ is rotationally invariant, i.e., for an orthonormal rotation matrix $\bm{Q} \in \mathbb{R}^{2\times2}$ with
\begin{equation*}
\bm{Q}(\theta) = \begin{pmatrix}
\cos(\theta) & -\sin(\theta)\\
\sin(\theta) & \cos(\theta)
\end{pmatrix} \quad \text{ for } \theta \in \left[0,2\pi\right)
\end{equation*}
and for $u \in BV(\Omega)$ it holds that $\check{u} \in BV(\Omega)$, where $\check{u}=u\circ \bm{Q}$, i.e. $\check{u}(x) = u(\bm{Q}x)$ for a.e. $x \in \Omega$, and
\begin{equation*}
R_{\bm \beta}(\check{u}) = R_{\bm \beta}(u).
\end{equation*}
\end{thm}
\begin{proof}
In order to prove the assertion we consider $\check{u}=u \circ \bm{Q}$ and show that we obtain $R_{\bm \beta}(\check{u}) = R_{\bm \beta}(u)$, where as before
\begin{equation*}
R_{\bm \beta}(u) = \inf_{w \in \mathcal{M}(\Omega,\mathbb{R}^2)} \Vert \nabla u - w \Vert_{\mathcal{M}(\Omega,\mathbb{R}^2)} + \Vert \text{diag}({\bm \beta})\nabla_N w\Vert_{\mathcal{M}(\Omega,\mathbb{R}^4)}.
\end{equation*}
Inserting $\check{u}$ in the first term of the regulariser, we realise that we obtain the equivalence to the first term of $R_{\bm \beta}(u)$ by choosing $\check{w}=\bm{Q}^{\top}w \circ \bm{Q}$, i.e., 
$$ \int_\Omega \varphi(x)~d\check w = \int_\Omega \bm{Q} \varphi(\bm{Q}^T x) ~dw, \qquad \forall \varphi \in C_0(\Omega;\mathbb{R}^2),$$
since
\begin{align*}
\Vert \nabla \check{u} - \check{w} \Vert_{\mathcal{M}(\Omega,\mathbb{R}^2)} & \; = \Vert \bm{Q}^{\top}\nabla u\circ \bm{Q} - \bm{Q}^{\top}w\circ \bm{Q}\Vert_{\mathcal{M}(\Omega,\mathbb{R}^2)}\\
& \; = \Vert \bm{Q}^{\top}\left(\nabla u \circ \bm{Q} - w \circ \bm{Q}\right)\Vert_{\mathcal{M}(\Omega,\mathbb{R}^2)}\\
& \; = \Vert \nabla (u \circ \bm{Q}) - w \circ \bm{Q} \Vert_{\mathcal{M}(\Omega,\mathbb{R}^2)}
\end{align*}

Thus, if we can show that for $\check{w}=\bm{Q}^{\top}w \circ \bm{Q}$  we also obtain the equivalence of the second term of the regulariser to the second term of $R_{\bm \beta}(u)$, we have proven the assertion. To this end we set $v = \bm{Q}^{\top}w$ and compute
\begin{equation*}
v = \begin{pmatrix}
\cos(\theta)w_1 + \sin(\theta)w_2\\ - \sin(\theta)w_1 + \cos(\theta)w_2
\end{pmatrix}.
\end{equation*}
In addition we need the Jacobian matrix $\nabla v$ of $v$, where
\begin{align*}
(\nabla v)_{11} &= \cos(\theta)\frac{\partial w_1}{\partial x_1} + \sin(\theta)\frac{\partial w_2}{\partial x_1}, \quad &(\nabla v)_{12} &= \cos(\theta)\frac{\partial w_1}{\partial x_2} + \sin(\theta)\frac{\partial w_2}{\partial x_2},\\
(\nabla v)_{21} &= -\sin(\theta)\frac{\partial w_1}{\partial x_1} + \cos(\theta) \frac{\partial w_2}{\partial x_1}, \quad &(\nabla v)_{22} &= -\sin(\theta)\frac{\partial w_1}{\partial x_2} + \cos(\theta) \frac{\partial w_2}{\partial x_2}.
\end{align*}
We can hence obtain the Jacobian matrix $\nabla \check{w}$ of $\check{w}$ by computing $\nabla \check{w} = \bm{Q}^{\top}\nabla v$ yielding
\begin{align*}
(\nabla \check{w})_{11} & \; = \cos^2(\theta) \frac{\partial w_1}{\partial x_1} + \cos(\theta)\sin(\theta) \frac{\partial w_2}{\partial x_1} + \cos(\theta)\sin(\theta) \frac{\partial w_1}{\partial x_2} + \sin^2(\theta) \frac{\partial w_2}{\partial x_2},\\
(\nabla \check{w})_{12} & \; = -\cos(\theta)\sin(\theta) \frac{\partial w_1}{\partial x_1} - \sin^2(\theta) \frac{\partial w_2}{\partial x_1} + \cos^2(\theta) \frac{\partial w_1}{\partial x_2} + \cos(\theta)\sin(\theta) \frac{\partial w_2}{\partial x_2},\\
(\nabla \check{w})_{21} & \; = \cos^2(\theta) \frac{\partial w_2}{\partial x_1} - \cos(\theta)\sin(\theta) \frac{\partial w_1}{\partial x_1} + \cos(\theta)\sin(\theta) \frac{\partial x_2}{\partial x_2} - \sin^2(\theta) \frac{\partial w_1}{\partial x_2},\\
(\nabla \check{w})_{22} & \; = -\cos(\theta)\sin(\theta) \frac{\partial w_2}{\partial x_1} + \sin^2(\theta) \frac{\partial w_1}{\partial x_1} + \cos^2(\theta) \frac{\partial w_2}{\partial x_2} - \cos(\theta)\sin(\theta)\frac{\partial w_1}{\partial x_2}.
\end{align*}
Based on the Jacobian $\nabla \check{w}$ we can calculate the curl, the divergence and the two components of the shear for $\check{w}$:
\begin{align*}
\mycurl(\check{w}) &= (\nabla \check{w})_{21} - (\nabla \check{w})_{12}\\
&= (\cos^2(\theta) + \sin^2(\theta)) \left(\frac{\partial w_2}{\partial x_1} - \frac{\partial w_1}{\partial x_2}\right) = \mycurl(w),
\end{align*}
\begin{align*}
\mydiv(\check{w}) &= (\nabla \check{w})_{11} - (\nabla \check{w})_{22}\\
&= (\cos^2(\theta) + \sin^2(\theta)) \left(\frac{\partial w_1}{\partial x_1} + \frac{\partial w_2}{\partial x_2}\right) = \mydiv(w),
\end{align*}
\begin{align*}
\mysheara(\check{w}) &= (\nabla \check{w})_{22} - (\nabla \check{w})_{11}\\
&= (\cos^2(\theta) - \sin^2(\theta)) \left(\frac{\partial w_2}{\partial x_2} - \frac{\partial w_1}{\partial x_1}\right) -2\cos(\theta)\sin(\theta) \left(\frac{\partial w_1}{\partial x_2} + \frac{\partial w_2}{\partial x_1} \right)\\
&= (\cos^2(\theta) - \sin^2(\theta)) \mysheara(w) - 2\cos(\theta)\sin(\theta) \myshearb(w),
\end{align*}
\begin{align*}
\myshearb(\check{w}) &= (\nabla \check{w})_{12} + (\nabla \check{w})_{21}\\
&= (\cos^2(\theta) - \sin^2(\theta)) \left(\frac{\partial w_1}{\partial x_2} + \frac{\partial w_2}{\partial x_1}\right) -2\cos(\theta)\sin(\theta) \left(\frac{\partial w_2}{\partial x_2} + \frac{\partial w_1}{\partial x_1} \right)\\
&= (\cos^2(\theta) - \sin^2(\theta)) \myshearb(w) + 2\cos(\theta)\sin(\theta) \mysheara(w),
\end{align*}
Next, we consider $\vert \text{diag}({\bm \beta})\nabla_N \check{w}\vert$, where for the sake of readability, we define 
\begin{align*}
a := (\cos^2(\theta) - \sin^2(\theta)), \qquad b := \cos(\theta)\sin(\theta).
\end{align*}
Then we obtain:
\begin{align*}
\vert \text{diag}({\bm \beta})\nabla_N \check{w} \vert & \; = \beta_1 (\mycurl(\check{w}))^2 + \beta_2 (\mydiv(\check{w}))^2 + \beta_3 (\mysheara(\check{w}))^2 + \beta_4 (\myshearb(\check{w}))^2\\
& \; = \beta_1 (\mycurl(w))^2 + \beta_2 (\mydiv(w))^2\\
& \qquad + \beta_3 a^2 (\mysheara(w))^2 - \beta_3 ab \mysheara(w)\myshearb(w) + \beta_3 4 b^2 (\myshearb(w))^2\\
& \qquad + \beta_4 a^2 (\myshearb(w))^2 + \beta_4 ab\mysheara(w)\myshearb(w) + \beta_4 4 b^2 (\mysheara(w))^2
\end{align*}
We conclude the proof by setting $\beta_3 = \beta_4$ yielding the equivalence of $\vert \text{diag}({\bm \beta})\nabla_N \check{w}\vert$ and $\vert \text{diag}({\bm \beta})\nabla_N w\vert$, which then in turn implies $R_{\bm \beta}(\check{u})=R_{\bm \beta}(u)$.
\begin{align*}
\vert \text{diag}({\bm \beta})\nabla_N \check{w} \vert & \; = \beta_1 (\mycurl(w))^2 + \beta_2 (\mydiv(w))^2\\
& \qquad + \beta_3 a^2 (\mysheara(w))^2 + \beta_3 4 b^2 (\mysheara(w))^2\\
& \qquad + \beta_4 a^2 (\myshearb(w))^2 + \beta_4 4 b^2 (\myshearb(w))^2\\
& \; = \beta_1 (\mycurl(w))^2 + \beta_2 (\mydiv(w))^2\\
& \qquad + \beta_3 (\cos^2(\theta) + \sin^2(\theta))^2 (\mysheara(w))^2\\
& \qquad + \beta_4 (\cos^2(\theta) + \sin^2(\theta))^2 (\myshearb(w))^2\\
& \; = \beta_1 (\mycurl(w))^2 + \beta_2 (\mydiv(w))^2 + \beta_3 (\mysheara(w))^2 + \beta_4 (\myshearb(w))^2\\
& \; = \vert \text{diag}({\bm \beta})\nabla_N w \vert.
\end{align*}

\end{proof}

\section{Alternative visualisations of parts of Figures \ref{fig:SVF_compression}, \ref{fig:interpolation} and \ref{fig:discretisation}}
\label{app:alternative_visualisations}

\begin{figure}[h]
\captionsetup[subfigure]{labelformat=empty}
\centering
\subfloat[Fig.\ \ref{fig:SVF_compression}: $v$ in $x_1$-direction]{\includegraphics[width=.3\textwidth]{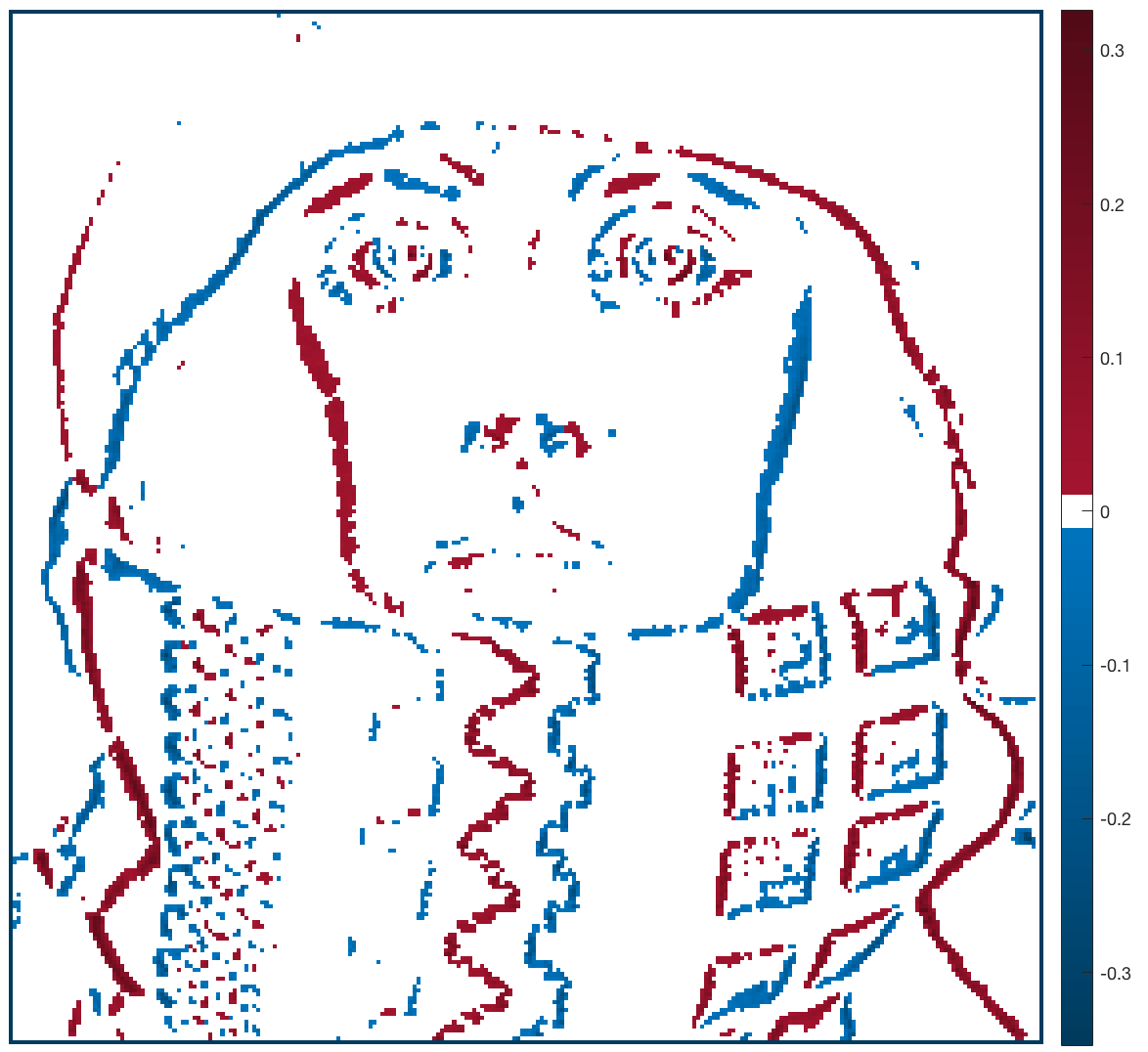}}\hfill
\subfloat[Fig.\ \ref{fig:SVF_compression}: $v$ in $x_2$-direction]{\includegraphics[width=.3\textwidth]{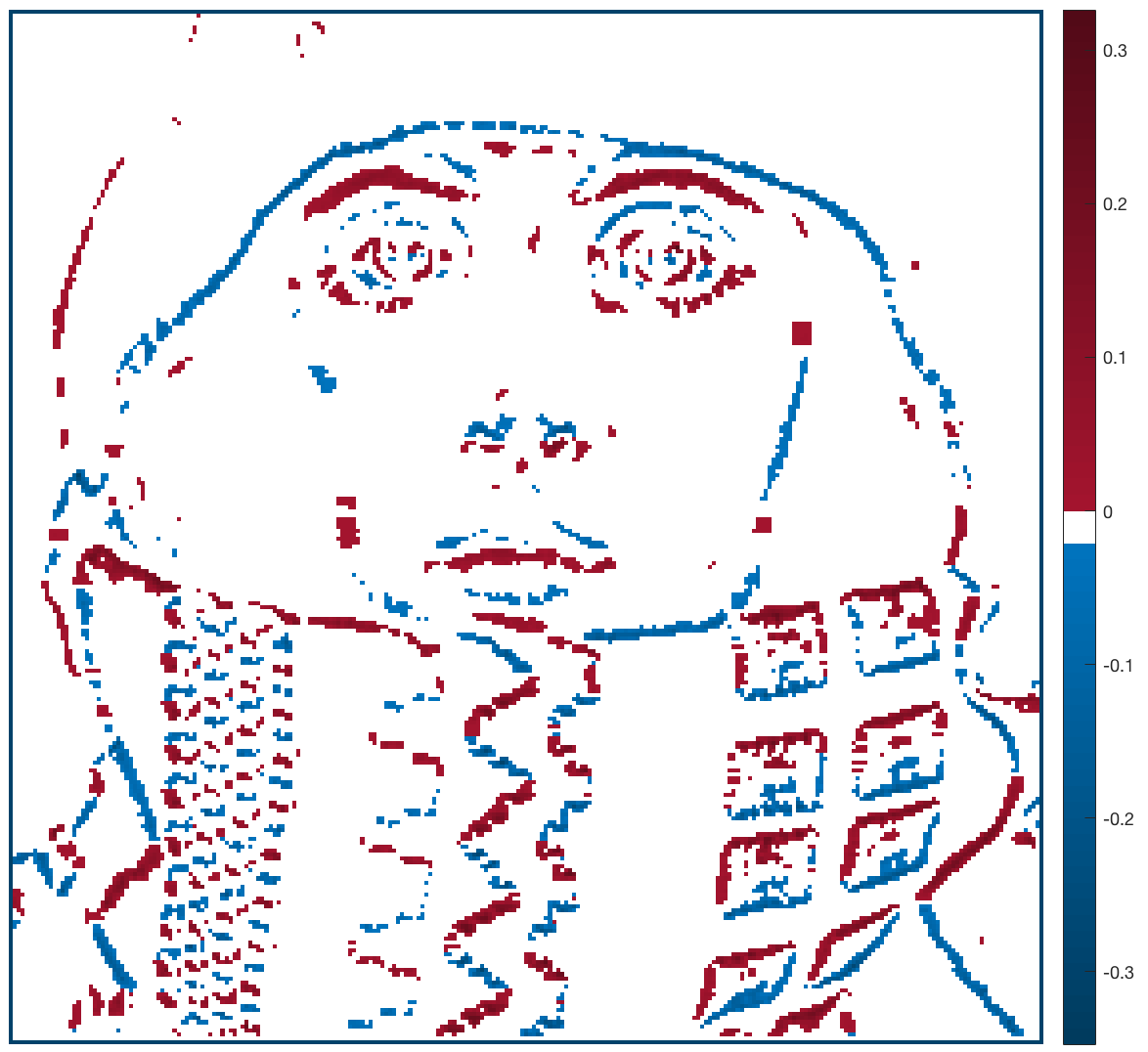}}\hfill
\subfloat[Fig.\ \ref{fig:SVF_compression}: $|v|$]{\includegraphics[width=.3\textwidth]{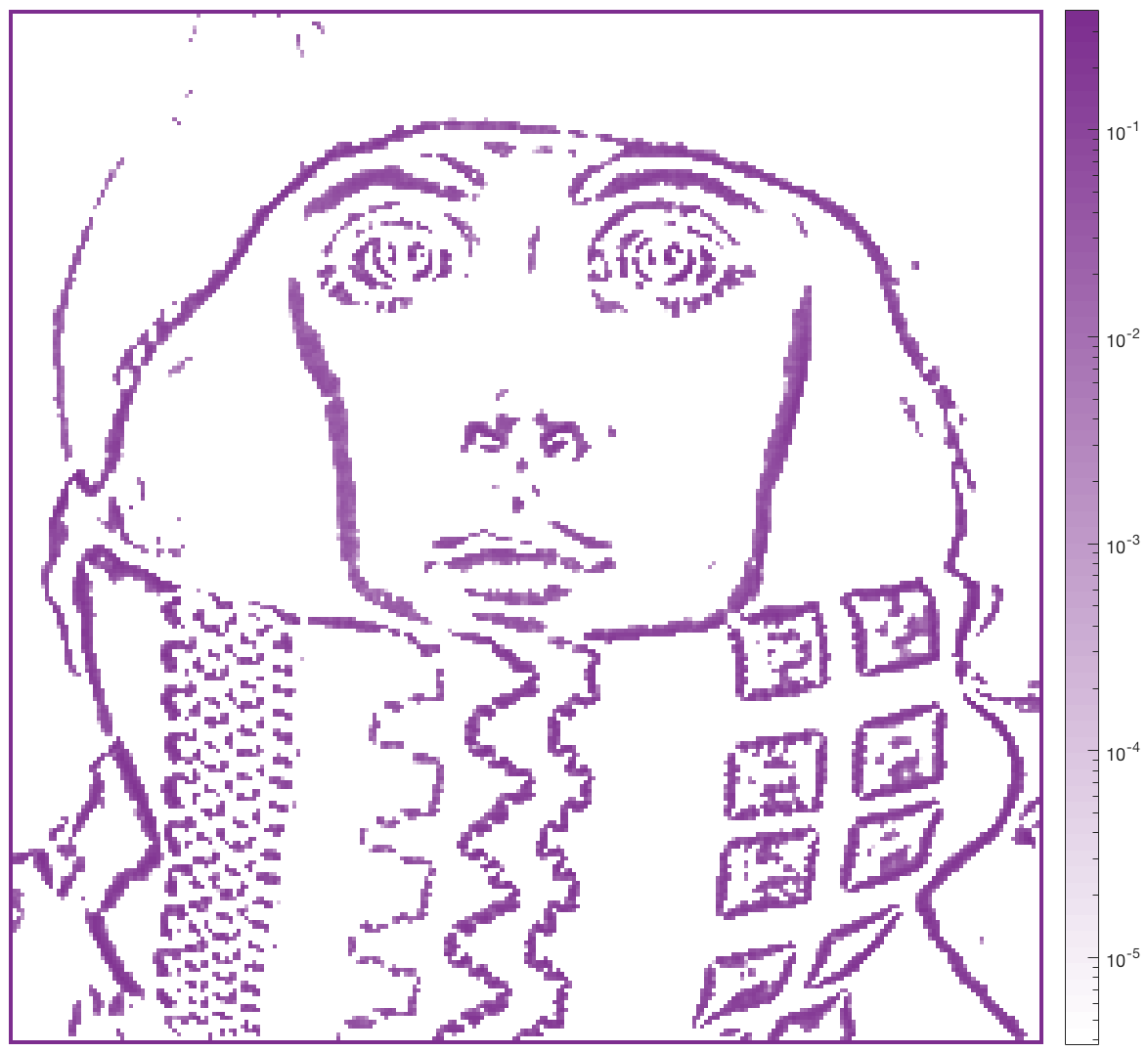}}
\end{figure}

\begin{figure}[h]
\captionsetup[subfigure]{labelformat=empty}
\centering
\subfloat[Fig.\ \ref{fig:interpolation}: TGV-type, $\mycurl$(w)]{\includegraphics[width=.3\textwidth]{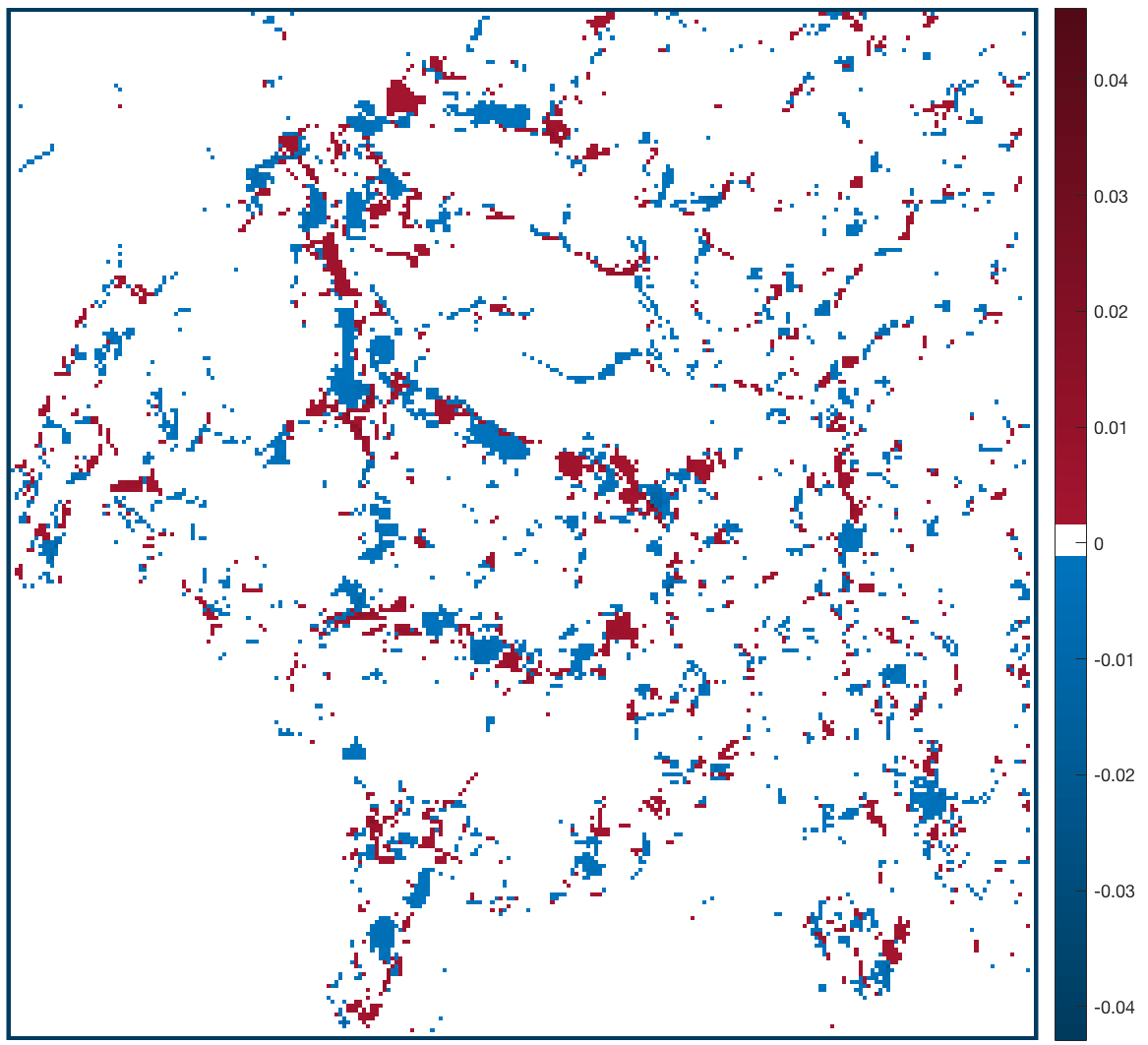}}\hfill
\subfloat[Fig.\ \ref{fig:interpolation}: interpolated, $\mycurl$(w)]{\includegraphics[width=.3\textwidth]{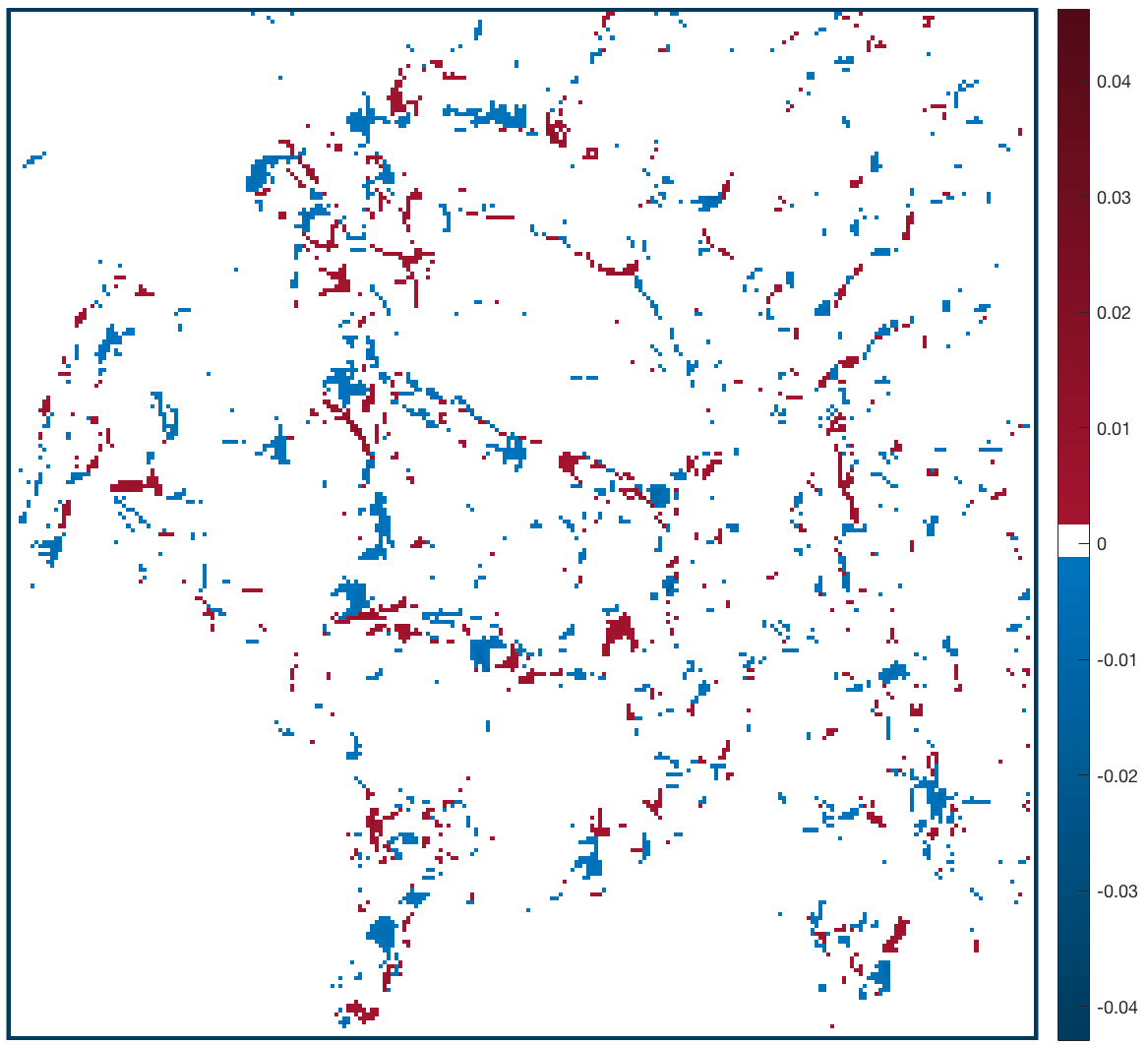}}\hfill
\subfloat[Fig.\ \ref{fig:interpolation}: ICTV-type, $\mycurl$(w)]{\includegraphics[width=.3\textwidth]{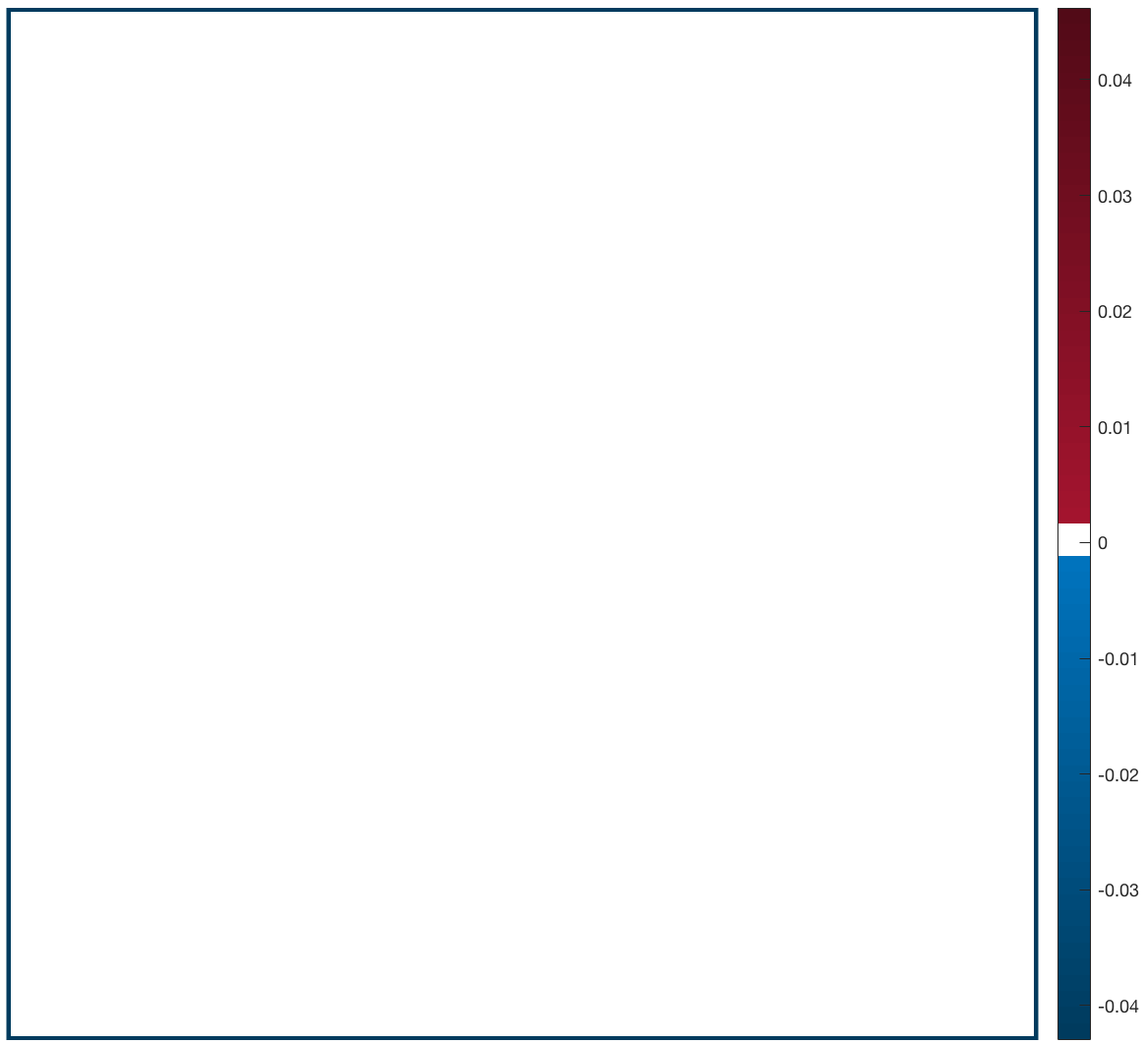}}\\
\subfloat[Fig.\ \ref{fig:interpolation}: TGV-type, $\mydiv$(w)]{\includegraphics[width=.3\textwidth]{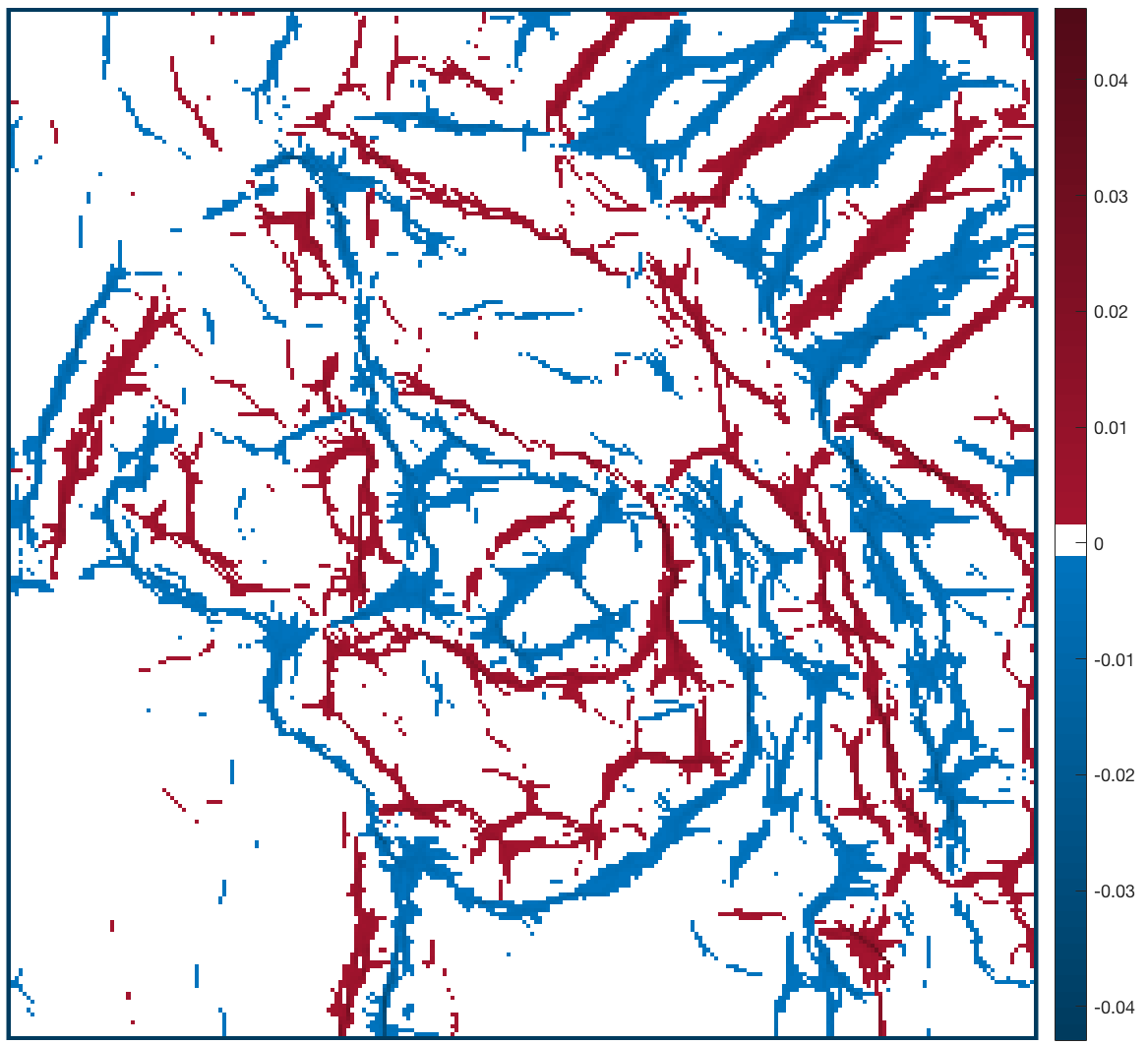}}\hfill
\subfloat[Fig.\ \ref{fig:interpolation}: interpolated, $\mydiv$(w)]{\includegraphics[width=.3\textwidth]{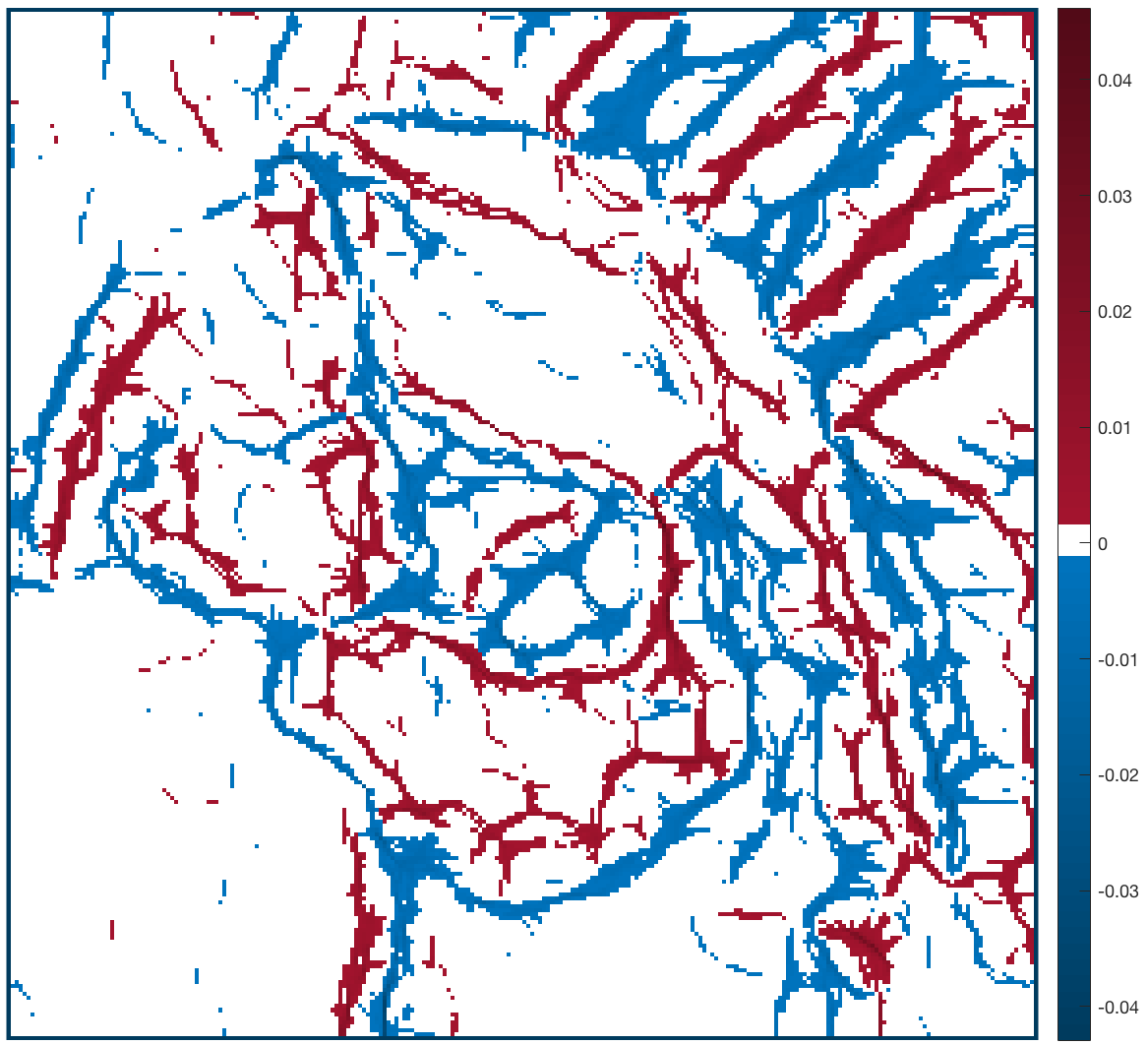}}\hfill
\subfloat[Fig.\ \ref{fig:interpolation}: ICTV-type, $\mydiv$(w)]{\includegraphics[width=.3\textwidth]{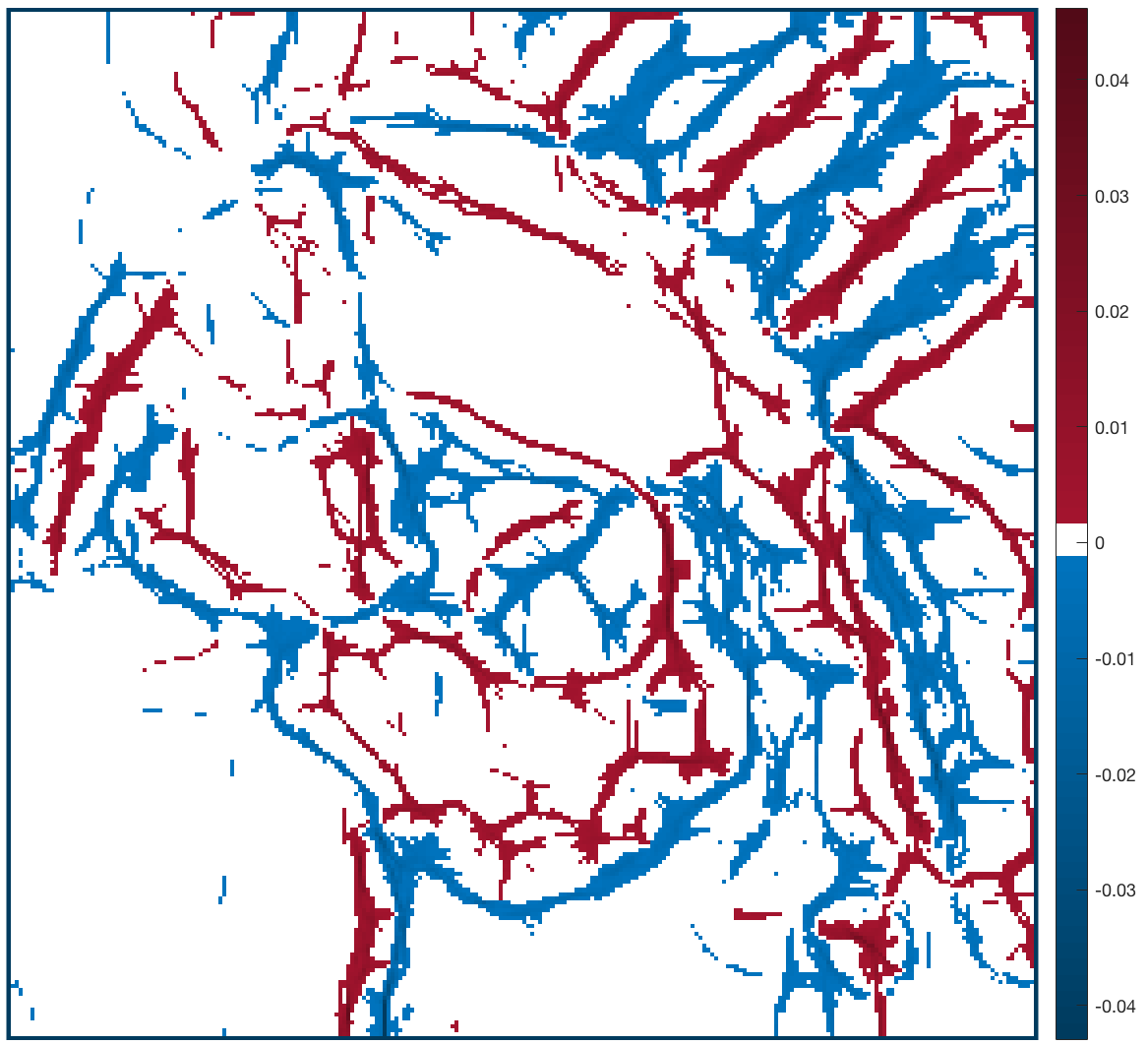}}\\
\subfloat[Fig.\ \ref{fig:interpolation}: TGV-type, $\mysheara$(w)]{\includegraphics[width=.3\textwidth]{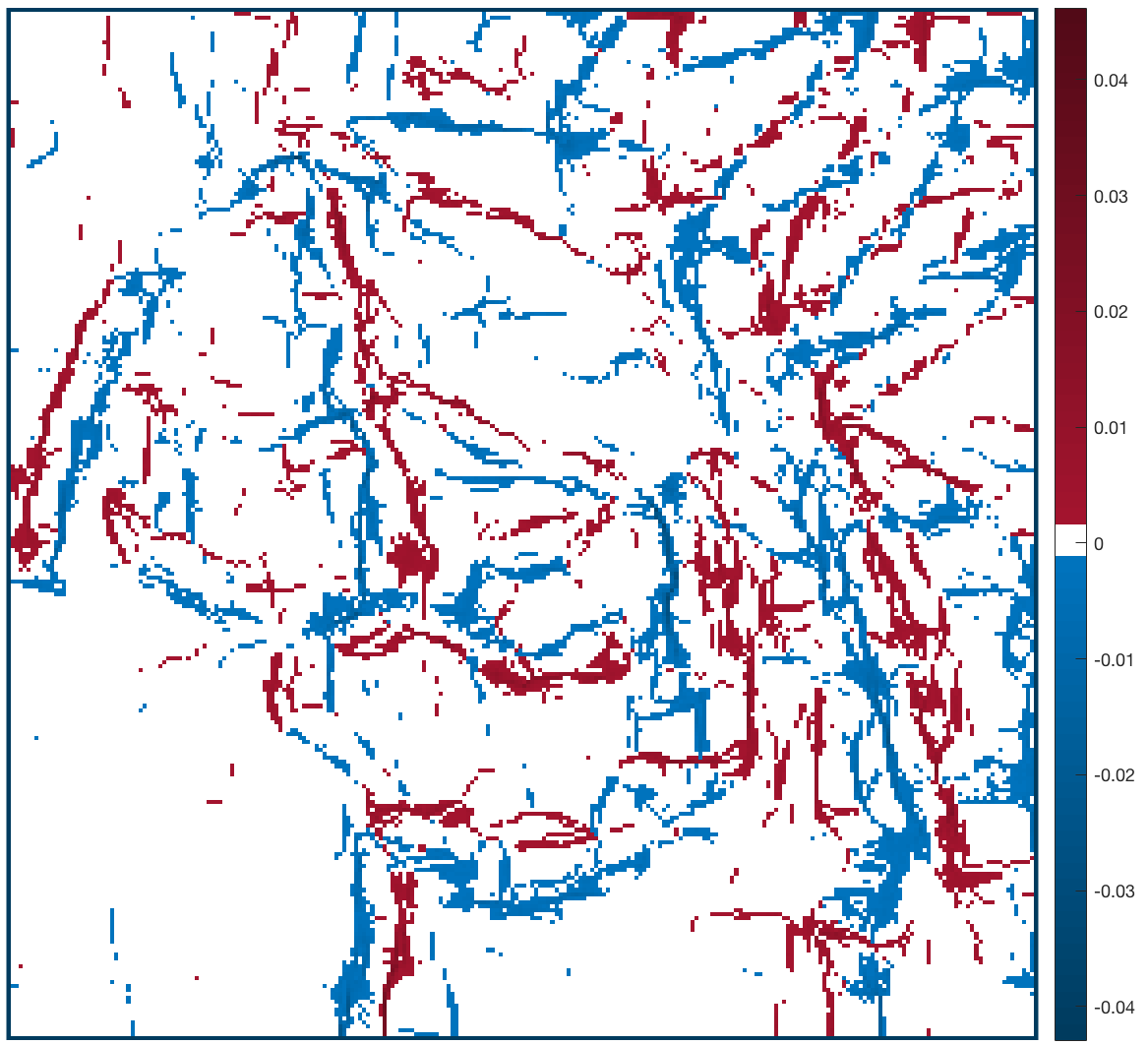}}\hfill
\subfloat[Fig.\ \ref{fig:interpolation}: interpolated, $\mysheara$(w)]{\includegraphics[width=.3\textwidth]{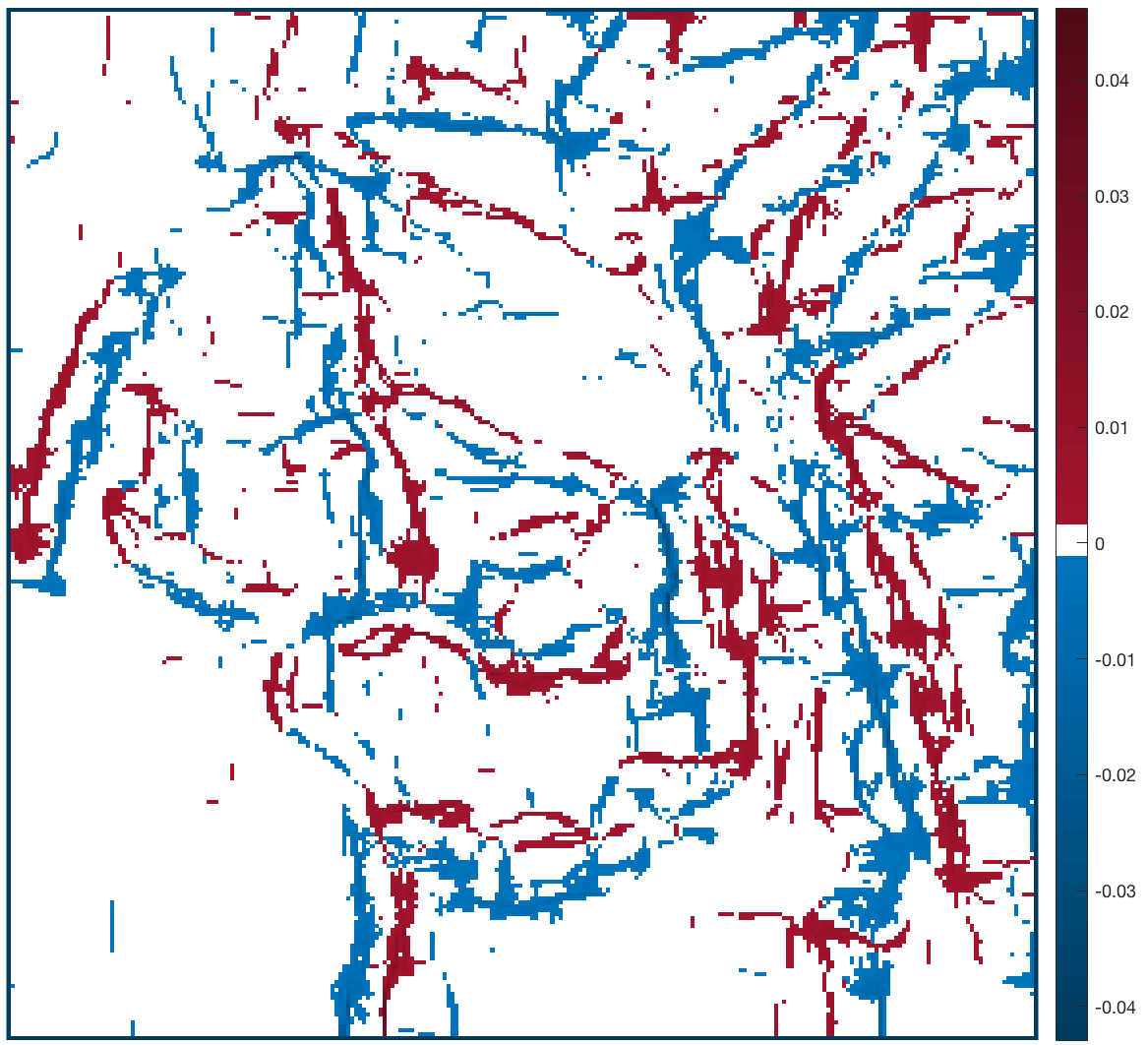}}\hfill
\subfloat[Fig.\ \ref{fig:interpolation}: ICTV-type, $\mysheara$(w)]{\includegraphics[width=.3\textwidth]{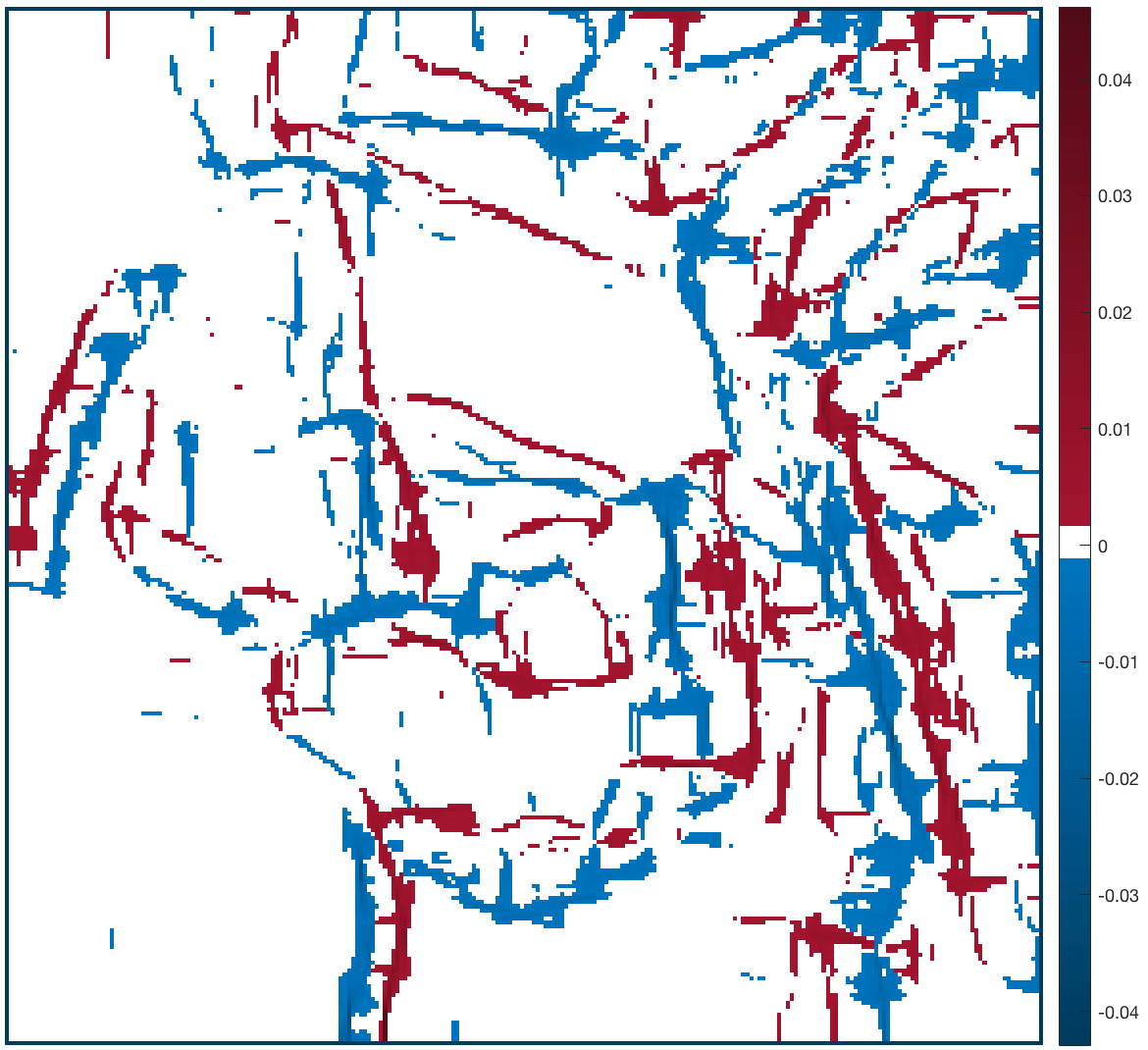}}\\
\subfloat[Fig.\ \ref{fig:interpolation}: TGV-type, $\myshearb$(w)]{\includegraphics[width=.3\textwidth]{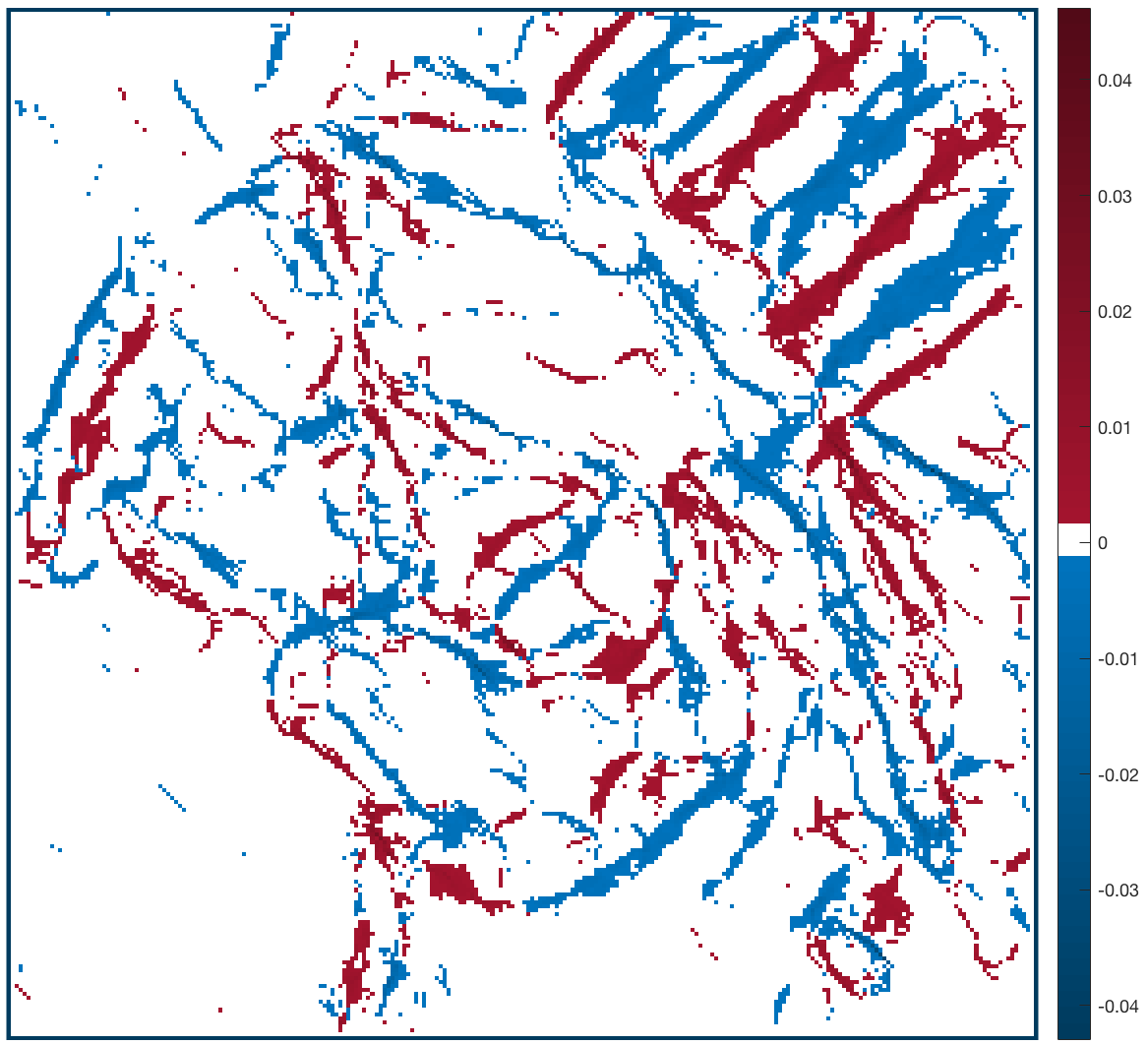}}\hfill
\subfloat[Fig.\ \ref{fig:interpolation}: interpolated, $\myshearb$(w)]{\includegraphics[width=.3\textwidth]{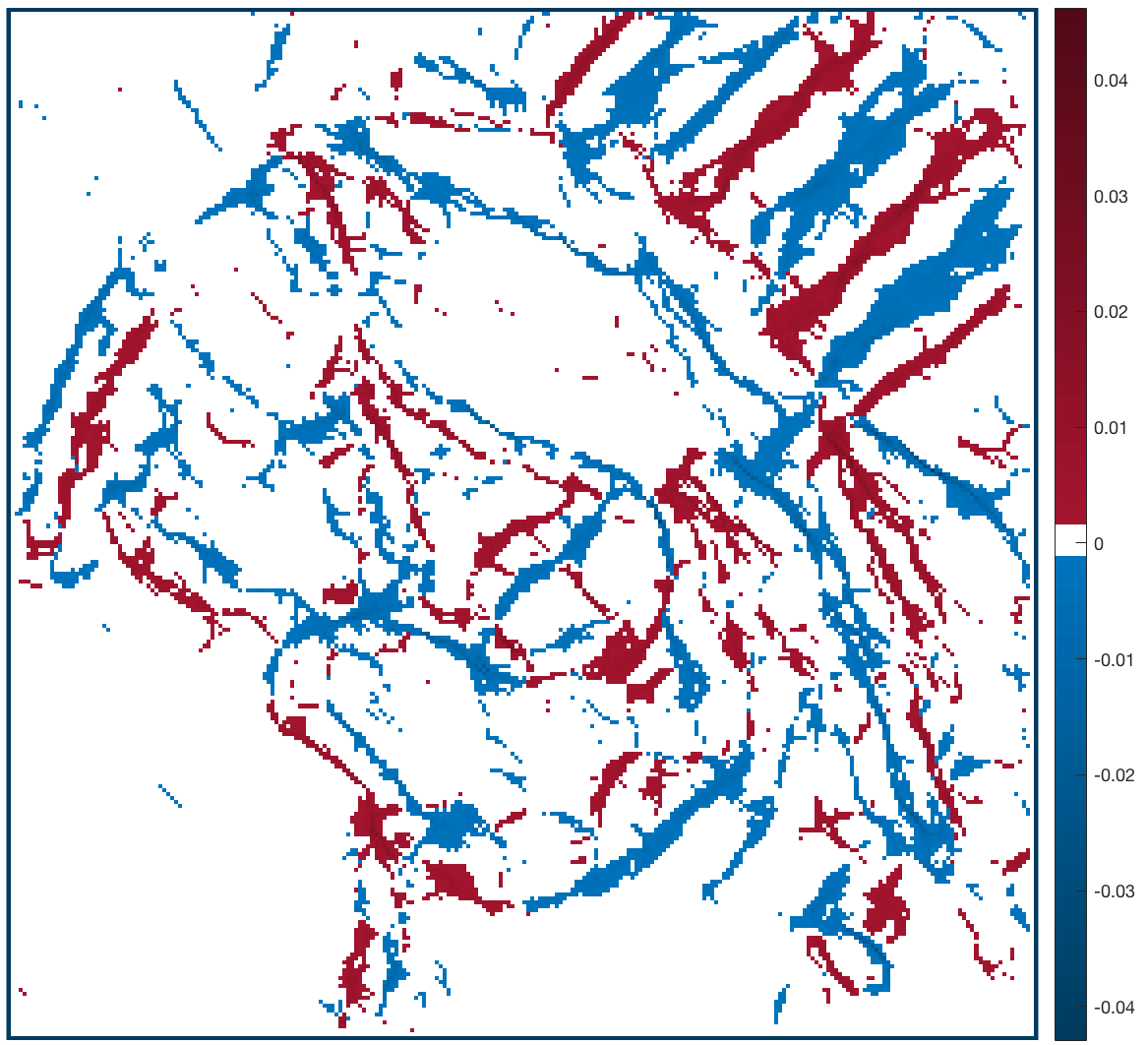}}\hfill
\subfloat[Fig.\ \ref{fig:interpolation}: ICTV-type, $\myshearb$(w)]{\includegraphics[width=.3\textwidth]{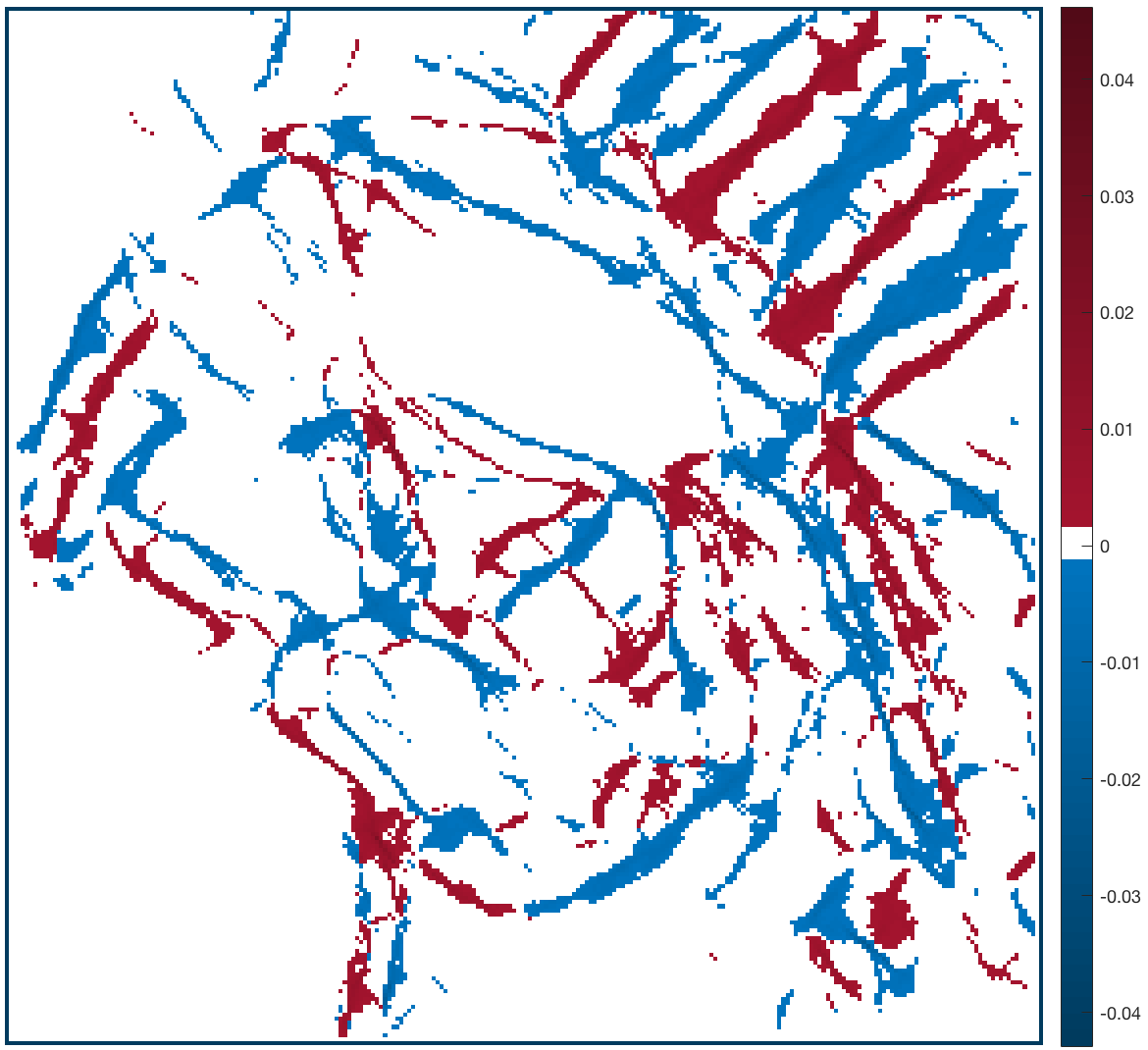}}
\end{figure}

\begin{figure}[h]
\captionsetup[subfigure]{labelformat=empty}
\centering
\subfloat[Fig.\ \ref{fig:discretisation}: difference image, $500 \times 500$ pixels]{\includegraphics[width=.45\textwidth]{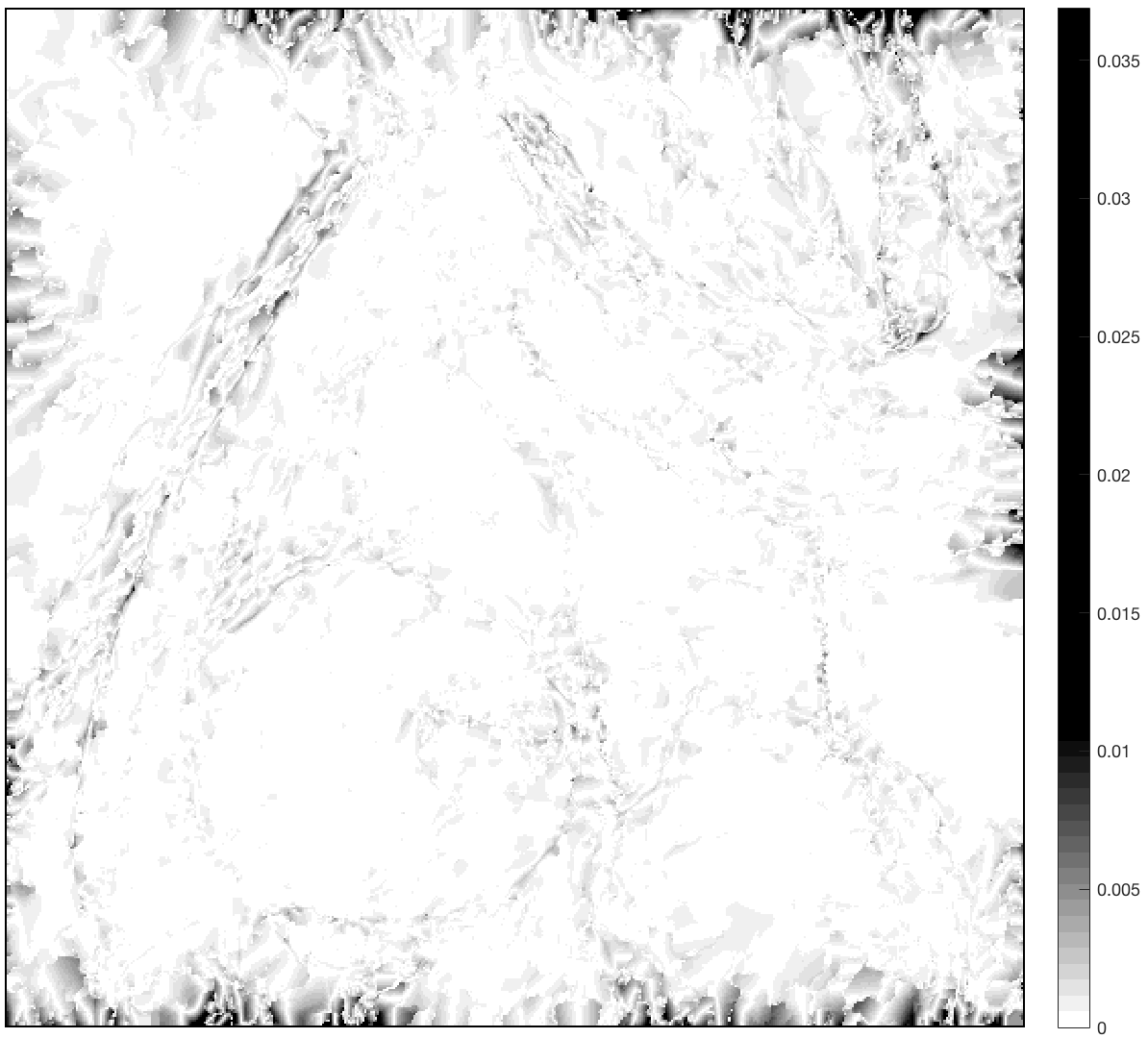}}\hfill
\subfloat[Fig.\ \ref{fig:discretisation}: difference image, $250 \times 250$ pixels]{\includegraphics[width=.45\textwidth]{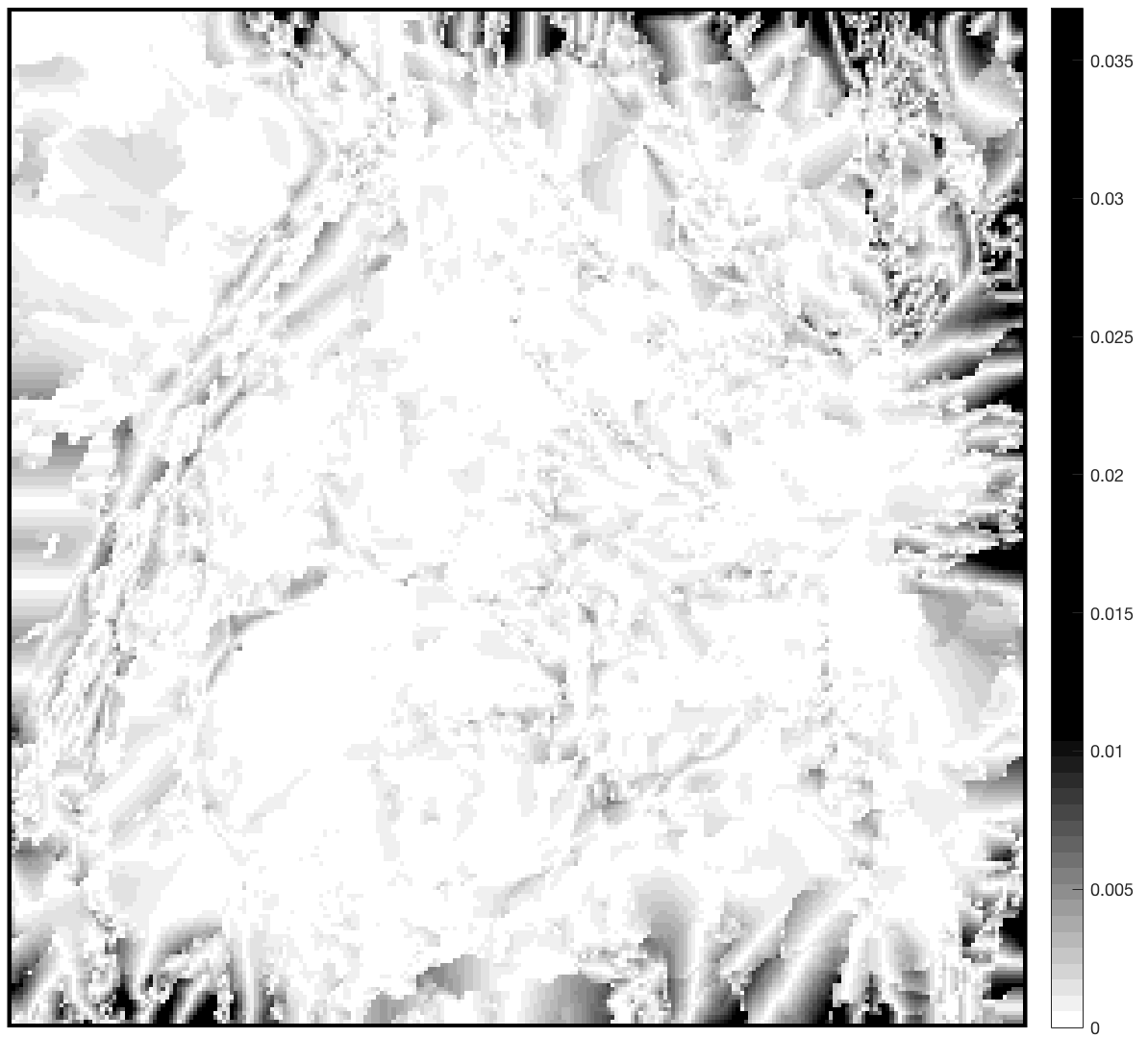}}
\end{figure}

\end{document}